\documentclass[11pt, a4paper,oneside, reqno]{amsart}
\usepackage[english]{babel}
\usepackage{amsmath, amsthm, amsfonts, mathrsfs, amssymb, amscd}
\usepackage{bm}
\usepackage{mathtools}
\mathtoolsset{centercolon}
\usepackage{mathabx}
\usepackage{accents}
\usepackage[toc]{appendix}
\usepackage{microtype}
\usepackage{comment}

\usepackage[shortlabels]{enumitem}
\usepackage[colorlinks, citecolor = blue, urlcolor={red}]{hyperref}
\usepackage[colorinlistoftodos,prependcaption,textsize=small,textwidth=25mm]{todonotes}
\usepackage{geometry}
\usepackage{tikz}
\usepackage{pgfplots}
\pgfplotsset{compat=newest}

\allowdisplaybreaks

\usepackage{color}

\newtheorem{theorem}{Theorem}
\newtheorem{corollary}[theorem]{Corollary}
\newtheorem{lemma}[theorem]{Lemma}
\newtheorem{proposition}[theorem]{Proposition}

\theoremstyle{remark} 
\newtheorem{remark}[theorem]{Remark}

\theoremstyle{definition} 
\newtheorem{definition}[theorem]{Definition}

\newtheorem{example}[theorem]{Example}

\numberwithin{theorem}{section}
\numberwithin{equation}{section}

\def\N{{\mathbb N}}
\def\R{{\mathbb R}}
\newcommand{\om}{\omega}
\renewcommand{\vec}[1]{\bm{#1}}
\newcommand{\Dom}{\mathcal{O}}
\newcommand{\BDom}{\partial\Dom}

\newcommand{\Dir}{{\rm Dir}}
\newcommand{\BC}{{\rm BC}}
\newcommand{\DD}{\Delta_{\Dir}}
\newcommand{\mc}{\mathcal}
\newcommand{\mrm}{\mathrm}
\newcommand{\Hinf}{H^{\infty}}
\newcommand{\RR}{\mathbb{R}}
\newcommand{\CC}{\mathbb{C}}
\newcommand{\NN}{\mathbb{N}}
\newcommand{\OO}{\mathcal{O}}
\newcommand{\WW}{\mathbb{W}}

\renewcommand{\SS}{\mathcal{S}}
\newcommand{\PP}{\mathcal{P}}
\newcommand{\II}{\mathcal{I}}

\newcommand{\half}{\frac{1}{2}}
\newcommand{\RRdh}{\RR^d_+}
\newcommand{\RRd}{\RR^d}
\newcommand{\Cc}{C_{\mathrm{c}}}

\renewcommand{\d}{\partial}
\newcommand{\del}{\Delta}
\newcommand{\grad}{\nabla}
\newcommand{\delDir}{\del_{\operatorname{Dir}}}
\newcommand{\delNeu}{\del_{\operatorname{Neu}}}
\newcommand{\odd}{\operatorname{odd}}

\newcommand{\dist}{\operatorname{dist}}

\newcommand{\eps}{\varepsilon}
\newcommand{\ph}{\varphi}

\newcommand{\gam}{\gamma}

\newcommand{\cDir}{{\rm c,Dir}}
\newcommand{\cNeu}{{\rm c,Neu}}
\newcommand{\Neu}{{\rm Neu}}
\newcommand{\Tr}{\operatorname{Tr}}
\newcommand{\loc}{{\rm loc}}
\renewcommand{\b}{{\rm b}}
\newcommand{\cl}{{\rm cl}}

\renewcommand{\tilde}[1]{\widetilde{#1}}
\newcommand{\cir}[1]{\accentset{\circ}{#1}}

\DeclarePairedDelimiter\abs{\lvert}{\rvert}

\DeclarePairedDelimiter\cbrace\{\}

\DeclarePairedDelimiter{\nrm}\lVert\rVert

\newcommand{\hab}[1]{\bigl(#1\bigr)}
\newcommand{\cbraceb}[1]{\bigl\{#1\bigr\}}

\newcommand{\bracb}[1]{\bigl[#1\bigr]}

\newcommand{\has}[1]{\Bigl(#1\Bigr)}
\newcommand{\cbraces}[1]{\Bigl\{#1\Bigr\}}

\newcommand{\dd}{\hspace{2pt}\mathrm{d}}

\newcommand{\ii}{{\rm i}}

\DeclareMathOperator{\ind}{\mathbf{1}}
\DeclareMathOperator{\UMD}{UMD}
\DeclareMathOperator{\BIP}{BIP}
\DeclareMathOperator{\id}{id}

\begin{document}

\title[Functional calculus for the Laplacian on rough domains]{Functional calculus on weighted Sobolev spaces for the Laplacian on rough domains}

\author[N. Lindemulder]{Nick Lindemulder}
\address[Nick Lindemulder]{}
\email{nick.lindemulder@gmail.com}

\author[E. Lorist]{Emiel Lorist}
\address[Emiel Lorist]{Delft Institute of Applied Mathematics\\
Delft University of Technology \\ P.O. Box 5031\\ 2600 GA Delft\\The
Netherlands} \email{e.lorist@tudelft.nl}

\author[F.B. Roodenburg]{Floris B. Roodenburg}
\address[Floris Roodenburg]{Delft Institute of Applied Mathematics\\
Delft University of Technology \\ P.O. Box 5031\\ 2600 GA Delft\\The
Netherlands} \email{f.b.roodenburg@tudelft.nl}

\author[M.C. Veraar]{Mark C. Veraar}
\address[Mark Veraar]{Delft Institute of Applied Mathematics\\
Delft University of Technology \\ P.O. Box 5031\\ 2600 GA Delft\\The
Netherlands} \email{m.c.veraar@tudelft.nl}

\makeatletter
\@namedef{subjclassname@2020}{%
  \textup{2020} Mathematics Subject Classification}
\makeatother

\subjclass[2020]{Primary: 47A60; Secondary: 35K20, 46E35}

\keywords{Functional calculus, Laplace operator, weights, maximal regularity, rough domains}

\thanks{The first author is supported by the grant OCENW.KLEIN.358 of the Dutch Research Council (NWO). The second author is partially supported by the Academy of Finland through grant no. 336323 and partially supported by the VENI grant VI.Veni.242.057 of the NWO. The third and fourth authors are supported by the VICI grant VI.C.212.027 of the NWO}

\begin{abstract}
We study the Laplace operator on domains subject to Dirichlet or Neumann boundary conditions. We show that these operators admit a bounded $\Hinf$-functional calculus on weighted Sobolev spaces, where the weights are powers of the distance to the boundary. 
Our analysis applies to bounded $C^{1,\lambda}$-domains with $\lambda\in[0,1]$, revealing a crucial trade-off: lower domain regularity can be compensated by enlarging the weight exponent. As a primary consequence, we establish maximal regularity for the corresponding heat equation. This extends the well-posedness theory for parabolic equations to domains with minimal smoothness, where classical methods are inapplicable. 
\end{abstract}

\maketitle

\setcounter{tocdepth}{1}
\tableofcontents

\section{Introduction}\label{sec:intro}
This paper contributes to the extensive study of the Laplace operator on domains with minimal boundary regularity (often referred to as rough domains), see, e.g., \cite{JK81, JK89, JK95, Ver84, Wo07} and the monographs \cite{Gr85, MS09} and references therein. In particular, we are interested in the $\Hinf$-functional calculus for the Laplacian on inhomogeneous weighted Sobolev spaces. 
The $\Hinf$-functional calculus provides a powerful framework for establishing well-posedness and regularity results for (possibly nonlinear) partial and stochastic partial differential equations ((S)PDEs).
Therefore, the $\Hinf$-calculus for sectorial operators is widely studied, see for instance \cite{DHP03, HNVW17, HNVW24, KW04} and the references therein. Applications to PDEs and SPDEs can, e.g., be found in \cite{DDHPV04, DK13, KKW06, KW17, PrSi, We06} and \cite{AV22, AV25, NVW12, NVW15b}, respectively.\\

Given a bounded $C^2$-domain $\OO\subseteq \R^d$, it is well known that the Laplacian with Dirichlet boundary conditions on $L^p(\OO)$ with $p\in(1,\infty)$ and domain $W^{2,p}(\OO)\cap W^{1,p}_0(\OO)$ generates an analytic $C_0$-semigroup, has the maximal regularity property and admits a bounded $\Hinf$-functional calculus. However, if the regularity of $\OO$ is too low (say Lipschitz or $C^1$), these properties fail and explicit counterexamples can be constructed, see \cite{Da79_Green, MS09}. In such counterexamples, the derivatives of the solutions to the resolvent equation
\begin{align*}
\lambda u-\Delta u &= f, \\
u|_{\partial \OO} &= 0,
\end{align*}
 can drastically blow up near the boundary $\d\OO$. As a consequence,
the canonical domain of the Dirichlet Laplacian on $L^p(\OO)$ is no longer a closed subspace of $W^{2,p}(\OO)$. 
Moreover, if one is interested in higher-order Sobolev regularity of the solution $u$, then more smoothness of $\OO$ is needed (see \cite{Ev10, KrBook08}), and 
additional boundary conditions for the data $f$ (compatibility conditions) need to be imposed (see \cite{DD11}). These additional boundary conditions for the data occur, in particular, in the study of mixed-order systems (see \cite{DS15}).

To set up a satisfying well-posedness and regularity theory for PDE \emph{without} such additional regularity or compatibility conditions, one can use a weighted function space for the solution $u$. 
In particular, one can consider spatial weights of the form $w_{\gam}^{\d\OO}(x):=\operatorname{dist}(x, \d\OO)^\gam$ for some suitable $\gam\in\RR$, which compensate the blow-up of the derivatives of the solution near $\d\OO$ and relax compatibility conditions. Partial differential equations on weighted spaces have already been studied extensively, see for instance \cite{DK15,DK18,DKZ16,Ki08, KL13,Kr99b,Kr01, Ne12} for deterministic equations and \cite{KK18,Ki04_divform,Ki04_varcoef, Kr94b, KL99} for stochastic equations.  \\

As stated, we are interested in the $\Hinf$-functional calculus for the Laplacian on inhomogeneous weighted Sobolev spaces of order $k \in \N_0$. This was studied in \cite{LLRV24, LV18} for the Dirichlet and Neumann Laplacian on the half-space $\R^d_+$. In the present paper, we extend the results to bounded domains $\OO$ with minimal smoothness, while ensuring that the canonical domain of the Laplacian is a closed subspace of a weighted Sobolev space of order $k+2$.

 Our main result for the Dirichlet Laplacian is as follows, see Theorems \ref{thm:Dirichlet_Laplacian} and \ref{thm:scalarDir}. For the definition of the involved spaces, the reader is referred to Section \ref{sec:weighted_Sob_spaces}.

\begin{theorem}[$\Hinf$-calculus for the Dirichlet Laplacian]\label{thm:introDir}
  Let $p\in(1,\infty)$, $k\in\NN_0$, $\lambda\in [0,1]$ and  $\gam\in (-1, 2p-1)\setminus\{p-1\}$. Furthermore, suppose that
\[ \lambda > 1- \tfrac{\gam+1}{p}\qquad \text{or, equivalently} \qquad \gam> (1-\lambda)p-1
\]
and  $\OO$ is a bounded $C^{1,\lambda}$-domain.
Then for all $\mu\geq 0$ the operator 
\begin{equation*}
    \mu-\delDir \quad \text{ on }\quad W^{k,p}(\OO, w_{\gam+kp}^{\d\OO}) \quad\text{ with }\quad D(\delDir)= W^{k+2,p}_\Dir(\OO, w_{\gam+kp}^{\d\OO})
\end{equation*}
has a bounded $\Hinf$-calculus of angle zero.
\end{theorem}

Theorem \ref{thm:introDir} generalises the result in \cite[Theorem 6.1]{LV18}, which is restricted to the case $k=0$ and to bounded $C^2$-domains. Theorem \ref{thm:introDir} allows for bounded $C^1$-domains if $\gam\in(p-1, 2p-1)$, while for $\gam\in(-1,p-1)$ we obtain that the smoothness of the domain may depend on the weight: if the power of the weight is larger, then a rougher domain is allowed. The smoothness parameter $\lambda$ is almost optimal. Indeed, solving the Dirichlet problem in the scale of weighted Sobolev spaces with a gain of two derivatives for the solution requires the boundary of the domain to have $W^{2-(\gam+1)/p,p}$-smoothness, see \cite[Theorem 15.6.1 applied to $\ell=2-(\gam+1)/p$]{MS09} and \cite[Section 14.6.1]{MS09} for an explicit counterexample with $C^1$-domains.
Furthermore, for $\gam=p-1$ the domain characterisation in Theorem \ref{thm:introDir} in terms of spaces with vanishing traces fails, see \cite[Remark 4.3]{LLRV24}, and for this reason we omit this case. \\

Concerning the Neumann Laplacian on bounded domains, we prove the following result, see Theorems \ref{thm:Neumann_Laplacian} and \ref{thm:scalarNeu}. 
\begin{theorem}[$\Hinf$-calculus for the Neumann Laplacian]\label{thm:introNeu}
  Let $p\in(1,\infty)$ and $\lambda\in (0,1]$. Furthermore, suppose that either
  \begin{enumerate}[(i)]
    \item $k\in\NN_0$,  $\gam\in (p-1, 2p-1)$, $\lambda> 2-\frac{\gam+1}{p}$ and $\OO$ is a bounded $C^{1,\lambda}$-domain, or, 
    \item $k\in\NN_1$, $\gam\in (-1, p-1)$, $\lambda> 1-\frac{\gam+1}{p}$ and $\OO$ is a bounded $C^{2,\lambda}$-domain.
  \end{enumerate}
Then for all $\mu>0$ the operator 
\begin{equation*}
    \mu-\delNeu \quad \text{ on }\quad W^{k,p}(\OO, w_{\gam+(k-1)p}^{\d\OO}) \quad\text{ with }\quad D(\delNeu)= W^{k+2,p}_\Neu(\OO, w_{\gam+(k-1)p}^{\d\OO})
\end{equation*}
has a bounded $\Hinf$-calculus of angle zero. Moreover, using function spaces modulo constants gives the result for all $\mu\geq 0$.
\end{theorem}

Note that, compared to Theorem \ref{thm:introDir}, the Sobolev spaces in Theorem \ref{thm:introNeu} have a smaller weight exponent, which is consistent with \cite[Theorem 1.2]{LLRV24}. Figure \ref{fig:1} visualises the parameters of the spaces in Theorem \ref{thm:introDir} and \ref{thm:introNeu} where we obtain a bounded $\Hinf$-calculus. Similar to the case of Dirichlet boundary conditions, we expect that the regularity of the domain in Theorem \ref{thm:introNeu} is almost optimal as well, see \cite[Section 15.6]{MS09} for some related results in this direction.

\begin{figure}[!ht]
  \centering
  \begin{tikzpicture}[x=0.75pt,y=0.75pt,yscale=1,xscale=1]
\begin{axis} [set layers=axis on top,
axis lines=middle,
xlabel near ticks,
ylabel near ticks,
xmin=0,
xmax=4.3,
ymin=-1,
ymax=15.2,
xtick={0,1,...,3,4},
ytick={-1,2,...,14},
yticklabels={$-1$,$p-1$, $2p-1$, $3p-1$, $4p-1$, $5p-1$},
grid style={line width=.1pt, draw=black!50},
major grid style={line width=.5pt,draw=black!50},
legend style={at={(0.055,0.97)},anchor=north west,fill=none},legend style={cells={align=left}},
enlargelimits=0.05,
disabledatascaling,
ylabel style={rotate=-90},
]

\addplot[fill=darkgray, fill opacity = 0.5, draw = none, area legend] coordinates {(0,2) (0,5) (4,17) (4,14)};

\addplot[fill=gray, fill opacity = 0.5, draw = none,area legend] coordinates {(0,-1) (0,2) (4,14) (4,11)};

\addplot[fill=lightgray, fill opacity = 0.5, draw = none,area legend] coordinates {(1,-1) (0,-1) (4,11) (4,8)};

\end{axis}
\node at (260,-15) [anchor=south, black]{smoothness $k$};
\node at (-45,130) [anchor=south, rotate=90, black]{weight exponent $\alpha$};
\node at (130,130) [anchor=south, rotate=33.5, black]{\scriptsize $\mu-\delDir$ on bounded $C^1$-domain};
\node at (130,95) [anchor=south, rotate=33.5, black]{\scriptsize $\mu-\delDir$ and $\mu-\delNeu$ on bounded $C^{1,\lambda}$-domain};
\node at (147,70) [anchor=south, rotate=33.5, black]{\scriptsize $\mu-\delNeu$ on bounded $C^{2,\lambda}$-domain};
\end{tikzpicture}
  \caption{The spaces $W^{k,p}(\OO, w^{\d\OO}_{\alpha})$ where $\mu-\delDir$ and $\mu-\delNeu$ as in Theorems \ref{thm:introDir} and \ref{thm:introNeu} (with $\alpha=\gam+kp$ and $\alpha=\gam+(k-1)p$, respectively) admit a bounded $\Hinf$-calculus.}\label{fig:1}
\end{figure}

The main novelties of our results are the following.
\begin{enumerate}[(i)]
  \item We prove the boundedness of the $\Hinf$-calculus, which is, in general, much harder to prove than maximal regularity and yields the boundedness of many singular integral operators \cite{KW01}. In particular, boundedness of the $\Hinf$-calculus implies (stochastic) maximal regularity \cite{HNVW24,NVW12}. Maximal regularity and higher-order regularity results for the heat equation with Dirichlet and Neumann boundary conditions are contained in Section \ref{subsec:MR_Riesz}. In particular, we recover some maximal regularity results for the Dirichlet Laplacian from \cite{KK04} (for bounded $C^1$-domains) and \cite{KN14} (for bounded $C^{1,\lambda}$-domains and $k=0$). For the latter case, our results with $k\geq 1$ are new.
  The Neumann Laplacian on the half-space is studied on weighted Sobolev spaces in \cite{DK18, DKZ16} (for $k=0$) and \cite{LLRV24}, but a systematic study on bounded domains seems to be unavailable until now.
  \item The smoothness of the domain $\OO$ in Theorems \ref{thm:introDir} and \ref{thm:introNeu} is \emph{independent} of the smoothness $k$ of the Sobolev space. The reason for this is that we do not use the standard localisation procedure from the half-space to domains (see, e.g., \cite{DHP03, Ev10, KrBook08}). This standard localisation procedure typically works for $C^{k+2}$-domains. Instead, we apply a more sophisticated  $C^1$-diffeomorphism suitable for the weighted setting. We discuss this in more detail below.  
\end{enumerate}

The key ingredient in the proofs of Theorems \ref{thm:introDir} and \ref{thm:introNeu} is the perturbation of the $\Hinf$-calculus on the half-space (obtained in \cite{LLRV24}) to \emph{special domains}, i.e. domains above the graph of a function with compact support. A common method is to relate the Laplacian on the half-space and on a special domain via a diffeomorphism. However, due to the low regularity of the domain, we cannot use the standard diffeomorphism as in, e.g., \cite{DHP03, Ev10, KrBook08, Wl87}. Instead, we construct a variant of the \emph{Dahlberg--Kenig--Stein pullback}, based on ideas in \cite{KK04, KimD07, Wl87}.
This diffeomorphism straightens the boundary, preserves the distance to the boundary and leaves the direction of the normal vector at the boundary invariant.Moreover, higher-order derivatives exist, but blow up near the boundary of the domain. This blow-up is compensated by the weights in our spaces. 

With estimates on this diffeomorphism at hand, we can employ perturbation theorems for the $\Hinf$-calculus to extend the results to special domains. Another difficulty arising in this perturbation argument is that, if the regularity of the domain is too low, then the perturbations are of the same order as the Laplacian. It is known that the $\Hinf$-calculus is not stable under small perturbations \cite{MY90}. Additionally, we need the perturbations to be well behaved with respect to a fractional power of the original operator. This requires the identification of certain complex interpolation spaces and fractional domains to perform the perturbation argument.
Finally, by another localisation argument, based on lower-order perturbations, the $\Hinf$-calculus on special domains is transferred to bounded domains. \\

We comment on some related and open problems. Theorems \ref{thm:introDir} and \ref{thm:introNeu} provide the bounded $\Hinf$-calculus on Sobolev spaces with integer smoothness, and with complex interpolation, the bounded $\Hinf$-calculus can also be obtained on spaces with fractional smoothness. However, an intrinsic characterisation of these complex interpolation spaces seems unavailable. Furthermore, we expect that our results can be extended to spaces with negative smoothness via duality. Some results for the weak (Dirichlet) Laplacian on weighted spaces are contained in \cite{BDS16, OS19}.

An interesting question regarding the smoothness of the domain is whether for $\gam\in(p-1,2p-1)$ the assumption of $C^1$-domains can be weakened to Lipschitz domains. In general, the analysis for Lipschitz domains becomes much more involved and different techniques are required than for $C^1$-domains, see for instance \cite{JK81, JK89, JK95, Wo07} and the references therein. We believe that our method should work for domains with a small Lipschitz character. The $\Hinf$-calculus on Lipschitz domains could be important for studying SPDEs in the weighted setting, see \cite{Ki08_Lipschitz, Ki09, Ki14,  Ki09_erratum}, where the range of weights is significantly smaller than $\gam\in(p-1,2p-1)$.

\subsection*{Outline}  The outline of this paper is as follows. In Section \ref{sec:prelim} we introduce some preliminary concepts and results needed throughout the paper. In Section \ref{sec:weighted_Sob_spaces} we study weighted Sobolev spaces on domains and prove characterisations for these spaces. In Section \ref{sec:frac_domain} we prove results on the fractional domains of the Laplacian on the half-space, which are required for perturbation of the $\Hinf$-calculus. In Section \ref{sec:calc_spec_dom} we perturb the $\Hinf$-calculus from the half-space to special domains, and in Section \ref{sec:calc_dom} we perform a localisation procedure to obtain the $\Hinf$-calculus on bounded domains. Moreover, as a consequence, we obtain maximal regularity for the heat equation and boundedness of Riesz transforms. Finally, in Appendix \ref{sec:appendix_lemma} we provide  localisation techniques on rough domains.

\section{Preliminaries}\label{sec:prelim}
\subsection{Notation} We denote by $\NN_0$ and $\NN_1$ the set of natural numbers starting at $0$ and $1$, respectively. 
For $a\in\RR$, we use the notation $(a)_+=a$ if $a\geq 0$ and $(a)_+=0$ otherwise.

For $d\in\NN_1$, the half-space is given by $\RRdh=\RR_+\times\RR^{d-1}$, where $\RR_+=(0,\infty)$ and for $x\in \RRdh$ we write $x=(x_1,\tilde{x})$ with $x_1\in \RR_+$ and $\tilde{x}\in \RR^{d-1}$. For $\gam\in\RR$, $\OO\subseteq \RR^d$ open and $x\in\OO$ we define the power weight $w_{\gam}^{\d\OO}(x):=\mrm{dist}(x,\d\OO)^\gam$.

For two topological vector spaces $X$ and $Y$, the space of continuous linear operators is $\mc{L}(X,Y)$ and $\mc{L}(X):=\mc{L}(X,X)$. Unless specified otherwise, $X$ will always denote a Banach space with norm $\|\cdot\|_X$ and the dual space is $X':=\mc{L}(X,\CC)$.

For a linear operator $A:X\supseteq D(A)\to X$ on a Banach space $X$ we denote by $\sigma(A)$ and $\rho(A)$  the spectrum and resolvent set, respectively. For $\lambda\in\rho(A)$, the resolvent operator is given by $R(\lambda,A)=(\lambda-A)^{-1}\in \mc{L}(X)$.

We write $f\lesssim g$ (resp. $f\gtrsim g$) if there exists a constant $C>0$, possibly depending on parameters which will be clear from the context or will be specified in the text, such that $f\leq Cg$ (resp. $f\geq Cg$). Furthermore, $f\eqsim g$ means $f\lesssim g$ and $g\lesssim  f$.\\

For an open and non-empty $\OO\subseteq \RR^d$ and $\ell\in\NN_0\cup\{\infty\}$, the space $C^\ell(\OO;X)$ denotes the space of $\ell$-times continuously differentiable functions from $\OO$ to some Banach space $X$. In the case $\ell=0$ we write $C(\OO;X)$ for $C^0(\OO;X)$. Furthermore, we write $C^\ell_\b(\OO;X)$ for the space of all functions $f\in C^\ell(\OO;X)$ such that $\d^{\alpha} f$ is bounded on $\OO$ for all multi-indices $\alpha\in \NN_0^d$ with $|\alpha|\leq \ell$. 

Let $\Cc^{\infty}(\OO;X)$ be the space of compactly supported smooth functions on $\OO$ equipped with its usual inductive limit topology. The space of $X$-valued distributions is given by $\mc{D}'(\OO;X):=\mc{L}(\Cc^{\infty}(\OO);X)$. Moreover, $\Cc^{\infty}(\overline{\OO};X)$ is the space of smooth functions with their support in a compact set contained in $\overline{\OO}$.

 We denote the Schwartz space by $\SS(\RRd;X)$ and $\SS'(\RRd;X):=\mc{L}(\SS(\RRd);X)$ is the space of $X$-valued tempered distributions. For $\OO\subseteq \RR^d$ we define $\SS(\OO;X):=\{u|_{\OO}: u\in\SS(\RRd;X)\}$.

Finally, for $\theta\in(0,1)$ and a compatible couple $(X,Y)$ of Banach spaces, the complex interpolation space is denoted by $[X,Y]_{\theta}$.

\subsection{Holomorphic functional calculus}\label{sec:prelim_func_calc} In this section, we collect the required preliminaries on sectorial operators with a bounded $\Hinf$-calculus.

\subsubsection{Definitions} 
For $\om\in(0,\pi)$, let $\Sigma_{\om}:=\{z\in\CC\setminus\{0\}:|\operatorname{arg}(z)|<\om\}$ be a sector in the complex plane.
\begin{definition}\label{def:sect}
  An injective, closed linear operator $(A,D(A))$ with dense domain and dense range on a Banach space $X$ is called \emph{sectorial} if there exists a $\om\in (0,\pi)$ such that $\sigma(A)\subseteq\overline{\Sigma_{\om}}$ and
  \begin{equation*}
    \sup_{\lambda\in \CC\setminus \overline{\Sigma_{\om}}}\|\lambda R(\lambda,A)\|<\infty.
  \end{equation*}
  Furthermore, the angle of sectoriality $\om(A)$ is defined as the infimum over all possible $\om>0$.
\end{definition}

To continue, we introduce the following Hardy spaces. Let $\om\in(0,\pi)$, then $H^1(\Sigma_{\om})$ is the space of all holomorphic functions $f:\Sigma_{\om}\to \CC$ such that
\begin{equation*}
  \|f\|_{H^1(\Sigma_{\om})}:=\sup_{|\nu|<\om}\|t\mapsto f(e^{\ii\nu}t)\|_{L^1(\RR_+,\frac{\mathrm{d}t}{t})}<\infty.
\end{equation*}
Moreover, let $H^{\infty}(\Sigma_{\om})$ be the space of all bounded holomorphic functions on the sector with norm
\begin{equation*}
  \|f\|_{H^{\infty}(\Sigma_{\om})}:=\sup_{z\in\Sigma_{\om}}|f(z)|.
\end{equation*}

\begin{definition}
  Let $A$ be a sectorial operator on a Banach space $X$ and let $\om\in(\om(A),\pi)$, $\nu\in(\om(A), \om)$ and $f\in H^{1}(\Sigma_{\om})$. We define the operator
  \begin{equation*}
    f(A):=\frac{1}{2\pi \ii}\int_{\partial \Sigma_{\nu}}f(z)R(z,A)\dd z,
  \end{equation*}
  where $\d \Sigma_{\nu}$ is oriented counterclockwise. The operator $A$ has a \emph{bounded $\Hinf(\Sigma_{\om})$-calculus} if there exists a $C>0$ such that
  \begin{equation*}
    \|f(A)\|\leq C\|f\|_{\Hinf(\Sigma_{\om})}\quad \text{ for all }f\in H^1(\Sigma_{\om})\cap\Hinf(\Sigma_{\om}).
  \end{equation*}
  Furthermore, the angle of the $\Hinf$-calculus $\om_{\Hinf}(A)$ is defined as the infimum over all possible $\om>\om(A)$.
\end{definition}
For more details on the $\Hinf$-calculus, the reader is referred to \cite{Ha06} and \cite[Chapter 10]{HNVW17}.

\subsubsection{Fractional domains} Let $A$ be a sectorial operator and let $\alpha\in\CC$. To define fractional powers $A^{\alpha}$, we need a functional calculus allowing for holomorphic functions of polynomial growth. This is known as the \emph{extended functional calculus} and the reader is referred to \cite[Chapter 15]{HNVW24} or \cite[Appendix 15.C]{KW04} for a detailed study of extended functional calculi and fractional powers. In particular, $A^\alpha$ is again sectorial.

A sectorial operator $A$ on a Banach space $X$ has \emph{bounded imaginary powers} ($\BIP$) if $A^{\ii s}$  extends to a bounded operator on $X$ for every $s\in\RR$. The angle is given by $\om_{\BIP}(A):=\inf\{\om\in\RR:\sup_{s\in\RR}e^{-\om|s|}\|A^{\ii s}\|<\infty\}$.
Moreover, a bounded $\Hinf$-calculus implies $\BIP$ and $\om_{\BIP}(A)\leq \om_{\Hinf}(A)$, see \cite[Section 15.3]{HNVW24}.

We recall a result on the interpolation of fractional domains. For details on interpolation theory, the reader is referred to \cite{BL76} and \cite{Tr78}.
\begin{proposition}[{\cite[Corollary 15.3.10]{HNVW24}}]\label{prop:frac_domain}
  Let $A$ be a sectorial operator on a Banach space $X$ and assume that $A$ has $\BIP$. Then for all $\theta\in(0,1)$ and $0\leq \alpha<\beta$ we have
  \begin{equation*}
    D(A^{(1-\theta)\alpha+\theta\beta})=[D(A^{\alpha}), D(A^{\beta})]_{\theta}.
  \end{equation*}
\end{proposition}
Moreover, by \cite[Proposition 15.2.12]{HNVW24} we have for a sectorial operator $A$ that $D((\mu+A)^{\alpha})=D(A^{\alpha})$ for all $\mu\geq 0$ and $\alpha> 0$.

\subsubsection{Perturbation of the \texorpdfstring{$\Hinf$-}{holomorphic functional }calculus}\label{subsec:calc_pert} We collect some known perturbation results for the $\Hinf$-calculus. For further perturbation results for the $H^\infty$-calculus, the reader is referred to \cite{HNVW24, KKW06, KLW23, KW04}. We start with a result for shifting the $\Hinf$-calculus.
 \begin{proposition}[{\cite[Proposition 16.2.6]{HNVW24}}]\label{prop:calc_pert_Id}
  Let $A$ be a sectorial operator on a Banach space $X$ and let $\om \in (\om(A),\pi)$.
  \begin{enumerate}[(i)]
    \item If $A $ has a bounded $\Hinf(\Sigma_{\om})$-calculus, then $\mu+A$ has a bounded $\Hinf(\Sigma_{\om})$-calculus for all $\mu> 0$. Moreover, the constant in the estimate for the $\Hinf$-calculus can be taken independent of $\mu$.
    \item\label{it:prop:calc_pert_Id2} If $\mu_0 +A$  has a bounded $\Hinf(\Sigma_{\om})$-calculus for some $\mu_0> 0$, then $\mu+A$ has a bounded $\Hinf(\Sigma_{\om})$-calculus for all $\mu> 0$.
  \end{enumerate}
\end{proposition}
In the case of a lower-order perturbation, we have the following result.
\begin{theorem}[{\cite[Theorem 16.2.7]{HNVW24}}]\label{thm:lowe_ord_pert}
  Let $A$ be a sectorial operator on a Banach space $X$. Let $\om\in (\om(A), \pi)$ and assume that $A$ has a bounded $\Hinf(\Sigma_{\om})$-calculus. Let $\alpha\in (0,1)$ and assume that $B$ is a linear operator on $X$ such that $D(B)\supseteq D(A^{\alpha})$ and
  \begin{equation}\label{eq:lower_ord}
    \|Bu\|_X\leq C\|A^{\alpha}u\|_X,\qquad u\in D(A),
  \end{equation}
  for some $C>0$. Then there exists a $\mu\geq 0$ such that $\mu+A+B$ with $D(\mu+A+B)=D(A)$ has a bounded $\Hinf(\Sigma_{\om})$-calculus.
\end{theorem}
To extend the $\Hinf$-calculus of the Laplacian on $\RRdh$ to domains in Sections \ref{sec:calc_spec_dom} and \ref{sec:calc_dom}, we need to deal with perturbations that are not of lower order. Unfortunately, the $H^\infty$-calculus is not stable under small perturbations, as shown in a counterexample by McIntosh and Yagi \cite{MY90}. Instead, for the $H^\infty$-calculus, one has statements of the following type, in which the perturbation is in addition required to be well behaved with respect to a fractional power of the original operator.

\begin{theorem}[{\cite[Theorem 16.2.8]{HNVW24}}]\label{thm:perturbcalculus}
Let $A$ be a sectorial operator on a Banach space $X$ such that $0\in\rho(A)$. Let $\om\in(\om(A),\pi)$ and assume that $A$ has a bounded $H^\infty(\Sigma_{\om})$-calculus. Let $B$ be a linear operator on $X$ such that $ D(B)\supseteq D(A)$. Suppose that there is an $\eta>0$ such that
\begin{enumerate}[(i)]
  \item\label{it:thm:perturbcalculus1} $\nrm{Bu}_X \leq \eta \,\nrm{Au}_X, \qquad u \in D(A).$
\end{enumerate}
Moreover, suppose that at least one of the following relative bounds is satisfied:
\begin{enumerate}[resume*]
  \item\label{it:thm:perturbcalculus2} there exists an $\alpha \in (0,1)$ such that $B(D(A^{1+\alpha}))\subseteq D(A^\alpha)$ and
\begin{equation*}
\nrm{A^\alpha Bu}_X \leq C \nrm{A^{1+\alpha}u}_X, \qquad u\in D(A^{1+\alpha}),
\end{equation*}
  \item\label{it:thm:perturbcalculus3} there exists an $\alpha \in (0,1)$ such that
  \begin{equation*}
\nrm{A^{-\alpha} Bu}_X \leq C \nrm{A^{1-\alpha}u}_X, \qquad u\in D(A^{1-\alpha}),
\end{equation*}
\end{enumerate}
for some $C>0$. Then there exists an $\tilde{\eta}>0$ such that, if \ref{it:thm:perturbcalculus1} holds with $\eta<\tilde{\eta}$, then $A+B$ with $D(A+B)=D(A)$ has a bounded $H^\infty(\Sigma_\om)$-calculus.
\end{theorem}
\begin{remark} Theorem \ref{thm:perturbcalculus} is taken from \cite[Theorem 16.2.8]{HNVW24}, where it should be noted that their condition of $R$-sectoriality on $B$ is redundant, see also \cite{KW01b} and the errata to \cite{HNVW24}. 
A version of Theorem \ref{thm:perturbcalculus} for positive fractional powers also appeared in \cite[Theorem 3.2]{DDHPV04}.
\end{remark}

\subsection{The \texorpdfstring{$\UMD$}{UMD} property}
Throughout this paper, we work mostly with vector-valued Sobolev spaces (although our results are also new for the scalar-valued case), and for this, we need the $\UMD$ property for Banach spaces. We recall that a Banach space $X$ satisfies the condition $\UMD$ (unconditional martingale differences) if and only if the Hilbert transform extends to a bounded operator on $L^p(\RR;X)$. We list the following relevant properties of $\UMD$ spaces, see for instance \cite[Chapter 4 \& 5]{HNVW16}.
\begin{enumerate}[(i)]
\item Hilbert spaces are $\UMD$ Banach spaces. In particular, $\CC$ is a $\UMD$ space.
  \item If $p\in(1,\infty)$, $(S,\Sigma,\mu)$ is a $\sigma$-finite measure space and $X$ is a $\UMD$ Banach space, then $L^p(S;X)$ is a $\UMD$ Banach space.
  \item $\UMD$ Banach spaces are reflexive.
\end{enumerate}
The UMD property is known to be necessary for many results on vector-valued Sobolev spaces (see \cite{ABK}, \cite[Section 5.6]{HNVW16} and \cite[Corollary 13.3.9]{HNVW24}). Moreover, the boundedness of the $H^\infty$-calculus of $-\Delta$ on spaces such as $L^p(\R^d;X)$ also is equivalent to the UMD property (see \cite[Section 10.5]{HNVW17}).

\subsection{Domains}\label{subsec:Ckdomains}
Let $\lambda\in (0,1]$ and let $\OO\subseteq \RR^{d-1}$ be open. A function $h:\OO\to \RR$ is called \emph{uniformly $\lambda$-H\"older continuous on $\OO$} if 
\begin{equation*}
  [h]_{\lambda, \OO}:= \sup_{\substack{x, y\in \OO\\ x\neq y}}\frac{|h(x)-h(y)|}{|x-y|^\lambda}<\infty.
\end{equation*}
In addition, for $\ell\in \NN_0$ we define the space of $\lambda$-H\"older continuous functions by 
\begin{equation*}
  C_{\b}^{\ell,\lambda}(\OO):=\{f\in C_{\b}^{\ell}(\OO): [\d^\alpha h]_{\lambda, \OO}<\infty\text{ for all }|\alpha| = \ell\}.
\end{equation*}
For $\lambda=0$ we write $C_\b^{\ell,0}(\OO):=C^{\ell}_\b(\OO)$. By $\Cc^{\ell,\lambda}(\OO)$ we denote the subset of functions in $C^{\ell,\lambda}(\OO)$ with compact support in $\OO$. Moreover, on $C_\b^{\ell,\lambda}(\OO)$ we define the norm
\begin{equation*}
  \|h\|_{C^{\ell,\lambda}(\OO)}:= \sum_{|\alpha|\leq \ell}\sup_{x\in\OO}|\d^\alpha h(x)|+ \sum_{|\alpha|=\ell}[\d^\alpha h]_{\lambda, \OO}.
\end{equation*}

\begin{definition}\label{def:domains}
  Let $\Dom \subseteq \R^d$ be a domain, i.e., a connected open set. Let $\ell \in \N_0$ and $\lambda\in [0,1]$.
  \begin{enumerate}[(i)]
    \item We call $\mc{O}$ a \emph{special $\Cc^{\ell,\lambda}$-domain} if, after translation and rotation, it is of the form
        \begin{equation}\label{eq:specialdomainh}
          \mc{O} = \cbrace{(x_1,\tilde{x})\in \R^d: x_1>h(\tilde{x})}
        \end{equation}
        for some $h \in \Cc^{\ell,\lambda}(\R^{d-1};\RR)$.
    \item Given a special $\Cc^{\ell,\lambda}$-domain $\mc{O}$, we define
\begin{equation*}
[\mc{O}]_{C^{\ell,\lambda}}:= \nrm{h}_{C^{\ell,\lambda}(\R^{d-1})},
\end{equation*}
where $h\in \Cc^{\ell,\lambda}(\R^{d-1};\RR)$ is such that, after rotation and translation, \eqref{eq:specialdomainh} holds. Note that $[\mc{O}]_{C^{\ell,\lambda}}$ is uniquely defined due to the compact support of $h$.
    \item We call $\mc{O}$ a \emph{$C^{\ell,\lambda}$-domain} if every boundary point $x \in \BDom$ admits an open neighbourhood $V$ such that
        \begin{equation*}
        \mc{O}\cap V = W \cap V \qquad \text{and}\qquad \BDom \cap V = \partial W \cap V
        \end{equation*}
        for some special $\Cc^{\ell,\lambda}$-domain $W$.
  \end{enumerate}
  If $\lambda=0$, then we write $C^\ell$ for $C^{\ell,0}$ in the definitions above.
\end{definition}
For any $\delta>0$ and $C^\ell$-domain $\mc{O}$, the special $\Cc^\ell$-domains $W$ can always be chosen such that $[W]_{C^\ell}<\delta$. If $\lambda\in(0,1]$, $\eps\in (0, \lambda)$ and $\OO$ is a $C^{\ell,\lambda}$-domain, then for any $\delta>0$, the special $\Cc^{\ell,\lambda}$-domains $W$ can be chosen such that $[W]_{C^{\ell,\lambda-\eps}}<\delta$. Indeed, if $h\in \Cc^{\ell,\lambda}(\RR^{d-1};\RR)$ is associated with $W$, then for any $|\alpha|=\ell$, we have
\begin{equation*}
  [\d^\alpha h]_{\lambda-\eps,\OO} = \sup_{\substack{x, y\in \OO\\ x\neq y}}\frac{|\d^\alpha h(x)-\d^\alpha h(y)|}{|x-y|^\lambda}|x-y|^\eps< \delta,
\end{equation*}
whenever $|x-y|^\eps$ is small enough. Note that for $\eps=0$, the quantity $[\d^{\alpha} h]_{\lambda,\OO}$ cannot be made arbitrarily small.

\section{Weighted Sobolev spaces and trace characterisations}\label{sec:weighted_Sob_spaces}
Let $\Dom\subseteq \R^{d}$ be a domain with non-empty boundary $\BDom$. A locally integrable function $w:\OO\to (0,\infty)$ is called a \emph{weight}. For $\gam \in\RR$ we define the spatial power weight $w^{\BDom}_{\gam}$ on $\Dom$ by
\begin{equation*}
w^{\BDom}_{\gam}(x) := \mrm{dist}(x,\BDom)^{\gam}, \qquad x \in \Dom,
\end{equation*}
and denote $w_{\gam} := w_\gam^{\smash{\d\RRdh}}$.

For $p \in [1,\infty)$, $\gam\in \RR$ and $X$ a Banach space we define the weighted Lebesgue space $L^p(\Dom,w^{\BDom}_{\gam};X)$ as the Bochner space consisting of all strongly measurable $f\colon \mc{O}\to X$ such that
\begin{equation*}
\nrm{f}_{L^p(\Dom,w^{\BDom}_{\gam};X)} := \has{\int_{\Dom}\|f(x)\|^p_X\:w^{\BDom}_{\gam}(x)\dd x }^{1/p}<\infty.
\end{equation*}
Let $w_\gam^{\d\OO}$ be such that $(w_\gam^{\d\OO})^{-\frac{1}{p-1}}\in L^1_{\loc}(\OO)$. The $k$-th order weighted Sobolev space for $k \in \N_0$ is defined as
\begin{equation*}
W^{k,p}(\Dom,w^{\BDom}_{\gam};X) := \left\{ f \in \mc{D}'(\Dom;X) : \forall |\alpha| \leq k, \partial^{\alpha}f \in L^p(\Dom,w^{\BDom}_{\gam};X) \right\}
\end{equation*}
equipped with the canonical norm. If $\gam=0$, then we simply write $W^{k,p}(\OO;X)$.
\begin{remark}\label{rem:L1loc}
  The local $L^1$ condition for $(w_\gam^{\d\OO})^{-\frac{1}{p-1}}$ ensures that all the derivatives $\d^{\alpha}f$ are locally integrable in $\OO$. If $\OO$ is the half-space $\RRdh$ or a bounded domain, then this condition holds for all $\gam\in\RR$. For $\OO=\RR^d$ the local $L^1$ condition holds only for weights $w_\gam(x)=|x_1|^\gam$ with $\gam\in(-\infty,p-1)$. For $\gam\geq p-1$, one has to be careful with defining the weighted Sobolev spaces on the full space because functions might not be locally integrable near $x_1=0$, see \cite{KO1984}. This explains why, for example, we cannot employ classical reflection arguments from $\RRdh$ to $\RR^d$ if $\gam>p-1$.
\end{remark}

Let $p\in(1,\infty)$, $k\in\NN_0$, $\gam>-1$ and let $X$ be a Banach space. To impose zero boundary conditions, we define
\begin{equation}\label{eq:Sobolevzerodef}
  \cir{W}_0^{k,p}(\Dom,w^{\BDom}_{\gam};X) := \overline{\Cc^\infty(\Dom;X)}^{W^{k,p}(\Dom,w^{\BDom}_{\gam};X)}.
\end{equation}
Furthermore, to impose Dirichlet and Neumann boundary conditions, we set
\begin{align*}
  C^\infty_{\cDir}(\overline{\mc{O}};X)&:= C^{\infty}(\OO;X)\cap\cbraceb{f \in \Cc(\overline{\mc{O}};X): f|_{\partial\mc{O}}=0},\\
 C^\infty_{\cNeu}(\overline{\mc{O}};X)&:= C^{\infty}(\OO;X)\cap\cbraceb{f \in \Cc^1(\overline{\mc{O}};X): (\d_{\vec{n}} f)|_{\partial\mc{O}}=0},
\end{align*}
which contain functions that are smooth in the interior of $\OO$, satisfy the boundary condition and have compact support at infinity (in the case of unbounded domains). Here, $\vec{n}$ denotes the inward unit normal vector at $\d\OO$ and $\d_{\vec{n}}= \vec{n}\cdot \grad$. We define
\begin{equation}\label{eq:Sobolevdirdef}
\begin{aligned}
  \cir{W}^{k,p}_{\Dir}(\Dom,w^{\BDom}_{\gam};X)& := \overline{C^\infty_{\cDir}(\overline{\Dom};X)}^{W^{k,p}(\Dom,w^{\BDom}_{\gam};X)},\\
  \cir{W}^{k,p}_{\Neu}(\Dom,w^{\BDom}_{\gam};X) &:= \overline{C^\infty_{\cNeu}(\overline{\Dom};X)}^{W^{k,p}(\Dom,w^{\BDom}_{\gam};X)}.
  \end{aligned}
\end{equation}
The notation $\cir{W}_0^{k,p}$, $\cir{W}_\Dir^{k,p}$ and $\cir{W}_\Neu^{k,p}$ as in \eqref{eq:Sobolevzerodef} and \eqref{eq:Sobolevdirdef} will mean that the spaces are defined as the closure of some space of test functions. Alternative characterisations of these spaces with boundary conditions in terms of traces (which will be denoted by $W_0^{k,p}$, $W_\Dir^{k,p}$ and $W_\Neu^{k,p}$) are derived in Sections \ref{subsec:trace-char}, \ref{subsec:tracechar_dom} and \ref{subsec:tracechar_bdd_dom}. The characterisations involving traces are also used in \cite{LLRV24, LV18} to define Sobolev spaces with boundary conditions.\\

We recall from \cite[Lemma 3.1]{LV18} that for $p\in[1,\infty)$, $\gam\in (-\infty,p-1)$ and $X$ a Banach space, we have the Sobolev embedding
\begin{equation*}
  W^{1,p}(\RR_+, w_{\gam}; X)\hookrightarrow C([0,\infty);X).
\end{equation*}
Hardy’s inequality plays a central role in the analysis of weighted Sobolev spaces. We state a version on $\R_+$ from \cite[Lemma 3.2]{LV18}. A version for $\RRdh$ will be given in Corollary \ref{cor:Sob_embRRdh}. For Hardy's inequality on more general domains, the reader is referred to \cite[Section 8.8]{Ku85}.
\begin{lemma}[Hardy's inequality on $\RR_+$]\label{lem:Hardy}
  Let $p\in[1,\infty)$ and let $X$ be a Banach space. Let $u\in W^{1,p}(\RR_+,w_{\gam};X)$ and assume either
\begin{enumerate}[(i)]
\item $\gam<p-1$ and $u(0)=0$, or,
\item $\gam>p-1$.
\end{enumerate}Then
  \begin{equation*}
    \|u\|_{L^p(\RR_+,w_{\gam-p};X)}\leq C_{p,\gam}\|u'\|_{L^p(\RR_+,w_{\gam};X)}.
  \end{equation*}
\end{lemma}

\subsection{Trace characterisations for weighted Sobolev spaces on the half-space}\label{subsec:trace-char} In the following three sections, we present characterisations of the spaces in \eqref{eq:Sobolevzerodef} and \eqref{eq:Sobolevdirdef} as closed subspaces of $W^{k,p}(\OO, w^{\d\OO}_{\gam};X)$ with vanishing traces. In this section, we start with the special case $\OO=\RRdh$.\\

For $p\in(1,\infty)$, $k\in\NN_0$, $\gam\in (-1,\infty)\setminus\{jp - 1:j\in\NN_1\}$ and $X$ a Banach space, we define the following spaces with vanishing traces
  \begin{align*}
W^{k,p}_{0}(\RR_+^d, w_{\gam}; X)&:=\left\{f\in W^{k,p}(\RR_+^d,w_{\gam};X): \operatorname{Tr}( \d^{\alpha}f)=0 \text{ if }k-|\alpha|>\tfrac{\gam+1}{p}\right\}, \\
  W^{k,p}_{\Dir}(\RR_+^d, w_{\gam}; X)&:=\left\{f\in W^{k,p}(\RR_+^d,w_{\gam};X): \operatorname{Tr}(f)=0 \text{ if }k>\tfrac{\gam+1}{p}\right\},\\
  W^{k,p}_{\Neu}(\RR_+^d, w_{\gam}; X)&:=\left\{f\in W^{k,p}(\RR_+^d,w_{\gam};X): \operatorname{Tr}(\d_1f)=0 \text{ if }k-1>\tfrac{\gam+1}{p}\right\}.
\end{align*}
All the traces in the above definitions are well defined, see \cite[Section 3.1]{LLRV24}. Although we will not consider weights $w_{\gam}$ with $\gam\leq -1$, we can nonetheless define
\begin{equation*}
  W^{k,p}_\Dir(\RRdh, w_{\gam};X):=W^{k,p}_0(\RRdh, w_{\gam};X):=W^{k,p}(\RRdh,w_{\gam};X),
\end{equation*}
see \cite[Lemma 3.1(2)]{LV18}.\\

In \cite{LV18} the above spaces are also used to define weighted Sobolev spaces on domains. However, since we consider domains with low regularity, we cannot do this, as will be explained in Remark \ref{rem:notPhi}.  Therefore, we first defined the Sobolev spaces as the closure of test functions in \eqref{eq:Sobolevzerodef} and \eqref{eq:Sobolevdirdef}. The following proposition relates the spaces $W^{k,p}_{{\rm BC}}$ and $\cir{W}^{k,p}_{{\rm BC}}$, where $\BC\in \{0, \Dir, \Neu\}$ stands for boundary conditions. That is, we prove that certain classes of test functions are dense in Sobolev spaces with zero trace conditions. 
\begin{proposition}[Trace characterisation on $\RRdh$]\label{prop:tracechar_RRdh}
  Let $p\in(1,\infty)$, $k\in\NN_0$, $\gam\in (-1,\infty)\setminus\{jp - 1:j\in\NN_1\}$ and let $X$ be a Banach space. 
   For $\BC\in \{0, \Dir, \Neu\}$ we have the trace characterisations
 \begin{equation*}
   \cir{W}_{\BC}^{k,p}(\RRdh,w_{\gam};X)= W_{\BC}^{k,p}(\RRdh,w_{\gam};X).
 \end{equation*}
\end{proposition}
\begin{proof}
From \cite[Proposition 3.8]{LV18} we have that $\Cc^{\infty}(\RRdh;X)$ is dense in $W^{k,p}_{0}(\RRdh,w_{\gam};X)$ and therefore the trace characterisation for $\cir{W}^{k,p}_0(\RRdh,w_{\gam};X)$ follows.

Let $({\rm BC},j)\in \{(\Dir,0),(\Neu,1)\}$. Then \cite[Proposition 4.8]{Ro25} implies that 
\begin{equation*}
  \overline{\cbraceb{f \in \Cc^{\infty}(\overline{\RRdh};X): (\d_1^j f)|_{\partial\RRdh}=0}}^{W^{k,p}(\RRdh,w_{\gam};X)} = W^{k,p}_{{\rm BC}}(\RR_+^d, w_{\gam}; X).
\end{equation*}
Since 
\begin{equation*}
  \cbraceb{f \in \Cc^{\infty}(\overline{\RRdh};X): (\d_1^j f)|_{\partial\RRdh}=0}\subseteq C^\infty_{{\rm c, BC}}(\overline{\RRdh};X),
\end{equation*}
the trace characterisations for the Dirichlet and Neumann boundary conditions follow.
\end{proof}

Before we continue with trace characterisations on domains, we record the following Hardy inequalities. As a corollary of Hardy's inequality on $\RR_+$ (Lemma \ref{lem:Hardy}), we have the following Hardy's inequality on $\RRdh$, see also \cite[Corollary 3.4]{LV18}. 
\begin{corollary}[Hardy's inequality on $\RRdh$]\label{cor:Sob_embRRdh}
  Let $p\in(1,\infty)$, $k\in\NN_1$, $\gam\in \RR$ and let $X$ be a Banach space. Then 
        \begin{align*}
     W_0^{k,p}(\RRdh,w_{\gam};X)&\hookrightarrow W^{k-1,p} (\RRdh,w_{\gam-p};X) &&\text{ if }\gam<p-1,\\
     W^{k,p}(\RRdh,w_{\gam};X)&\hookrightarrow W^{k-1,p} (\RRdh,w_{\gam-p};X) &&\text{ if }\gam>p-1,\\
      W_0^{k,p}(\RRdh,w_{\gam};X)&\hookrightarrow W_0^{k-1,p} (\RRdh,w_{\gam-p};X)&&\text{ if }\gam\notin \{jp-1:j\in\NN_1\}.
  \end{align*}
\end{corollary}

Moreover, as a consequence of Hardy's inequality above, we obtain the following non-sharp Hardy's inequality. 
\begin{lemma}\label{lem:frac_Hardy2}
  Let $p\in(1,\infty)$, $\gam\in (-1,\infty)\setminus\{jp-1:j\in\NN_1\}$, $s\in [0,\infty)$ such that $\gam>sp-1$ and let $X$ be a Banach space. Then for any integer $k\geq s$ it holds that
  \begin{equation*}
    W^{k,p}(\RRdh, w_{\gam};X)\hookrightarrow L^p(\RRdh, w_{\gam-sp};X).
  \end{equation*}
\end{lemma}
\begin{proof}
  Let $\ph_1,\ph_2\in C^{\infty}(\RR_+; [0,1])$ such that $\ph_1(x_1)=0$ for $x_1\geq 2$ and $\ph_2(x_1)=0$ for $x_1\leq 1$. In addition, take $\ph_1$ and $\ph_2$ such that $\ph_1+\ph_2=1$. Let $f\in W^{k,p}(\RRdh, w_\gam;X)$, with Hardy's inequality (Corollary \ref{cor:Sob_embRRdh} using that $\gam>sp-1$) we obtain
  \begin{align*}
    \|f\|_{L^p(\RRdh,w_{\gam-sp};X)} & \leq \|f \ph_1\|_{W^{k,p}(\RRdh, w_{\gam + (k-s)p};X)}+ \|f\ph_2\|_{L^p(\RRdh, w_{\gam-sp};X)}\\
    &\lesssim \|f\ph_1\|_{W^{k,p}(\RRdh, w_{\gam};X)} + \|f\ph_2\|_{L^p(\RRdh,w_\gam;X)}\lesssim \|f\|_{W^{k,p}(\RRdh,w_\gam;X)},
  \end{align*}
  where we have used that $w_{\gam+(k-s)p}(x)\lesssim w_{\gam}(x)$ for $x_1\leq 2$ (since $k\geq s$) and $w_{\gam-sp}(x)\lesssim w_{\gam}(x)$ for $x_1\geq 1$.
\end{proof}

Occasionally, we also need a sharp Hardy's inequality with fractional smoothness. We use complex interpolation to deal with spaces with fractional smoothness and weights $w_{\gam+kp}$ with $\gam\in(-1,p-1)$ and $k\in\NN_1$ outside the Muckenhoupt class.
\begin{lemma}\label{lem:frac_Hardy3}
    Let $p\in(1,\infty)$, $k\in\NN_0$, $\gam\in(-1,p-1)$, $s\in [0,1)$ such that $\gam>sp-1$ and let $X$ be a $\UMD$ Banach space. Then
    \begin{equation*}
        [W^{k,p}(\RRdh, w_{\gam+kp};X), W^{k+2,p}(\RRdh, w_{\gam+kp};X)]_{\frac{s+1}{2}} \hookrightarrow W^{k+1,p}(\RRdh, w_{\gam+(k-s)p};X).
    \end{equation*}
\end{lemma}
\begin{proof} For $s=0$ the result follows from \cite[Proposition 6.3]{Ro25}, so from now on we assume $s\in (0,1)$.
    We start with the case $k=0$. Let $|\alpha|\leq 1$, then by \cite[Lemma 3.7]{LMV17} (which also holds on $\RRdh$) and \cite[Propositions 5.5 \& 5.6]{LMV17}, we obtain
    \begin{align*}
        \|\d^\alpha f\|_{L^p(\RRdh, w_{\gam-sp};X)}& \lesssim \|\d^\alpha f\|_{H^{s,p}(\RRdh, w_{\gam};X)}
        \lesssim \|f\|_{H^{s+1,p}(\RRdh, w_{\gam};X)}\\
        &\eqsim \|f\|_{[L^p(\RRdh, w_{\gam};X), W^{2,p}(\RRdh, w_{\gam};X)]_{\frac{s+1}{2}} },
    \end{align*}
    where $H^{s,p}(\RRdh, w_{\gam};X)$ is a weighted Bessel potential space, see \cite[Section 3]{LMV17}.
    For $k\geq 1$, we proceed by induction. Assume that the statement of the lemma holds for some $k\in \NN_0$, then it remains to prove the statement for $k+1$. We recall from \cite[Section 3.2]{LLRV24} that $M$ is the pointwise multiplication operator given by $M u(x) = x_1 u(x)$ for $x\in \RRdh$. Then by \cite[Lemma 3.8]{LLRV24} (using that $\gam>sp-1$) and the induction hypothesis, we obtain
    \begin{align*}
        \|f\|_{W^{k+2,p}(\RRdh,w_{\gam+(k+1-s)p};X)} &\eqsim \sum_{|\beta|\leq 1} \|M \d^\beta f\|_{W^{k+1,p}(\RRdh, w_{\gam+(k-s)p};X)}\\
        &\lesssim \sum_{|\beta|\leq 1}\|M \d^\beta f \|_{ [W^{k,p}(\RRdh, w_{\gam+kp};X), W^{k+2,p}(\RRdh, w_{\gam+kp};X)]_{\frac{s+1}{2}} }\\
        &\lesssim \|f\|_{ [W^{k+1,p}(\RRdh, w_{\gam+(k+1)p};X), W^{k+3,p}(\RRdh, w_{\gam+(k+1)p};X)]_{\frac{s+1}{2}}},
    \end{align*}
    where the last estimate follows from the fact that for $|\beta|\leq 1$ the operators
    \begin{align*}
        M\d^\beta&:W^{k+1,p}(\RRdh, w_{\gam+(k+1)p};X)\to  W^{k,p}(\RRdh, w_{\gam+kp};X) \quad \text{ and }\\
        M\d^\beta&:W^{k+3,p}(\RRdh, w_{\gam+(k+1)p};X)\to  W^{k+2,p}(\RRdh, w_{\gam+kp};X)
    \end{align*}
    are bounded, see \cite[Lemma 3.6]{LLRV24}. 
\end{proof}

\subsection{Trace characterisations for weighted Sobolev spaces on special domains} \label{subsec:tracechar_dom}
For $\OO=\RRdh$ we have shown in Proposition \ref{prop:tracechar_RRdh} that the definition of weighted Sobolev spaces in \eqref{eq:Sobolevzerodef} and \eqref{eq:Sobolevdirdef} is equivalent to setting certain traces to zero. To define Sobolev spaces with vanishing traces for a special $\Cc^{\ell,\lambda}$-domain $\mc{O}$, we will employ the diffeomorphisms $\Phi, \Psi:\OO\to \RRdh$ from Lemmas \ref{lem:localization_weighted_blow-up} and \ref{lem:loc_normal} to construct isomorphisms between Sobolev spaces on $\OO$ and $\RRdh$. Which diffeomorphism we use depends on the boundary conditions. Throughout the rest of this paper, we will always use the diffeomorphism $\Phi$ from Lemma \ref{lem:localization_weighted_blow-up} for Dirichlet boundary conditions and the diffeomorphism $\Psi$ from Lemma \ref{lem:loc_normal} for Neumann boundary conditions. The diffeomorphism $\Phi$ is not applicable for Neumann boundary conditions, since it does not preserve the direction of the normal vector.

\begin{proposition}\label{prop:isomorphisms}
  Let $p \in (1,\infty)$, $\ell\in \NN_1$, $\lambda\in[0,1]$, $ k \in \N_0$ and let $X$ be a Banach space. Let $\gam\in (-1,\infty)\setminus\{jp-1:j\in\NN_1\}$ be such that $\gam > (k-(\ell+\lambda))_+p-1$.

\begin{enumerate}[(i)]
    \item\label{it:prop:isomorphism1}  Let $\OO$ be a special $\Cc^{\ell,\lambda}$-domain with $[\OO]_{C^{\ell,\lambda}}\leq 1$. Let $\Phi\colon \mc{O} \to \R^d_+$ be as in Lemma \ref{lem:localization_weighted_blow-up} and consider the change of coordinates mappings
  \begin{align*}
\Phi_*&\colon W^{k,p}(\Dom,w^{\BDom}_\gam;X) \to W^{k,p}(\R^d_+,w_\gam;X),\\
\Phi_*&\colon \cir{W}_{\BC}^{k,p}(\Dom,w^{\BDom}_\gam;X) \to \cir{W}_{\BC}^{k,p}(\R^d_+,w_\gam;X)\quad \text{ for }\,\BC\in \{0, \Dir\},
\end{align*}
defined by $\Phi_* f := f\circ \Phi^{-1}$.

\item\label{it:prop:isomorphism2} Let $\OO$ be a special $\Cc^{\ell,\lambda}$-domain with $[\OO]_{C^{\ell,\lambda}}\leq \Lambda$, where $\Lambda\in (0,1)$ is as in Lemma \ref{lem:loc_normal}.
Let $\Psi\colon \mc{O} \to \R^d_+$ be as in Lemma \ref{lem:loc_normal} and consider the change of coordinates mappings 
\begin{subequations}
  \begin{align}
\Psi_*&\colon W^{k,p}(\Dom,w^{\BDom}_\gam;X) \to W^{k,p}(\R^d_+,w_\gam;X),\label{eq:isoPsi*1}\\
\Psi_*&\colon \cir{W}_{\BC}^{k,p}(\Dom,w^{\BDom}_\gam;X) \to \cir{W}_{\BC}^{k,p}(\R^d_+,w_\gam;X)\quad \text{ for }\,\BC\in \{0, \Dir, \Neu\},\label{eq:isoPsi*2}
\end{align}
\end{subequations}
defined by $\Psi_* f := f\circ \Psi^{-1}$. 
\end{enumerate}
Then $\Phi_*$ and $\Psi_*$ are isomorphisms of Banach spaces for which $(\Phi^{-1})_{*}$ and $(\Psi^{-1})_{*}$, respectively, act as inverse.
\end{proposition}

\begin{proof} We give the proof of \ref{it:prop:isomorphism2} and the proof of \ref{it:prop:isomorphism1} is similar using Lemma \ref{lem:localization_weighted_blow-up} instead of Lemma \ref{lem:loc_normal}.

\textit{Step 1: proof of \eqref{eq:isoPsi*1}.} We start with some preparations. Let $k\in \NN_1$ and $f\in \Cc^{\ell,\lambda}(\overline{\OO};X)$. Note that by Lemma \ref{lem:loc_normal}
we have that $\Psi_* f\in \Cc^{\ell,\lambda}(\overline{\RRdh};X)$. 
Let $\alpha \in \N^d_0 \setminus \{0\}$ with $|\alpha| \leq k$, then by \cite[Theorem 2.1]{CS96} we have the multivariate Fa\`a di Bruno's formula
$$
\d^\alpha \Psi_* f = \sum_{ 1 \leq |\beta| \leq |\alpha|} (\Psi_* \d^\beta f ) \sum_{s=1}^{|\alpha|}\sum_{p_s(\alpha,\beta)}  \prod_{j=1}^{s}c_{\alpha,\vec{k}_j,\vec{\ell}_j}[\d^{\vec{\ell}_j}\Psi^{-1}]^{\vec{k}_j},
$$
for some constants $c_{\alpha,\vec{k}_j,\vec{\ell}_j}$ and sets $p_s(\alpha,\beta)$ contained in
\begin{align}\label{eq:setp_s1}
 \Big\{ (\vec{k}_1,\ldots,\vec{k}_s;\vec{\ell}_1,\ldots,\vec{\ell}_s) \in (\N^d_0 \setminus \{0\})^{s} \times (\N^d_0 \setminus \{0\})^{s}
 : \sum_{j=1}^s|\vec{k}_j| = |\beta|, \sum_{j=1}^s|\vec{k}_j||\vec{\ell}_j| = |\alpha| \Big\}.
\end{align}
Therefore, we have 
\begin{align}
\|\d^\alpha \Psi_* f\|_{L^p(\RRdh,w_{\gam};X)}
\lesssim&\;  \sum_{ 1 \leq |\beta| \leq |\alpha|}  \sum_{s=1}^{|\alpha|}\sum_{p_s(\alpha,\beta)}  \|(\Psi_* \d^\beta f ) \prod_{j=1}^{s}[\d^{\vec{\ell}_j}\Psi^{-1}]^{\vec{k}_j}\|_{L^p(\RRdh,w_{\gam};X)}  \nonumber \\
\lesssim &\; \sum_{ 1 \leq |\beta| \leq |\alpha|}  \sum_{s=1}^{|\alpha|}\sum_{p_s(\alpha,\beta)}  \|\Psi_* \d^\beta f\|_{L^p(\RRdh,w_{\gam-\sum_{j=1}^s(|\vec{\ell}_j|-(\ell+\lambda))_+|\vec{k}_j|p};X)}   \nonumber \\
&\;\;\;\cdot \prod_{j=1}^{s}\| y \mapsto y_1^{(|\vec{\ell}_j|-(\ell+\lambda))_+} \d^{\vec{\ell}_j}\Psi^{-1}(y)\|_{L^{\infty}(\RRdh; \RRd)}^{|\vec{k}_j|}.
\label{eq:prop:isomorphisms;proof_est1}
\end{align}
From Lemma \ref{lem:loc_normal}\ref{it:lem:loc_normal5} we obtain 
\begin{equation}\label{eq:prop:isomorphisms;proof_est2}
\prod_{j=1}^{s}\| y \mapsto y_1^{(|\vec{\ell}_j|-(\ell+\lambda))_+} \d^{\vec{\ell}_j}\Psi^{-1}(y)\|_{L^{\infty}(\RRdh; \RR^d)}^{|\vec{k}_j|} \lesssim 1.
\end{equation}

\textit{Step 1a: proof of \eqref{eq:isoPsi*1} if $\ell+\lambda\geq k$.} If $k=0$, then \eqref{eq:isoPsi*1} follows immediately from Lemma \ref{lem:loc_normal}. Let $k\in\NN_1$ and note that $|\vec{\ell}_j|\leq |\alpha|\leq k\leq \ell+\lambda$. Therefore, $(|\vec{\ell}_j|-(\ell+\lambda))_+=0$ in \eqref{eq:prop:isomorphisms;proof_est1} and the case $k=0$ implies
\begin{equation}\label{eq:prop:isomorphisms;proof_est4}
  \|\Psi_* \d^\beta f\|_{L^p(\RRdh,w_{\gam};X)} \lesssim \| \d^\beta f\|_{L^p(\OO,w^{\d\OO}_{\gam};X)} \leq \|f\|_{W^{k,p}(\OO, w_{\gam}^{\d\OO};X)},\qquad 1\leq |\beta|\leq|\alpha |,
\end{equation}
and we find
\begin{equation*}
  \|\Psi_* f\|_{W^{k,p}(\RRdh, w_{\gam};X)}\lesssim \|f\|_{W^{k,p}(\OO, w_{\gam}^{\d\OO};X)},\qquad f\in \Cc^{\ell,\lambda}(\overline{\OO};X),
\end{equation*}
and by density the estimate extends to $f\in W^{k,p}(\OO, w_{\gam}^{\d\OO};X)$. Recall from Lemma \ref{lem:loc_normal} that $\Psi$ is invertible and thus $(\Psi^{-1})_*$ is the inverse of $\Psi_*$.
The estimate for the inverse $(\Psi^{-1})_*$ can be shown using similar estimates as in \eqref{eq:prop:isomorphisms;proof_est1}, \eqref{eq:prop:isomorphisms;proof_est2} and \eqref{eq:prop:isomorphisms;proof_est4}. This shows that $\Psi_*$ in \eqref{eq:isoPsi*1} is an isomorphism if $\ell+\lambda\geq k$.

\textit{Step 1b: proof of \eqref{eq:isoPsi*1} if $\ell+\lambda < k$.}
We claim that in \eqref{eq:prop:isomorphisms;proof_est1} we have
\begin{equation}\label{eq:cond_Hardy}
  \gam-\sum_{j=1}^s(|\vec{\ell}_j|-(\ell+\lambda))_+|\vec{k}_j|p > -1.
\end{equation}
Indeed, if $|\vec{\ell}_j| \leq \ell+\lambda$ for all $j\in\{1,\dots,s\}$, then
$$
\gam-\sum_{j=1}^s(|\vec{\ell}_j|-(\ell+\lambda))_+|\vec{k}_j|p = \gam > (k-(\ell+\lambda))p-1 > -1,
$$
and if $|\vec{\ell}_{j_0}| > \ell+\lambda$ for some $j_0 \in \{1,\ldots,s\}$, then with \eqref{eq:setp_s1} we obtain
\begin{align*}
\gam-\sum_{j=1}^s(|\vec{\ell}_j|-(\ell+\lambda))_+|\vec{k}_j|p
&= \gam-\Big(\sum_{\substack{j=1\\j \neq j_0}}^s(|\vec{\ell}_j|-(\ell+\lambda))_+|\vec{k}_j| + (|\vec{\ell}_{j_0}|-(\ell+\lambda))|\vec{k}_{j_0}| \Big)p  \\
&\geq \gam-\Big(\sum_{\substack{j=1\\j \neq j_0}}^s|\vec{\ell}_j||\vec{k}_j| + |\vec{\ell}_{j_0}||\vec{k}_{j_0}| - (\ell+\lambda) \Big)p  \\
&= \gam - (|\alpha|-(\ell+\lambda))p \geq \gam - (k-(\ell+\lambda))p > -1.
\end{align*}
Moreover, again by \eqref{eq:setp_s1} we have
\begin{equation}\label{eq:cond_Hardy2}
  \begin{aligned}
\sum_{j=1}^s(|\vec{\ell}_j|-(\ell+\lambda))_+|\vec{k}_j|
&\leq \sum_{j=1}^s |\vec{\ell}_j||\vec{k}_j|-|\beta| = |\alpha|-|\beta| \leq k-|\beta|.
\end{aligned}
\end{equation}
Therefore, by Lemma \ref{lem:frac_Hardy2} (using \eqref{eq:cond_Hardy} and \eqref{eq:cond_Hardy2}) and Step 1a, we have for $1\leq |\beta|\leq |\alpha|\leq k=\ell+1$ that
\begin{equation}\label{eq:prop:isomorphisms;proof_est3}
  \begin{aligned}
\|\Psi_* \d^\beta f\|_{L^p(\RRdh,w_{\gam-\sum_{j=1}^s(|\vec{\ell}_j|-(\ell+\lambda))_+|\vec{k}_j|p};X)}
&\lesssim \|\Psi_*\d^\beta f\|_{W^{k-|\beta|,p}(\RRdh,w_{\gam};X)} \\
&\lesssim \|\d^\beta f\|_{W^{k-|\beta|,p}(\Dom,w^{\BDom}_{\gam};X)}\\
&\lesssim \|f\|_{W^{k,p}(\Dom,w^{\BDom}_{\gam};X)},\qquad f\in \Cc^{\ell,\lambda}(\overline{\OO};X).
\end{aligned}
\end{equation}
Now, density and \eqref{eq:prop:isomorphisms;proof_est1}, \eqref{eq:prop:isomorphisms;proof_est2} and \eqref{eq:prop:isomorphisms;proof_est3} yield that
\begin{equation}\label{eq:IH_iso}
  \begin{aligned}
  \Psi_*&:W^{k,p}(\mc{O},w_{\gam}^{\d\mc{O}};X)\to W^{k,p}(\RRdh,w_{\gam};X)
\end{aligned}
\end{equation}
is bounded for $k=\ell+1$.

The general case $k\geq \ell+1$ follows by induction on $k$. Assume that \eqref{eq:IH_iso} holds for some $k\geq \ell+1$ and let $1\leq |\beta|\leq |\alpha|\leq k+1$. Using the induction hypothesis instead of Step 1a in \eqref{eq:prop:isomorphisms;proof_est3}, we obtain the estimate
\begin{equation*}
  \|\Psi_* \d^\beta f\|_{L^p(\RRdh,w_{\gam-\sum_{j=1}^s(|\vec{\ell}_j|-(\ell+\lambda))_+|\vec{k}_j|p};X)}
\lesssim \|f\|_{W^{k+1,p}(\Dom,w^{\BDom}_{\gam};X)},
\end{equation*}
which proves \eqref{eq:IH_iso} for $k\geq \ell+1$.

The estimate for the inverse can be shown directly using similar estimates as in \eqref{eq:prop:isomorphisms;proof_est1} and \eqref{eq:prop:isomorphisms;proof_est2}, together with the estimate
\begin{equation*}
  \begin{aligned}
\|(\Psi^{-1})_* \d^\beta f\|_{L^p(\OO,w^{\d\OO}_{\gam-\sum_{j=1}^s(|\vec{\ell}_j|-(\ell+\lambda))_+|\vec{k}_j|p};X)}
&\lesssim \|\d^\beta f\|_{L^p(\RRdh,w_{\gam-\sum_{j=1}^s(|\vec{\ell}_j|-(\ell+\lambda))_+|\vec{k}_j|p};X)} \\
&\lesssim \|\d^\beta f\|_{W^{k-|\beta|,p}(\RRdh,w_{\gam};X)}\\
&\lesssim \|f\|_{W^{k,p}(\RRdh,w_{\gam};X)},\qquad f\in W^{k,p}(\RRdh,w_{\gam};X),
\end{aligned}
\end{equation*}
which follows from Step 1a and Lemma \ref{lem:frac_Hardy2}.
This completes the proof of \eqref{eq:isoPsi*1}.

\textit{Step 2: proof of \eqref{eq:isoPsi*2}.} The proof \eqref{eq:isoPsi*2} is similar to the proof of \eqref{eq:isoPsi*1} if we work with a suitable dense subspace, i.e.,
\begin{itemize}
  \item if $\BC=0$, take $f\in \Cc^{\infty}(\OO;X)$, 
  \item if $\BC\in\{\Dir, \Neu\}$, take $f\in C_{{\rm c}, \BC}^{\infty}(\overline{\OO};X)$,
\end{itemize}
see \eqref{eq:Sobolevzerodef} and \eqref{eq:Sobolevdirdef}. Note that in both cases Lemma \ref{lem:loc_normal} ensures that $\Psi_* f$ is in the respective dense subspace on $\RRdh$. In particular, for the Neumann boundary condition, we have
\begin{align*}
    (\d_{1}\Psi_* f)|_{\d\RRdh} = (\Psi_*(\d_{\nu} f))|_{\d\RRdh}, 
\end{align*}
where $\nu(y)=\nu(\tilde{y})=(1, -\grad_{\tilde{y}}h(\tilde{y}))^\top$ is the inward normal direction. Indeed, this follows from Lemma \ref{lem:loc_normal} since
\begin{align*}
    \d_{y_1}f(\Psi^{-1}(y)) &= (\grad f)(\Psi^{-1}(y))\cdot \d_{y_1}\Psi^{-1}(y)\\
    &= (\grad f)(\Psi^{-1}(y))\cdot \nu(\Psi^{-1}(y)) = (\d_\nu f)(\Psi^{-1}(y)),\qquad y=(0,\tilde{y})\in \d\RRdh,
\end{align*}
where we recall from the construction of $\Psi^{-1}$ in the proof of Lemma \ref{lem:loc_normal} that $\d_{y_1}\Psi^{-1}(y)=(1, -\grad_{\tilde{y}}h(\tilde{y}))^{\top} = \nu(\tilde{y})=\nu(\Psi^{-1}(y))$ if $y=(0,\tilde{y})\in \d\RRdh$. Furthermore, note that the conditions $(\d_{\vec{n}} f)|_{\d\OO}=0$ and $(\d_{\nu} f )|_{\d\OO}=0$ are equivalent.
\end{proof}

\begin{remark} 
By inspection of the proof of Proposition \ref{prop:isomorphisms}, we see that for $\BC=0$ no additional conditions on $\gam$ are necessary since Hardy's inequality always applies in this case. That is, we can allow for any $\gam\in (-1,\infty)\setminus\{jp-1:j\in\NN_1\}$. Furthermore, we expect that for Dirichlet boundary conditions, the range for $\gam$ can also be improved, although we will not need this. 
\end{remark}

We define the following spaces with vanishing traces at the boundary of a special $\Cc^{\ell,\lambda}$-domain. 
\begin{definition}\label{def:spaces_special}
  Let $p \in (1,\infty)$, $\ell\in \NN_1$, $\lambda\in[0,1]$, $ k \in \N_0$ and let $X$ be a Banach space. Let $\gam\in (-1,\infty)\setminus\{jp-1:j\in\NN_1\}$ be such that $\gam > (k-(\ell+\lambda))_+p-1$
  and let $\OO$ be a special $\Cc^{\ell,\lambda}$-domain. If $[\OO]_{C^{\ell,\lambda}}\leq 1$, let $\Phi_*$ be the isomorphism from Proposition \ref{prop:isomorphisms}\ref{it:prop:isomorphism1} and define
     \begin{align*}
  W_{0}^{k,p}(\Dom,w^{\BDom}_{\gam};X)&:= \cbraces{f \in W^{k,p}(\Dom,w^{\BDom}_{\gam};X): \Tr( \d^\alpha(\Phi_* f))=0 \text{ if } k -\abs{\alpha} >\tfrac{\gam+1}{p}},\\
  W^{k,p}_{\Dir}(\Dom,w^{\BDom}_{\gam};X)&:= \cbraces{f \in W^{k,p}(\Dom,w^{\BDom}_{\gam};X): \Tr(\Phi_*f)=0 \text{ if } k >\tfrac{\gam+1}{p}}.\intertext{If $[\OO]_{C^{\ell,\lambda}}\leq \Lambda$, where $\Lambda \in (0,1)$ is as in Lemma \ref{lem:loc_normal}, let $\Psi_*$ be the isomorphism from Proposition \ref{prop:isomorphisms}\ref{it:prop:isomorphism2} and define}
  W^{k,p}_{\Neu}(\Dom,w^{\BDom}_{\gam};X)&:= \cbraces{f \in W^{k,p}(\Dom,w^{\BDom}_{\gam};X): \Tr(\d_1(\Psi_*f))=0 \text{ if } k -1>\tfrac{\gam+1}{p}}.
\end{align*}
\end{definition}
The above spaces are well defined by Proposition \ref{prop:isomorphisms}. Furthermore, by Lemmas \ref{lem:localization_weighted_blow-up} and \ref{lem:loc_normal}, the definitions of the above spaces are consistent in the sense that viewing $\OO$ as either a special $\Cc^{\ell,\lambda}$-domain or a special $\Cc^1$-domain yields the same space. 
Moreover, the condition $\Tr(\d_1(\Psi_*f))=0$ correctly models the Neumann boundary condition, since $\Psi$ leaves the direction of the normal vector invariant, see Lemma \ref{lem:loc_normal}\ref{it:lem:loc_normal4}. Finally, we note that the spaces $W^{k,p}_0$ and $W^{k,p}_\Dir$ could also be defined using $\Psi_*$ instead of $\Phi_*$, yielding an equivalent definition by Proposition \ref{prop:tracechar_dom} below.
\\

Similar to Proposition \ref{prop:tracechar_RRdh} we can now characterise the spaces $\cir{W}^{k,p}_{\BC}(\Dom,w^{\BDom}_{\gam};X)$ in terms of vanishing traces with the aid of the isomorphisms from Proposition \ref{prop:isomorphisms}.

\begin{proposition}[Trace characterisation on special domains]\label{prop:tracechar_dom}
Let $p \in (1,\infty)$, $\ell\in \NN_1$, $\lambda\in[0,1]$, $ k \in \N_0$ and let $X$ be a Banach space. Let $\gam\in (-1,\infty)\setminus\{jp-1:j\in\NN_1\}$ be such that $\gam > (k-(\ell+\lambda))_+p-1$.

\begin{enumerate}[(i)]
    \item\label{it:prop_tracechar_dom1} Let $\OO$ be a special $\Cc^{\ell,\lambda}$-domain with $[\OO]_{C^{\ell,\lambda}}\leq 1$ and let $\Phi_*$ be the isomorphism from Proposition \ref{prop:isomorphisms}. Let $\BC\in \{0, \Dir\}$ and let $W_{\BC}^{k,p}$ be defined using $\Phi_*$. Then we have the trace characterisations
 \begin{equation*}
   \cir{W}_{\BC}^{k,p}(\Dom,w^{\BDom}_{\gam};X)= W_{\BC}^{k,p}(\Dom,w^{\BDom}_{\gam};X).
 \end{equation*}
 \item\label{it:prop_tracechar_dom2} Let $\OO$ be a special $\Cc^{\ell,\lambda}$-domain with $[\OO]_{C^{\ell,\lambda}}\leq \Lambda$, where $\Lambda\in(0,1)$ is as in Lemma \ref{lem:loc_normal}, and let $\Psi_*$ be the isomorphism from Proposition \ref{prop:isomorphisms}. Let $\BC\in \{0, \Dir, \Neu\}$ and let $W_{\BC}^{k,p}$ be defined using $\Psi_*$.Then we have the trace characterisations
 \begin{equation*}
   \cir{W}_{\BC}^{k,p}(\Dom,w^{\BDom}_{\gam};X)= W_{\BC}^{k,p}(\Dom,w^{\BDom}_{\gam};X).
 \end{equation*}
\end{enumerate}

\end{proposition}
\begin{proof} We only prove \ref{it:prop_tracechar_dom2}. The proof of \ref{it:prop_tracechar_dom1} is similar.
Let $\BC\in \{0, \Dir, \Neu\}$ and $f\in \cir{W}_{\BC}^{k,p}(\Dom,w^{\BDom}_{\gam};X)$, then by Propositions  \ref{prop:tracechar_RRdh} and \ref{prop:isomorphisms} we have $\Psi_*f \in \cir{W}^{k,p}_{\BC}(\RRdh, w_{\gam} ;X)=W^{k,p}_{\BC}(\RRdh, w_{\gam} ;X)$. This implies that all the required traces of $\Psi_*f$ are zero. Moreover, since $\Psi_*f\in W^{k,p}(\RRdh, w_{\gam};X)$ it follows by Proposition \ref{prop:isomorphisms} that $f=(\Psi^{-1})_*\Psi_* f\in W^{k,p}(\OO, w_{\gam}^{\d\OO};X)$ as well. 
This proves that $f\in W^{k,p}_{\BC}(\Dom,w^{\BDom}_{\gam};X)$. The other inclusion is similar.
\end{proof}

\begin{remark}\label{rem:notPhi}  If $h\in \Cc^{\ell,\lambda}(\RR^{d-1})$ is associated with the special $\Cc^{\ell,\lambda}$-domain, then the \emph{classical} diffeomorphism $\Phi_{\cl}:\overline{\OO}\to\overline{\RRdh}$ given by
\begin{equation*}
  \Phi_{\cl}(x) = (x_1-h(\tilde{x}),\tilde{x}), \qquad x=(x_1,\tilde{x})
  \in \overline{\mc{O}},
\end{equation*}
defines a $C^{\ell,\lambda}$-diffeomorphism. Moreover, the change of coordinates mapping $(\Phi_{\cl})_*$ becomes an isomorphism between $W_{{\rm BC}}^{k,p}(\Dom,w^{\BDom}_{\gam};X)$
and $W_{{\rm BC}}^{k,p}(\R^d_+,w_{\gam};X)$ for $\ell \geq k$ and $\BC\in \{0, \Dir\}$. In \cite[Section 3.2]{LV18}, this isomorphism is used to define weighted Sobolev spaces on domains. However, for $\ell < k$ or Neumann boundary conditions, this isomorphism is not sufficient, which is why we have employed the diffeomorphisms $\Phi$ and $\Psi$ from Lemma  \ref{lem:localization_weighted_blow-up} and \ref{lem:loc_normal} to define weighted Sobolev spaces with vanishing traces. We elaborate on the construction of the diffeomorphisms in Appendix \ref{sec:appendix_lemma}.
\end{remark}

\subsection{Trace characterisations for weighted Sobolev spaces on bounded domains} \label{subsec:tracechar_bdd_dom}
In this section, we define Sobolev spaces with vanishing traces for bounded domains $\mc{O}$. To this end, we will employ a localisation procedure to relate spaces on bounded domains with spaces on special domains. We start with a lemma containing a decomposition of weighted Sobolev spaces, see also \cite[Section 2.2]{LV18}.

\begin{lemma}\label{lem:decomp}
  Let $\ell\in\NN_1$, $\lambda\in[0,1]$ and let $\OO\subseteq\RR^d$ be a bounded $C^{\ell,\lambda}$-domain. Then for any $\delta>0$, the following statements hold.
\begin{enumerate}[(i)]
  \item\label{it:lem:decomp1} For all $\eps\in (0,\lambda)$ there exists a finite open cover $(V_n)_{n=1}^N$ of $\d \OO$, together with special $\Cc^{\ell,\lambda}$-domains $(\OO_n)_{n=1}^N$ which satisfy $[\OO_n]_{C^{\ell,\lambda-\eps}}<\delta$, such that
      \begin{equation*}
        \mc{O}\cap V_n = \OO_n \cap V_n \qquad \text{and}\qquad \BDom \cap V_n = \partial \OO_n \cap V_n,\quad n\in\{1,\dots, N\}.
        \end{equation*}
        If $\lambda=0$, then the special $\Cc^\ell$-domains $(\OO_n)_{n=1}^N$ can be chosen such that  $[\OO_n]_{C^{\ell}}<\delta$.
  \item\label{it:lem:decomp2} There exist $\eta_0\in \Cc^{\infty}(\OO)$ and $\eta_n \in \Cc^{\infty}(V_n)$ for $n\in\{1,\dots,N\}$ such that $0\leq \eta_n\leq 1$ for $n\in \{0,\dots, N\}$ and $\sum_{n=0}^{N}\eta_n^2=1$ on $\OO$ (partition of unity).
      \item\label{it:lem:decomp3} For $p\in(1,\infty)$, $k\in\NN_0$, $\gam\in\RR$ and $X$ a Banach space, the space $W^{k,p}(\OO, w_{\gam}^{\d\OO};X)$ has the direct sum decomposition
          \begin{equation}\label{eq:Fk}
            \WW^{k,p}_{\gam}:=W^{k,p}(\RRd;X)\oplus \bigoplus_{n=1}^NW^{k,p}(\OO_n,w_{\gam}^{\d\OO_n};X).
          \end{equation}
          Moreover, the mappings
           \begin{align*}
             \II \colon W^{k,p}(\OO, w_{\gam}^{\d\OO};X)\to \WW^{k,p}_{\gam}\quad \text{ and }\quad \PP\colon \WW^{k,p}_{\gam}\to W^{k,p}(\OO, w_{\gam}^{\d\OO};X)
           \end{align*}
           given by
          \begin{equation}\label{eq:retraction}
            \II f := (\eta_n f)_{n=0}^N\quad \text{ and }\quad \PP(f_n)_{n=0}^N:=\sum_{n=0}^N\eta_nf_n,
          \end{equation}
          are continuous and satisfy $\PP\II = \operatorname{id}$. Thus, $\PP$ is a retraction with coretraction $\II$.
\end{enumerate}
\end{lemma}
\begin{proof}
We note that the result in \ref{it:lem:decomp1} follows from the discussion after Definition \ref{def:domains} in Section \ref{subsec:Ckdomains}. The partition of unity in \ref{it:lem:decomp2} is standard, see for instance \cite[Section 8.4]{KrBook08} (noting that a $C^2$-domain is not required for constructing the partition of unity). Finally, using the partition of unity and the (co)retraction in \eqref{eq:retraction}, the direct sum decomposition in \ref{it:lem:decomp3} follows. Indeed, $\eta_0\in \Cc^\infty(\OO)$ and we can extend to the full space $\RRd$ without a weight since there is no boundary. Furthermore, for $n\in\{1,\dots, N\}$ we have $\eta_n\in \Cc^\infty(V_n)$, so the weight $w_\gam^{\d\OO}(x)$ can be replaced by $w_\gam^{\d\OO_n}(x)$ for $x\in \OO_n$.
\end{proof}

With Lemma \ref{lem:decomp} we can now define traces of functions in $W^{k,p}(\OO,w_{\gam}^{\d\OO};X)$ if $\OO$ is a bounded $C^{\ell,\lambda}$-domain. Furthermore, we define the following spaces with vanishing traces at the boundary. 
\begin{definition}\label{def:spaces_bounded}
  Let $p \in (1,\infty)$, $\ell\in \NN_1$, $\lambda\in[0,1]$, $ k \in \N_0$ and let $X$ be a Banach space. Let $\gam\in (-1,\infty)\setminus\{jp-1:j\in\NN_1\}$ be such that $\gam > (k-(\ell+\lambda))_+p-1$.
  Moreover, let $\OO$ be a bounded $C^{\ell,\lambda}$-domain, let $(\OO_n)_{n=1}^N$ be special $\Cc^{\ell,\lambda}$-domains and let $\mc{I}$ be the coretraction from Lemma \ref{lem:decomp}. We define
     \begin{align*}
  W_{0}^{k,p}(\Dom,w^{\BDom}_{\gam};X)&:= \cbraces{f \in W^{k,p}(\Dom,w^{\BDom}_{\gam};X): \mc{I} f\in W^{k,p}(\RRd;X)\oplus \bigoplus_{n=1}^NW_0^{k,p}(\OO_n,w_{\gam}^{\d\OO_n};X) },\\
  W^{k,p}_{\Dir}(\Dom,w^{\BDom}_{\gam};X)&:= \cbraces{f \in W^{k,p}(\Dom,w^{\BDom}_{\gam};X): \mc{I} f\in W^{k,p}(\RRd;X)\oplus \bigoplus_{n=1}^NW_{\Dir}^{k,p}(\OO_n,w_{\gam}^{\d\OO_n};X)},\\
  W^{k,p}_{\Neu}(\Dom,w^{\BDom}_{\gam};X)&:= \cbraces{f \in W^{k,p}(\Dom,w^{\BDom}_{\gam};X): \mc{I} f\in W^{k,p}(\RRd;X)\oplus \bigoplus_{n=1}^NW_{\Neu}^{k,p}(\OO_n,w_{\gam}^{\d\OO_n};X)}.
\end{align*}
\end{definition}
Note that the above spaces are well defined by Lemma \ref{lem:decomp} and Definition \ref{def:spaces_special}. Moreover, the definitions are independent of the chosen covering of $\d\OO$ and the partition of unity in Lemma \ref{lem:decomp}.\\

Similar to Propositions \ref{prop:tracechar_RRdh} and \ref{prop:tracechar_dom} we can now relate the spaces $\cir{W}^{k,p}_{\BC}(\Dom,w^{\BDom}_{\gam};X)$ and $W^{k,p}_{\BC}(\Dom,w^{\BDom}_{\gam};X)$ for bounded domains.
\begin{proposition}[Trace characterisation on bounded domains]\label{prop:tracechar_bdd_dom}
Let $p \in (1,\infty)$, $\ell\in \NN_1$, $\lambda\in[0,1]$, $ k \in \N_0$ and let $X$ be a Banach space. Let $\gam\in (-1,\infty)\setminus\{jp-1:j\in\NN_1\}$ be such that $\gam > (k-(\ell+\lambda))_+p-1$.
Moreover, let $\OO$ be a bounded $C^{\ell,\lambda}$-domain. For $\BC\in \{0, \Dir, \Neu\}$ we have the trace characterisations
 \begin{equation*}
   \cir{W}_{\BC}^{k,p}(\Dom,w^{\BDom}_{\gam};X)= W_{\BC}^{k,p}(\Dom,w^{\BDom}_{\gam};X).
 \end{equation*}
\end{proposition}
\begin{proof}
  We only prove the statement for $\BC=0$ since the proof for the other cases is similar. Let $f\in W^{k,p}_{0}(\OO, w_{\gam}^{\d\OO};X)$. Proposition \ref{prop:tracechar_dom} and the fact that $\Cc^\infty(\RRd;X)$ is dense in $W^{k,p}(\RRd;X)$, allows us to approximate $\mc{I}f$ by a sequence $g:=(g_{0,m}, g_{1,m},\dots, g_{N,m})_{m\geq 1}$ where $(g_{0,m})_{m\geq1}\subseteq \Cc^{\infty}(\RRd;X)$ and $(g_{n,m})_{m\geq1}\subseteq \Cc^{\infty}(\OO_n;X)$ for all $n\in\{1,\dots,N\}$. Using Lemma \ref{lem:decomp} we see that $f=\mc{P}\mc{I} f$ can be approximated by the sequence $\mc{P}g\subseteq \Cc^{\infty}(\OO;X)$. 
\end{proof}

\subsection{Complex interpolation of weighted Sobolev spaces} To conclude this section, we recall the following two interpolation results for weighted Sobolev spaces on $\RRdh$ with boundary conditions from \cite{Ro25}, which also hold for special and bounded domains by the results from Sections \ref{subsec:trace-char}, \ref{subsec:tracechar_dom} and \ref{subsec:tracechar_bdd_dom}. 
\begin{proposition}\label{prop:complex_int_W_Dir}
  Let $p\in(1,\infty)$, $k\in\NN_0$,  $\lambda\in[0,1]$, $\gam \in ((1-\lambda)p-1,2p-1)\setminus\{p-1\}$ and let $X$ be a $\UMD$ Banach space. Moreover, let $\OO$ be a special $\Cc^{1,\lambda}$-domain with $[\OO]_{C^{1,\lambda}} \leq 1$ or  a bounded $C^{1,\lambda}$-domain.
    Then
    \begin{equation*}
        [W^{k,p}(\OO, w^{\d\OO}_{\gam+kp};X), W^{k+2,p}_{\Dir}(\OO,w^{\d\OO}_{\gam+kp};X)]_{\half} = W^{k+1,p}_{\Dir}(\OO,w^{\d\OO}_{\gam+kp};X).
    \end{equation*}
\end{proposition}

\begin{proposition}\label{prop:complex_int_W_Neu}
Let $p\in(1,\infty)$, $k\in\NN_0$,  $\lambda\in(0,1]$, $\gam \in ((1-\lambda)p-1,p-1)$, $j\in\{0,1\}$ and let $X$ be a $\UMD$ Banach space. Moreover, let $\OO$ be a special $\Cc^{j+1,\lambda}$-domain with $[\OO]_{C^{j+1,\lambda}} \leq \Lambda$, where $\Lambda\in(0,1)$ is as in Lemma \ref{lem:loc_normal} or a bounded $C^{j+1,\lambda}$-domain.
    Then
    \begin{equation*}
        [W^{k+j,p}(\OO, w^{\d\OO}_{\gam+kp};X), W^{k+2+j,p}_{\Neu}(\OO,w^{\d\OO}_{\gam+kp};X)]_{\half} = W^{k+1+j,p}_{\Neu}(\OO,w^{\d\OO}_{\gam+kp};X).
    \end{equation*}
\end{proposition}
\begin{proof}[Proof of Propositions \ref{prop:complex_int_W_Dir} and \ref{prop:complex_int_W_Neu}]
  By Propositions \ref{prop:isomorphisms}, \ref{prop:tracechar_dom} and Lemma \ref{lem:decomp}, it suffices to prove the statements for $\OO=\RRdh$, which follows from \cite[Theorem 6.5]{Ro25}.
  \end{proof}

We remark that in the above two propositions the conditions on $[\OO]_{C^{j+1,\lambda}}$ can be omitted and in this case the implicit constants will depend on the domain.

\section{Fractional domains of the Laplacian on the half-space}\label{sec:frac_domain}
In this section, we establish properties of the Laplacian on the half-space that are required for Sections \ref{sec:calc_spec_dom} and \ref{sec:calc_dom}. There, we will transfer the $\Hinf$-calculus for the Laplacian from $\RRdh$ to domains using the perturbation results in Section \ref{sec:prelim_func_calc}. The aim of the present section is to recall the bounded $\Hinf$-calculus for the Laplacian on $\RRdh$ from \cite{LLRV24} and to characterise the relevant fractional domains and interpolation spaces. These characterisations are one of the key ingredients in the perturbation theorems in Section \ref{sec:calc_spec_dom}.\\

Throughout this section, the Dirichlet and Neumann Laplacian on $\RRdh$ will be defined as follows.
\begin{definition}\label{def:delRRdh}
Let $p\in(1,\infty)$, $k\in\NN_0$ and let $X$ be a $\UMD$ Banach space.
\begin{enumerate}[(i)]
    \item\label{it:def:delRRdh1} Let $\gam\in(-1,2p-1)\setminus\{p-1\}$. The \emph{Dirichlet Laplacian $\delDir$ on $W^{k,p}(\RRdh,w_{\gam+kp};X)$} is defined by
  \begin{equation*}
    \delDir u := \del u\quad \text{ with }\quad D(\delDir):=W^{k+2,p}_{\Dir}(\RRdh, w_{\gam+kp};X).
  \end{equation*}
    \item\label{it:def:delRRdh2} Let $\gam\in (-1,p-1)$ and $j\in\{0,1\}$. The \emph{Neumann Laplacian $\delNeu$ on $W^{k+j,p} (\RRdh, $ $ w_{\gam+kp}; X)$} is defined by
  \begin{equation*}
    \delNeu u := \del u\quad \text{ with }\quad D(\delNeu):=W^{k+j+2,p}_{\Neu}(\RRdh, w_{\gam+kp};X).
  \end{equation*}
  Note that equivalently we can write $\delNeu$ on $W^{k,p}(\RRdh, w_{\gam+(k-1)p};X)$ where $k\in\NN_0$ and $\gam\in (p-1,2p-1)$, or, $k\in\NN_1$ and $\gam\in(-1,p-1)$. This matches the notation in Theorem \ref{thm:introNeu}.
  \end{enumerate}
\end{definition} 

We recall from \cite{LLRV24} that these Laplace operators admit a bounded $\Hinf$-calculus.
\begin{theorem}[{\cite[Theorem 1.1 \& Remark 1.3(i)]{LLRV24}}]\label{thm:LLRVthm1.1Dir} 
Let $p\in(1,\infty)$, $k\in\NN_0$, $\gam\in(-1,2p-1)\setminus\{p-1\}$ and let $X$ be a $\UMD$ Banach space. Let $\delDir$ on $W^{k,p}(\RRdh, w_{\gam+kp};X)$ be as in Definition \ref{def:delRRdh}\ref{it:def:delRRdh1}. Then for all $\mu>0$ we have that
  \begin{enumerate}[(i)]
    \item $\mu-\delDir$ is sectorial of angle $\om(\mu-\delDir)=0$,
    \item $\mu-\delDir$ has a bounded $\Hinf$-calculus of angle $\om_{\Hinf}(\mu-\delDir)=0$. 
  \end{enumerate}
  Moreover, the statements hold for $\mu=0$ as well if $\gam+kp\in (-1,2p-1)$.
\end{theorem}

\begin{theorem}[{\cite[Theorem 1.2 \& Remark 1.3(i)]{LLRV24}}]\label{thm:LLRVthm1.2Neu} Let $p\in(1,\infty)$, $k\in\NN_0$, $\gam\in(-1,p-1)$, $j\in\{0,1\}$ and let $X$ be a $\UMD$ Banach space. Let $\delNeu$ on $W^{k+j,p}(\RRdh, w_{\gam+kp};X)$ be as in Definition \ref{def:delRRdh}\ref{it:def:delRRdh2}. 
  Then for all $\mu>0$ we have that
  \begin{enumerate}[(i)]
    \item $\mu-\delNeu$ is sectorial of angle $\om(\mu-\delNeu)=0$,
    \item $\mu-\delNeu$ has a bounded $\Hinf$-calculus of angle $\om_{\Hinf}(\mu-\delNeu)=0$. 
  \end{enumerate}
  Moreover, the statements hold for $\mu=0$ as well if $k=0$.
\end{theorem}
\begin{remark}\label{rem:domains}
  The domain $D(A)$ of an operator $A$ on a Banach space $Y$ is endowed with the graph norm $\|u\|_Y +\|Au\|_Y$ for $u\in D(A)$. It follows from Theorems \ref{thm:LLRVthm1.1Dir} and \ref{thm:LLRVthm1.2Neu} that the graph norm is equivalent to the norm of the domain in Definition \ref{def:delRRdh}. Under the conditions of Theorem \ref{thm:LLRVthm1.1Dir}, we have for the Dirichlet Laplacian that
  \begin{align*}
    \|u\|_{W^{k+2,p}(\RRdh, w_{\gam+kp};X)} & \eqsim_{p, k, \gam, \mu, X} \|u\|_{W^{k,p}(\RRdh, w_{\gam+kp};X)}+\|(\mu-\delDir)u\|_{W^{k,p}(\RRdh, w_{\gam+kp};X)}\\
    &\eqsim_{p, k, \gam, \mu, X} \|(\mu-\delDir)u\|_{W^{k,p}(\RRdh, w_{\gam+kp};X)}, \quad u\in W_\Dir^{k+2,p}(\RRdh, w_{\gam+kp};X),
  \end{align*}
  where the latter identity only holds for $\mu>0$. A similar norm equivalence holds for the Neumann Laplacian.
\end{remark}
To transfer the $\Hinf$-calculus for the Laplacian from $\RRdh$ to domains, we need to identify certain fractional domains and interpolation spaces. This will be done in Section \ref{sec:frac_dom_RRdh_Dir} and \ref{sec:frac_dom_RRdh_Neu} for the Dirichlet and Neumann Laplacian, respectively. We additionally define for $\gam\in(-1,\infty)\setminus\{jp-1:j\in\NN_1\}$ and $k\in\NN_0$ the following weighted Sobolev spaces with boundary conditions (cf. \cite[Section 6.3]{LV18})
\begin{align*}
   W^{k,p}_{\del, \Dir}(\RRdh, w_{\gam};X)&:=\Big\{u\in W^{k,p}(\RRdh, w_{\gam};X):  \Tr (\del ^j u)=0, \forall j<\tfrac{1}{2}\big(k-\tfrac{\gam+1}{p}\big)\Big\},\\
   W^{k,p}_{\del, \Neu}(\RRdh, w_{\gam};X)&:=\Big\{u\in W^{k,p}(\RRdh, w_{\gam};X):  \Tr (\del ^j \d_1u)=0, \forall j<\tfrac{1}{2}\big(k-1-\tfrac{\gam+1}{p}\big)\Big\}.
\end{align*}

\subsection{Fractional domains for the Dirichlet Laplacian}\label{sec:frac_dom_RRdh_Dir}
We begin with an elliptic regularity result for the shifted Dirichlet Laplacian on spaces with additional boundary conditions.
\begin{lemma}\label{lem:sect_Dir_ WdelDir}
  Let $p\in(1,\infty)$, $k\in\NN_0$, $\gam\in(-1,2p-1)\setminus\{p-1\}$, $\mu>0$ and let $X$ be a $\UMD$ Banach space. Then for all $f\in W^{k+1,p}_{\del, \Dir}(\RRdh, w_{\gam+kp};X)$ there exists a unique $u\in W^{k+3,p}_{\del,\Dir}(\RRdh, w_{\gam+kp};X)$ such that $\mu u -\del u = f$. Moreover, this solution satisfies
  \begin{equation*}
    \|u\|_{W^{k+3,p}(\RRdh, w_{\gam+kp};X)}\leq C\|f\|_{W^{k+1,p}(\RRdh, w_{\gam+kp};X)},
  \end{equation*}
  where the constant $C>0$ only depends on $p,k, \gam, \mu, d$ and $X$.
\end{lemma}
\begin{proof}
  \textit{Step 1: the case $\gam\in(-1,p-1)$.} Let $\gam\in (-1,p-1)$ and note that $$W^{k+1,p}_{\del, \Dir}(\RRdh, w_{\gam+kp};X)= W^{k+1,p}_{\Dir}(\RRdh, w_{\gam+kp};X) =W^{k+1,p}_{0}(\RRdh, w_{\gam+kp};X),$$
  which has $\Cc^{\infty}(\RRdh;X)$ as a dense subspace, see Proposition \ref{prop:tracechar_RRdh}. We claim that for $f\in \Cc^{\infty}(\RRdh;X)$ there exists a unique solution $u\in \SS(\RRdh;X)$ to $\mu u -\del u = f $ on $\RRdh$ that satisfies $u(0, \cdot)=(\del u)(0,\cdot)=0$. Indeed, by the proof of \cite[Lemma 5.3]{LLRV24} we obtain an odd function $\overline{u}\in \SS(\RR^d;X)$ which solves $\mu \overline{u}-\del \overline{u}= f_{\odd}\in \SS(\RRd;X)$ on $\RRd$. We recall from \cite{LLRV24} that $f_{\odd}(x)=\operatorname{sign}(x_1)f(|x_1|,\tilde{x})$ for $x\in \RRd$ is the odd extension of $f$ with respect to $x_1=0$. Since $\overline{u}$ is odd, it follows that $\del \overline{u}$ is odd as well. Then $u:=\overline{u}|_{\RRdh}\in \SS(\RRdh;X)$ is a solution to $\mu u -\del u = f $ on $\RRdh$ and satisfies $u(0, \cdot)=(\del u)(0,\cdot)=0$. The uniqueness follows from \cite[Corollary 4.3]{LV18}. This proves the claim.
  
  Let $f\in \Cc^{\infty}(\RRdh;X)$ and let $u\in \SS(\RRdh;X)$ be the solution to $\mu u -\del u =f$ as follows from the claim. In particular, we have that $\Tr (\d_1^2 u) = 0$. We define $v_0:= u$ and $v_j := \d_j u$ for $j\in\{1,\dots, d\}$. These functions satisfy the equations
  \begin{equation*}
    \begin{split}
    \mu v_0 - \del v_0 &= f\\
       \mu v_1 -\del v_1 &= \d_1 f \\
        \mu v_j -\del v_j &=\d_j f 
    \end{split}
    \qquad 
    \begin{split}
    v_0(0,\cdot)&=u(0,\cdot)=0,\\
       (\d_1v_1)(0,\cdot)&=(\d_1^2 u)(0,\cdot)=0,\\
       v_j(0,\cdot)&=0,\quad j\in\{2,\dots, d\}.
    \end{split}
  \end{equation*}
  Therefore, by \cite[Propositions 5.4 \& 5.6]{LLRV24} we have for $j\in \{0,\dots, d\}$ the estimates
  \begin{equation*}
    \|v_j\|_{W^{k+2,p}(\RRdh, w_{\gam+kp};X)}   \leq C \|f\|_{W^{k+1,p}(\RRdh, w_{\gam+kp};X)},
  \end{equation*}
  where the constant $C>0$ only depends on $p,k, \gam, \mu, d$ and $X$. This implies that
  \begin{align*}
    \|u\|_{W^{k+3,p}(\RRdh, w_{\gam+kp};X)} & \eqsim \sum_{j=0}^d\|v_j\|_{W^{k+2,p}(\RRdh, w_{\gam+kp};X)} 
     \lesssim  \|f\|_{W^{k+1,p}(\RRdh, w_{\gam+kp};X)},
  \end{align*}
  where the constant $C>0$
  only depends on $p,k, \gam, \mu, d$ and $X$. A density argument, similar to the proof of \cite[Proposition 5.4]{LLRV24}, yields the desired result for the case $\gam\in(-1,p-1)$. Note that the uniqueness of $u\in W^{k+3,p}_{\del, \Dir}(\RRdh, w_{\gam+kp};X)\hookrightarrow W^{k+2,p}_{\Dir}(\RRdh, w_{\gam+kp};X)$ follows from \cite[Proposition 5.4]{LLRV24}.
  
  \textit{Step 2: the case $\gam\in (p-1, 2p-1)$. }Note that for $\gam\in(p-1, 2p-1)$ we have 
$$W^{k+1,p}_{\del, \Dir}(\RRdh, w_{\gam+kp};X)= W^{k+1,p}(\RRdh, w_{\gam-p+(k+1)p};X).$$ Since $\gam-p\in (-1,p-1)$ and
$$ W^{k+3,p}_{\del,\Dir}(\RRdh, w_{\gam+kp};X)= W^{(k+1)+2,p}_{\Dir}(\RRdh, w_{\gam-p+(k+1)p};X),$$
the result follows from Theorem \ref{thm:LLRVthm1.1Dir} (see also \cite[Proposition 5.4]{LLRV24}).
  \end{proof}

We can now proceed with characterising fractional domains of the Dirichlet Laplacian.
\begin{proposition}\label{prop:domain_char_RRdh}
   Let $p\in(1,\infty)$, $k\in\NN_0$, $\gam\in(-1,2p-1)\setminus\{p-1\}$, $\mu>0$ and let $X$ be a $\UMD$ Banach space.  Let $\delDir$ on $ W^{k,p}(\RRdh, w_{\gam+kp};X)$ as in Definition \ref{def:delRRdh}. Then
   \begin{align*}
     D\big((\mu-\delDir)^{\half}\big)&= W^{k+1,p}_{\Dir}(\RRdh, w_{\gam+kp};X),\\ D\big((\mu-\delDir)^{\frac{3}{2}}\big)&=W^{k+3,p}_{\del, \Dir}(\RRdh, w_{\gam+kp};X).
   \end{align*}
\end{proposition}
\begin{proof} We write $A_\Dir:=\mu-\delDir$. 
For $\gam\in (-1,2p-1)\setminus\{p-1\}$ it holds that $A_\Dir$ has $\BIP$ by Theorem \ref{thm:LLRVthm1.1Dir}, so Propositions \ref{prop:frac_domain} and \ref{prop:complex_int_W_Dir} imply
\begin{equation*}
  D(A_\Dir^{\half}) = [W^{k,p}(\RRdh, w_{\gam+kp};X), W^{k+2,p}_{\Dir}(\RRdh, w_{\gam+kp};X)]_{\half} = W^{k+1,p}_{\Dir}(\RRdh, w_{\gam+kp};X).
\end{equation*}
By \cite[Theorem 15.2.5]{HNVW24} and the characterisation of $D(A_\Dir^{\half})$ we find
\begin{equation}\label{eq:DA32}
  \begin{aligned}
  D(A_\Dir^{\frac{3}{2}}) & =\{u\in D(A_\Dir): A_\Dir u \in D(A_\Dir^{\half})\} \\
  &= \{u\in W^{k+2,p}_{\Dir}(\RRdh, w_{\gam+kp};X): A_\Dir u\in W^{k+1,p}_{\Dir}(\RRdh, w_{\gam+kp};X) \}.
\end{aligned}
\end{equation}
It is straightforward to check that the embedding $W^{k+3,p}_{\del, \Dir}(\RRdh, w_{\gam+kp};X)\hookrightarrow D(A_\Dir^{\frac{3}{2}})$ holds. The converse embedding follows from \eqref{eq:DA32} and Lemma \ref{lem:sect_Dir_ WdelDir}.
\end{proof}

As a consequence of Proposition \ref{prop:domain_char_RRdh}, we can characterise the fractional domains as complex interpolation spaces as well.
\begin{corollary}\label{cor:frac_domain_Dir}
  Let $p\in(1,\infty)$, $k\in\NN_0$, $k_0,k_1\in \{0,1,2,3\}$, $\theta\in(0,1)$ and let $X$ be a $\UMD$ Banach space. For $\mu>0$ and $\delDir$ on $ W^{k,p}(\RRdh, w_{\gam+kp};X)$ be as in Definition \ref{def:delRRdh}.
  \begin{enumerate}[(i)]
    \item If $\gam\in(-1,p-1)$, then
      \begin{equation*}
    D\big((\mu-\delDir)^{\frac{(1-\theta)k_0+\theta k_1}{2}}\big)=\big[W^{k+k_0,p}_{\del, \Dir}(\RRdh,w_{\gam+kp};X), W^{k+k_1,p}_{\del,\Dir}(\RRdh,w_{\gam+kp};X)\big]_{\theta}.
  \end{equation*}
    \item If $\gam\in(p-1,2p-1)$, then
      \begin{equation*}
    D\big((\mu-\delDir)^{\frac{(1-\theta)k_0+\theta k_1}{2}}\big)=\big[W^{k+k_0,p}_{\Dir}(\RRdh,w_{\gam+kp};X), W^{k+k_1,p}_{\Dir}(\RRdh,w_{\gam+kp};X)\big]_{\theta}.
  \end{equation*}
  \end{enumerate}
\end{corollary}
\begin{proof}
   The fractional domains of the shifted Dirichlet Laplacian on $W^{k,p}(\R^d_+,w_{\gam+kp};X)$ form a complex interpolation scale by Proposition \ref{prop:frac_domain} and Theorem \ref{thm:LLRVthm1.1Dir}, so the statements are a direct consequence of Proposition \ref{prop:domain_char_RRdh}.
\end{proof}

We close this section about the Dirichlet Laplacian with a complex interpolation identification, which follows from reiteration and the work of {\v{S}}ne\u{\i}berg \cite{Sn73, Sn74} on the openness of the set of $\theta\in (0,1)$ for which a bounded operator $T\colon [X_0,X_1]_\theta \to [Y_0,Y_1]_\theta$ is invertible.
\begin{proposition}\label{prop:Sneiberg}
  Let $p\in (1,\infty)$, $k\in \NN_0$, $k_0\in \{0,1,2\}$, $\gam\in (p-1,2p-1)$ and let $X$ be a $\UMD$ Banach space. Then there exists an $\eps>0$ such that for all $\theta\in \big(0, \frac{2-k_0}{3-k_0}+\eps\big)$ we have
  \begin{align*}
    \big[W^{k+k_0,p}_{\Dir}(\RRdh&, w_{\gam+kp};X),W^{k+3,p}_{\Dir}(\RRdh, w_{\gam+kp};X) \big]_{\theta}\\&=\big[W^{k+k_0,p}_{0}(\RRdh, w_{\gam+kp};X), W^{k+3,p}_{0}(\RRdh, w_{\gam+kp};X)\big]_{\theta}.
  \end{align*}
\end{proposition}
\begin{proof} Let $\mu>0$ and define $A_\Dir:=\mu-\delDir$ on $W^{k,p}(\RRdh, w_{\gam+kp};X)$ as in Definition \ref{def:delRRdh}.
  First consider the case $k_0=0$ and $\theta = \frac23$, in which case we have by Corollary \ref{cor:frac_domain_Dir} and \cite[Proposition 6.2]{Ro25}
  \begin{equation}\label{eq:WDirW0interp}
      \begin{aligned}
     [W^{k,p}(\R^d_+&,w_{\gam+kp};X),W^{k+3,p}_{\Dir}(\R^d_+,w_{\gam+kp};X)]_{\frac23}=D(A_\Dir) \\&=
      W^{k+2,p}_{\Dir}(\R^d_+,w_{\gam+kp};X) = W_0^{k+2,p}(\R^d_+,w_{\gam+kp};X) \\&= [W^{k,p}(\R^d_+,w_{\gam+kp};X),W_0^{k+3,p}(\R^d_+,w_{\gam+kp};X)]_{\frac23}.
  \end{aligned}
  \end{equation}
   Next, for $\theta \in (0,\frac23)$, we set $\tilde{\theta} = \theta \cdot \frac32 \in (0,1)$. Then, by reiteration for the complex interpolation method (see \cite[Theorem 4.6.1]{BL76}) and \eqref{eq:WDirW0interp} we have
\begin{align*}
     [W^{k,p}(\R^d_+&,w_{\gam+kp};X),W^{k+3,p}_{\Dir}(\R^d_+,w_{\gam+kp};X)]_{\theta}\\ &= \bracb{W^{k,p}(\R^d_+,w_{\gam+kp};X) ,[W^{k,p}(\R^d_+,w_{\gam+kp};X),W^{k+3,p}_{\Dir}(\R^d_+,w_{\gam+kp};X)]_{\frac23}}_{\tilde{\theta}}\\
     &= \bracb{W^{k,p}(\R^d_+,w_{\gam+kp};X) ,[W^{k,p}(\R^d_+,w_{\gam+kp};X),W^{k+3,p}_{0}(\R^d_+,w_{\gam+kp};X)]_{\frac23}}_{\tilde{\theta}} \\&= [W^{k,p}(\R^d_+,w_{\gam+kp};X),W_0^{k+3,p}(\R^d_+,w_{\gam+kp};X)]_{\theta}.
\end{align*}
Note that the identity mapping is bounded on $W^{k,p}(\R^d_+,w_{\gam+kp};X)$ and $$\id:W^{k+3,p}_{0}(\R^d_+,w_{\gam+kp};X)\to W_{\Dir}^{k+3,p}(\R^d_+,w_{\gam+kp};X)\quad \text{is bounded.}$$
Moreover, we have proved that it is invertible as a mapping
\begin{align*}
\id:[W^{k,p}(\R^d_+&,w_{\gam+kp};X),W^{k+3,p}_{0}(\R^d_+,w_{\gam+kp};X)]_{\theta}\\
 &\to [W^{k,p}(\R^d_+,w_{\gam+kp};X),W^{k+3,p}_{\Dir}(\R^d_+,w_{\gam+kp};X)]_{\theta}
\end{align*}
for $\theta\in (0, \frac23]$. Since the collection of $\theta \in (0,1)$ for which this mapping is invertible is open (see \cite[Theorem 1.3.24]{Eg15}), the proposition in the case $k_0=0$ follows.

Finally, for $k_0 \in \cbrace{1,2}$, let $\varepsilon>0$ be such that the proposition holds for $k_0=0$ and fix $\theta\in\big(0,\frac{2-k_0}{3-k_0}+\varepsilon\big)$. Then we have
\begin{align*}
  (1-\theta)\tfrac{k_0}{3}+ \theta = \tfrac{k_0}{3} + (\tfrac{3-k_0}{3})\theta <\tfrac{k_0}{3} + \tfrac{2-k_0}{3} +\varepsilon =\tfrac23 +\varepsilon.
\end{align*}
Therefore, using \cite[Proposition 6.2]{Ro25}, reiteration for the complex interpolation method and the case $k_0=0$, we obtain
\begin{align*}
  [W^{k+k_0,p}_0(\R^d_+&,w_{\gam+kp};X),W_0^{k+3,p}(\R^d_+,w_{\gam+kp};X)]_{\theta} \\&= \bracb{[W^{k,p}_0(\R^d_+,w_{\gam+kp};X),W_0^{k+3,p}(\R^d_+,w_{\gam+kp};X)]_{\frac{k_0}{3}},W_0^{k+3,p}(\R^d_+,w_{\gam+kp};X)}
  _{\theta}\\
  &= [W^{k,p}_0(\R^d_+,w_{\gam+kp};X),W_0^{k+3,p}(\R^d_+,w_{\gam+kp};X)]_{(1-\theta)\frac{k_0}{3}+ \theta}\\
  &= [W^{k,p}_\Dir(\R^d_+,w_{\gam+kp};X),W^{k+3,p}_\Dir(\R^d_+,w_{\gam+kp};X)]_{(1-\theta)\frac{k_0}{3}+ \theta}.
\end{align*}
Using Corollary \ref{cor:frac_domain_Dir} two more times, we have
\begin{align*}
[W^{k,p}_{\Dir}(\R^d_+,w_{\gam+kp};X)&,W^{k+3,p}_\Dir(\R^d_+,w_{\gam+kp};X)]_{(1-\theta)\frac{k_0}{3}+ \theta} \\ &= D(A_\Dir^{(1-\theta)\frac{k_0}{2}+\frac32\theta})\\
&= [W_{\Dir}^{k+k_0,p}(\R^d_+,w_{\gam+kp};X),W^{k+3,p}_{\Dir}(\R^d_+,w_{\gam+kp};X)]_{\theta},
\end{align*}
proving the proposition.
\end{proof}

\begin{remark}\label{rem:interpolation}
  We conjecture that, e.g., in the case $k=k_0=0$, there is actually the equality of complex interpolation spaces
\begin{equation}\label{eq:rmk:lem:pert_half-space_g-sectorial:int}
[L^p(\R^d_+,w_\gam;X),W^{3,p}_{\Dir}(\R^d_+,w_\gam;X)]_{\theta} = [L^p(\R^d_+,w_\gam;X),W_0^{3,p}(\R^d_+,w_\gam;X)]_{\theta}
\end{equation}
for all $\theta \in \big(0, \frac{1}{3}(1+\frac{\gam+1}{p})\big)$, which is suggested by results on interpolation with boundary conditions as studied in \cite{LMV17, Ro25}.
However, at the moment, the case $\gam \in (p-1,2p-1)$ of \eqref{eq:rmk:lem:pert_half-space_g-sectorial:int} for the parameter range $ \theta \in\big(\tfrac{2}{3} +\varepsilon, \tfrac{1}{3}\hab{1+\tfrac{\gam+1}{p}}\big)$ is an interesting open problem that seems to require a novel approach to interpolation with boundary conditions.
\end{remark}

\subsection{Fractional domains for the Neumann Laplacian}\label{sec:frac_dom_RRdh_Neu}
Similar to the Dirichlet Laplacian above, we now characterise fractional domains for the Neumann Laplacian. The proofs are similar to those in Section \ref{sec:frac_dom_RRdh_Dir}, but for the convenience of the reader, we provide the details.
\begin{lemma}\label{lem:sect_Dir_ WdelNeu}
  Let $p\in(1,\infty)$, $k\in\NN_0\cup\{-1\}$, $\gam\in(-1,2p-1)\setminus\{p-1\}$ such that $\gam+kp>-1$, $\mu>0$ and let $X$ be a $\UMD$ Banach space. Then for all $f\in W^{k+2,p}_{\del, \Neu}(\RRdh, w_{\gam+kp};X)$ there exists a unique $u\in W^{k+4,p}_{\del,\Neu}(\RRdh, w_{\gam+kp};X)$ such that $\mu u -\delNeu u = f$. Moreover, this solution satisfies
  \begin{equation*}
    \|u\|_{W^{k+4,p}(\RRdh, w_{\gam+kp};X)}\leq C\|f\|_{W^{k+2,p}(\RRdh, w_{\gam+kp};X)},
  \end{equation*}
  where the constant $C>0$ only depends on $p,k, \gam, \mu, d$ and $X$.
\end{lemma}

\begin{proof}
  \textit{Step 1: the case $\gam\in(p-1,2p-1)$ and $k\geq -1$.} Note that for $\gam\in(p-1, 2p-1)$ we have 
$$W^{k+2,p}_{\del, \Neu}(\RRdh, w_{\gam+kp};X)= W^{k+2,p}(\RRdh, w_{\gam-p+(k+1)p};X).$$ Since $\gam-p\in (-1,p-1)$ and
$$ W^{k+4,p}_{\del,\Neu}(\RRdh, w_{\gam+kp};X)= W^{(k+2)+2,p}_{\Neu}(\RRdh, w_{\gam-p+(k+1)p};X),$$
the result follows from Theorem \ref{thm:LLRVthm1.2Neu} (see also \cite[Proposition 5.6]{LLRV24}).

  \textit{Step 2: the case $\gam\in(-1,p-1)$ and $k\geq 0$.} Note that for $\gam\in(-1,p-1)$ we have 
  \begin{equation*}
    W^{k+2,p}_{\del, \Neu}(\RRdh, w_{\gam+kp};X) = W^{k+2,p}_{\Neu}(\RRdh, w_{\gam+kp};X),
  \end{equation*}
  which has 
  \begin{equation*}
      C^{\infty}_{{\rm c},1}(\overline{\RRdh};X):=\{f\in \Cc^{\infty}(\overline{\RRdh};X): \d_1 f\in \Cc^{\infty}(\RRdh;X)\}
  \end{equation*}
  as a dense subspace, see \cite[Proposition 4.9]{Ro25}. For $f\in C^{\infty}_{{\rm c},1}(\overline{\RRdh};X)$ there exists a unique solution $u\in \SS(\RRdh;X)$ to $\mu u -\delNeu u =f$ on $\RRdh$ that satisfies $(\d_1 u)(0,\cdot)=(\del \d_1 u)(0,\cdot)=0$. This can be proved similarly as in Lemma \ref{lem:sect_Dir_ WdelDir} now using an even extension (cf. \cite[Lemma 5.5]{LLRV24}).
  
   Take $f\in C^{\infty}_{{\rm c},1}(\overline{\RRdh};X)$ and let $u\in \SS(\RRdh;X)$ be the solution to $\mu u -\delNeu u =f$ as above. We define $v_0:= u$ and $v_j:=\d_j u$ for $j\in\{1,\dots, d\}$. These functions satisfy the estimates
     \begin{equation*}
    \begin{split}
    \mu v_0 - \del v_0 &= f\\
       \mu v_1 -\del v_1 &= \d_1 f \\
        \mu v_j -\del v_j &=\d_j f 
    \end{split}
    \qquad 
    \begin{split}
    (\d_1v_0)(0,\cdot)&=0,\\
       v_1(0,\cdot)&=0,\\
       (\d_1v_j)(0,\cdot)&=0,\quad j\in\{2,\dots, d\}.
    \end{split}
  \end{equation*}
  If $j=1$, then by Lemma \ref{lem:sect_Dir_ WdelDir} (using that $(\d_1 f)|_{\d\RRdh}=0$) we have the estimate
  \begin{equation}\label{eq:sect_est_Neu1}
    \|v_1\|_{W^{k+3,p}(\RRdh, w_{\gam+kp};X)}\leq C\|\d_1 f\|_{W^{k+1,p}(\RRdh, w_{\gam+kp};X)}.
  \end{equation}
  If $j\in \{2,\dots, d\}$, then applying Step 1 with $k-1$ and $\gam+p\in (p-1,2p-1)$, yields
  \begin{equation}\label{eq:sect_est_Neu2}
  \begin{aligned}
  \|v_j\|_{W^{k+3,p}(\RRdh, w_{\gam+kp};X)}&= \|v_j\|_{W^{(k-1)+4,p}(\RRdh, w_{\gam+p+(k-1)p};X)}\\
  &\leq C\|\d_j f\|_{W^{k+1,p}(\RRdh, w_{\gam+kp};X)},
  \end{aligned}
  \end{equation} 
  and similarly for $j=0$ we obtain
  \begin{equation}\label{eq:sect_est_Neu3}
    \|v_j\|_{W^{k+3,p}(\RRdh, w_{\gam+kp};X)}\leq  C \|f\|_{W^{k+1,p}(\RRdh, w_{\gam+kp};X)}.
  \end{equation}
  The estimates \eqref{eq:sect_est_Neu1}, \eqref{eq:sect_est_Neu2} and \eqref{eq:sect_est_Neu3} imply that
  \begin{align*}
    \|u\|_{W^{k+4,p}(\RRdh, w_{\gam+kp};X)} & \eqsim \sum_{j=0}^d\|v_j\|_{W^{k+3,p}(\RRdh, w_{\gam+kp};X)} \\
     & \lesssim\|f\|_{W^{k+1,p}(\RRdh, w_{\gam+kp};X)} + \sum_{j=1}^d \|\d_j f\|_{W^{k+1,p}(\RRdh, w_{\gam+kp};X)} \\ &\lesssim \|f\|_{W^{k+2,p}(\RRdh, w_{\gam+kp};X)},
  \end{align*}
  where the constant only depends on $p,k, \gam, \mu, d$ and $X$.
  A density argument, similar to the proof of \cite[Proposition 5.4]{LLRV24}, yields the result. Note that the uniqueness of $u\in W^{k+4,p}_{\del, \Neu}(\RRdh, w_{\gam+kp};X)\hookrightarrow W^{k+3,p}_{\Neu}(\RRdh, w_{\gam+kp};X)$ follows from \cite[Proposition 5.6]{LLRV24}.
\end{proof}

We continue with the characterisation of fractional domains of the Neumann Laplacian.
\begin{proposition}\label{prop:domain_char_RRdh_Neu}
   Let $p\in(1,\infty)$, $k\in\NN_0\cup\{-1\}$, $\gam\in(-1,2p-1)\setminus\{p-1\}$ such that $\gam+kp>-1$, $\mu>0$ and let $X$ be a $\UMD$ Banach space.  Let $\delNeu$ on $ W^{k+1,p}(\RRdh, w_{\gam+kp};X)$ as in Definition \ref{def:delRRdh}. Then
   \begin{align*}
     D\big((\mu-\delNeu)^{\half}\big)&= W^{k+2,p}_{\Neu}(\RRdh, w_{\gam+kp};X),\\ D\big((\mu-\delNeu)^{\frac{3}{2}}\big)&=W^{k+4,p}_{\del, \Neu}(\RRdh, w_{\gam+kp};X).
   \end{align*}
\end{proposition}
\begin{proof} We write $A_\Neu:=\mu-\delNeu$. 
For $\gam\in (-1,2p-1)\setminus\{p-1\}$ it holds that $A_\Neu$ has $\BIP$ by Theorem \ref{thm:LLRVthm1.2Neu}, so Propositions \ref{prop:frac_domain} and \ref{prop:complex_int_W_Neu} imply
\begin{equation*}
  D(A_\Neu^{\half}) = [W^{k+1,p}(\RRdh, w_{\gam+kp};X), W^{k+3,p}_{\Neu}(\RRdh, w_{\gam+kp};X)]_{\half} = W^{k+2,p}_{\Neu}(\RRdh, w_{\gam+kp};X).
\end{equation*}
By \cite[Theorem 15.2.5]{HNVW24} and the characterisation of $D(A_\Neu^{\half})$ we find
\begin{equation*}
  \begin{aligned}
  D(A_\Neu^{\frac{3}{2}}) & =\{u\in D(A_\Neu): A_\Neu u \in D(A_\Neu^{\half})\} \\
  &= \{u\in W^{k+3,p}_{\Neu}(\RRdh, w_{\gam+kp};X): A_\Neu u\in W^{k+2,p}_{\Neu}(\RRdh, w_{\gam+kp};X) \}.
\end{aligned}
\end{equation*}
From this, the embedding $W^{k+4,p}_{\del, \Neu}(\RRdh, w_{\gam+kp};X)\hookrightarrow D(A_\Neu^{\frac{3}{2}})$ is straightforward and the converse embedding follows from Lemma \ref{lem:sect_Dir_ WdelNeu}.
\end{proof}

In contrast to the Dirichlet case, we do not need a version of Proposition \ref{prop:Sneiberg} for the Neumann Laplacian. This is simply due to the fact that we cannot consider the Neumann Laplacian on $W^{k,p}(\RRdh, w_{\gam+kp};X)$ with $\gam>p-1$, see Theorem \ref{thm:LLRVthm1.2Neu}.

\section{Functional calculus for the Laplacian on special domains}\label{sec:calc_spec_dom}
To derive the $\Hinf$-calculus for the Dirichlet and Neumann Laplacian on bounded domains, we will proceed in two steps:
\begin{enumerate}
  \item\label{it:Step1} Use the $\Hinf$-calculus for the Laplacian on the half-space (Theorems \ref{thm:LLRVthm1.1Dir} and \ref{thm:LLRVthm1.2Neu}) and known perturbation theorems  for the $\Hinf$-calculus (Section \ref{sec:prelim_func_calc}) to obtain the $\Hinf$-calculus for the Laplacian on special domains of the form $\OO:= \{x\in \RR^d: x_1>h(\tilde{x})\}$ for a compactly supported function $h$ on $\RR^{d-1}$ (see Definition \ref{def:domains}).
  \item\label{it:Step2} Perform a localisation procedure to transfer the $\Hinf$-calculus for the Laplacian on special domains to bounded domains.
\end{enumerate}
In this section, we will perform Step \ref{it:Step1}, while Step \ref{it:Step2} is postponed to Section \ref{sec:calc_dom}. 
While localisation procedures are standard in the literature (see, e.g., \cite{DHP03, Ev10, KrBook08}), the low regularity of the domains considered here leads to perturbation terms that, in some cases, are of the same order as the Laplacian. Therefore, we employ a localisation procedure that is different from the standard procedure as in the aforementioned literature. This leads to a far-reaching generalisation of the results in \cite[Theorem 6.1]{LV18} where exclusively bounded $C^2$-domains are considered for only the $L^p$-case (i.e., $k=0$).\\

We begin by defining the Laplacian on special domains. Recall that weighted Sobolev spaces on special domains with vanishing boundary conditions are defined in Definition \ref{def:spaces_special}.
\begin{definition}\label{def:delDirOspecial}
  Let $p\in(1,\infty)$, $k\in\NN_0$, $\lambda\in[0,1]$ and let $X$ be a $\UMD$ Banach space. 
   \begin{enumerate}[(i)]
    \item\label{it:def:delO1special} Let $\gam\in ((1-\lambda)p-1, 2p-1)\setminus\{p-1\}$ and $\OO$ a special $\Cc^{1,\lambda}$-domain with $[\OO]_{C^{1,\lambda}}\leq 1$.     
    The \emph{Dirichlet Laplacian $\delDir$ on $W^{k,p}(\OO,w^{\d\OO}_{\gam+kp};X)$} is defined by
  \begin{equation*}
    \delDir u := \del u\quad \text{ with }\quad D(\delDir):=W^{k+2,p}_{\Dir}(\OO, w^{\d\OO}_{\gam+kp};X).
  \end{equation*}
    \item\label{it:def:delO2special} Let $\gam\in ((1-\lambda)p-1, p-1)$, $j\in\{0,1\}$ and $\OO$ a special $\Cc^{j+1,\lambda}$-domain with $[\OO]_{C^{j+1,\lambda}}\leq \Lambda$, where $\Lambda\in(0,1)$ is as in Lemma \ref{lem:loc_normal}. The \emph{Neumann Laplacian $\delNeu$ on $W^{k+j,p}(\OO,w^{\d\OO}_{\gam+kp};X)$} is defined by
  \begin{equation*}
    \delNeu u := \del u\quad \text{ with }\quad D(\delNeu):=W^{k+j+2,p}_{\Neu}(\OO, w^{\d\OO}_{\gam+kp};X).
  \end{equation*}
  Moreover, the Dirichlet and Neumann Laplacian on $\RRdh$ as in Definition \ref{def:delRRdh} will be denoted by $\delDir^{\RRdh}$ and $\delNeu^{\RRdh}$, respectively.
  \end{enumerate}
\end{definition}

The main results from this section on the $\Hinf$-calculus for the Laplacian on special domains are summarised in the following two theorems.

\begin{theorem}[$\Hinf$-calculus for $\mu-\delDir$ on special domains]\label{thm:Dirichlet_Laplacian_special}
Let $p\in(1,\infty)$, $k\in\NN_0$, $\lambda\in[0,1]$, $\gam\in((1-\lambda)p-1, 2p-1)\setminus\{p-1\}$, $\mu>0$ and let $X$ be a $\UMD$ Banach space. Moreover, assume that $\OO$ is a special $\Cc^{1,\lambda}$-domain. Then there exists a $\delta\in(0,1)$ such that if $[\OO]_{C^{1,\lambda}}<\delta$, then $\mu-\DD$ on $W^{k,p}(\Dom,w^{\BDom}_{\gamma+kp};X)$ as in Definition \ref{def:delDirOspecial} has a bounded $H^\infty$-calculus with $\omega_{H^\infty}(\mu-\DD) =0$.
\end{theorem}

\begin{theorem}[$\Hinf$-calculus for $\mu-\delNeu$ on special domains]\label{thm:Neumann_Laplacian_special}
  Let $p\in(1,\infty)$, $k\in\NN_0$, $\lambda\in(0,1]$, $\gam\in ((1-\lambda)p-1, p-1)$, $j\in \{0,1\}$, $\mu>0$ and let $X$ be a $\UMD$ Banach space. Moreover, assume that $\OO$ is a special $\Cc^{j+1,\lambda}$-domain. Then there exists a $\delta\in(0,1)$  such that if $[\OO]_{C^{j+1,\lambda}}<\delta$, then $\mu-\delNeu$ on $W^{k+j,p}(\Dom,w^{\BDom}_{\gamma+kp};X)$ as in Definition \ref{def:delDirOspecial}  has a bounded $H^\infty$-calculus with $\omega_{H^\infty}(\mu-\delNeu) =0$. 
\end{theorem}
\begin{remark}Similar to Theorems \ref{thm:LLRVthm1.1Dir} and \ref{thm:LLRVthm1.2Neu}, we expect that Theorems \ref{thm:Dirichlet_Laplacian_special} and \ref{thm:Neumann_Laplacian_special} also hold for $\mu=0$ if $\gam+kp$ is small. We will not consider this minor improvement of the theorems here, since in Section \ref{sec:calc_dom} we consider bounded domains and use properties of the spectrum to obtain the $\Hinf$-calculus with $\mu=0$.
\end{remark}

The proofs of Theorems \ref{thm:Dirichlet_Laplacian_special} and \ref{thm:Neumann_Laplacian_special} are given in Section \ref{sec:proof_special} after having established some preliminary estimates in Section \ref{sec:prelim_pert}.

\subsection{Preliminary estimates}\label{sec:prelim_pert}
In the proofs of Theorems \ref{thm:Dirichlet_Laplacian_special} and \ref{thm:Neumann_Laplacian_special}, we derive the $\Hinf$-calculus on special domains by perturbing the corresponding calculus for the Laplacian on the half-space. To relate the Laplacian on special domains and the half-space, we use the diffeomorphisms $\Phi$ and $\Psi$ from Lemmas \ref{lem:localization_weighted_blow-up} and \ref{lem:loc_normal} for the Dirichlet and Neumann Laplacian, respectively.
The diffeomorphism $\Phi$ is easier to deal with, but it does not suffice for the Neumann Laplacian since it does not preserve the direction of the normal vector at the boundary, see Appendix \ref{sec:appendix_lemma}. \\

First, consider the case of Dirichlet boundary conditions. Let $\OO$ be a special $\Cc^1$-domain and recall that $\Phi_* f=f\circ \Phi^{-1}$ for $f\in L^1_{\loc}(\OO;X)$. Define $\Delta^{\Phi}: W^{2,1}_{\loc}(\RRdh;X)\to L^1_{\loc}(\RRdh; X)$ by 
\begin{equation*}
\Delta^{\Phi} := \Phi_* \circ \Delta\circ (\Phi^{-1})_{*}.
\end{equation*}
Making use of the explicit form of the diffeomorphism $\Phi(x)=(x_1-h_1(x), \tilde{x})$ for $x\in \OO$ (see Lemma \ref{lem:localization_weighted_blow-up}), an elementary computation involving the chain rule shows that $\del^\Phi = \del +B$, where the perturbation $B$ is given by
\begin{equation}\label{eq:pertBDIR}
\begin{aligned}
  B =  - 2((\nabla h_1) \circ \Phi^{-1}) \cdot \nabla \partial_1 + |(\nabla h_1) \circ \Phi^{-1}|^2\,\partial_1^2 - ((\Delta h_1) \circ \Phi^{-1})\partial_1.
\end{aligned}
\end{equation}

Note that the first two perturbation terms in \eqref{eq:pertBDIR} are second-order differential operators since $(\grad h_1)\circ\Phi^{-1}$ is bounded on $\RRdh$ if $\OO$ is a special $\Cc^1$-domain, see Lemma \ref{lem:localization_weighted_blow-up}. The order of the latter perturbation term in \eqref{eq:pertBDIR} depends on the smoothness of the domain.
\begin{itemize}
  \item If $\OO$ is a special $\Cc^2$-domain, then $(\del h_1)\circ \Phi^{-1}$ is bounded on $\RRdh$ and  thus the last term in \eqref{eq:pertBDIR} is a lower-order perturbation term.
  \item If $\OO$ is a special $\Cc^1$-domain, then $(\del h_1)(\Phi^{-1}(y))$ blows up like $y_1^{-1}$ in the neighbourhood of $y_1=0$, see Lemma \ref{lem:localization_weighted_blow-up}. Therefore, estimating, say, the $L^p(\RRdh, w_{\gam})$-norm of $((\del h_1)\circ \Phi^{-1})\d_1$ gives that 
  the weight exponent effectively decreases. However, this loss can be compensated by applying Hardy's inequality, which allows us to recover the original weight $w_{\gam}$. In this way, we also obtain an additional derivative from Hardy's inequality, meaning that all three perturbation terms in \eqref{eq:pertBDIR} have the same order. 
\end{itemize}
This demonstrates that if the smoothness of the domain is low, then the last perturbation term in \eqref{eq:pertBDIR} is more difficult to deal with.\\

Similarly, for Neumann boundary conditions, let $\Psi$ be the diffeomorphism from Lemma \ref{lem:loc_normal} and define $\del^{\Psi}: W^{2,1}_\loc(\RRdh;X)\to L^1_\loc(\RRdh;X)$ by 
\begin{equation*}
    \del^\Psi:= \Psi_* \circ \del \circ (\Psi^{-1})_*.
\end{equation*}
Recall that $\Psi(x)=(x_1 - h_1(x), \tilde{x}-\tilde{h}_1(\tilde{x}))$ for $x\in \OO$ and we write $h = (h_1, \tilde{h}_1) =: (h_1, h_2,\dots, h_d)$. Another tedious, but elementary calculation with the chain rule shows that $\del^{\Psi} = \del + B$ with
\begin{align*}
        B=\sum_{i,j=1}^d \big[\big(-H_{i,j}-H_{j,i}+(HH^\top)_{i,j}\big)\circ \Psi^{-1}\big]\d_{ij}^2 - \sum_{j=1}^d \big[\del h_j\circ \Psi^{-1}\big]\d_j,
    \end{align*}
    where $H:=D h$ is the Jacobi matrix of $h$. Compared to the case of Dirichlet boundary conditions in \eqref{eq:pertBDIR}, we have more perturbation terms, since the diffeomorphism $\Psi$ is more involved. To simplify the estimates later on, we simply note that $B$ is a linear combination of terms of the following forms
\begin{equation}\label{eq:pertBNEU}
    \begin{aligned}
    \big[(\d^{\nu_1} h_{i_1})\circ \Psi^{-1}\big]\d^{\mu}, \quad &&|\nu_1|=1, \,|\mu|=2,\\
    \big[(\d^{\nu_1} h_{i_1})\circ \Psi^{-1}\big]\big[(\d^{\nu_2} h_{i_2})\circ \Psi^{-1}\big]\d^{\mu},\quad &&|\nu_1|=|\nu_2|=1,\, |\mu|=2,\\
    \big[(\d^{\nu_1} h_{i_1})\circ \Psi^{-1}\big]\d^{\mu},\quad&& |\nu_1|=2,\, |\mu|=1,
\end{aligned}
\end{equation}
where $i_1,i_2\in\{1,\dots, d\}$.
Note that the perturbation terms in \eqref{eq:pertBDIR} can also be written as in \eqref{eq:pertBNEU} with $i_1=i_2=1$ and $\Phi$ instead of $\Psi$.\\

In the following lemmas we provide precise estimates for the perturbation terms. We note that these estimates work for both diffeomorphisms $\Phi$ and $\Psi$ since they have the same regularity properties. Throughout the rest of this subsection, we let $\mc{O}$ be a special $C^1_c$-domain and take the following standing assumptions:
\begin{itemize}
    \item In the case of $\Phi$: $[\mc{O}]_{C^1}\leq 1$ and take $h_1$ as in Lemma \ref{lem:localization_weighted_blow-up}.
    \item In the case of $\Psi$: $[\mc{O}]_{C^1}\leq \Lambda$, where $\Lambda \in (0,1)$ as in Lemma \ref{lem:loc_normal}, take $h_1$ and $\tilde{h}_1$ as Lemma \ref{lem:loc_normal} and set $(h_1, h_2,\dots, h_d):=(h_1, \tilde{h}_1)$. 
\end{itemize}
We start with the estimates for the perturbation terms with $|\mu|=2$.

\begin{lemma}[Estimates for $|\mu|=2$]\label{lem:estB1B2}
  Let $p\in(1,\infty)$, $k\in\NN_0$ and let $X$ be a Banach space. Take $|\mu|=2$, $|\nu_1|=|\nu_2|=1$ and $i_1,i_2\in\{1,\dots, d\}$. Let $\Upsilon\in\{\Phi, \Psi\}$ and define
\begin{align*}
    P_1&:=\big[(\d^{\nu_1} h_{i_1})\circ \Upsilon^{-1}\big]\d^{\mu},\\
    P_2&:=\big[(\d^{\nu_1} h_{i_1})\circ \Upsilon^{-1}\big]\big[(\d^{\nu_2} h_{i_2})\circ \Upsilon^{-1}\big]\d^{\mu}.
\end{align*}
  \begin{enumerate}[(i)]
    \item\label{it:est1} If $\gam\in (p-1,2p-1)$, then for $n\in\{0,1\}$ and $u\in W^{k+2+n,p}(\RRdh, w_{\gam+kp};X)$ it holds that
        \begin{equation*}
          \|P_mu\|_{W^{k+n,p}(\RRdh, w_{\gam+kp};X)}\leq C\cdot[\OO]_{C^1}\cdot\|u\|_{W^{k+2+n,p}(\RRdh, w_{\gam+kp};X)}, \qquad m\in\{1,2\}.
        \end{equation*}
    \item\label{it:est2} If $\lambda\in (0,1]$, $\gam\in ((1-\lambda)p-1, p-1)$, $j\in\{0,1\}$ and $\OO$ is a special $\Cc^{j+1, \lambda}$-domain, then for $n\in\{0,1\}$ and $u\in W^{k+2+j+n, p}(\RRdh, w_{\gam+kp};X)$ it holds that 
        \begin{equation*}
          \qquad\|P_m u\|_{W^{k+j+n,p}(\RRdh, w_{\gam+kp};X)}\leq C\cdot[\OO]_{C^{j+1,\lambda}}\cdot\|u\|_{W^{k+2+j+n,p}(\RRdh, w_{\gam+kp};X)}, \qquad m\in\{1,2\}.
        \end{equation*}
  \end{enumerate}
  In all cases, the constant $C>0$ only depends on $p,k,\gam, \lambda, d $ and $X$.
\end{lemma}

\begin{proof}
  For notational convenience we write $W^{k,p}(w_{\gam}):=W^{k,p}(\RRdh,w_{\gam};X)$. 
  
  \textit{Step 1: preparations.}
We prove the estimates for $P_1$, where, from now on, we omit the subscripts from $\nu$ and $i$. The estimates for $P_2$ are derived in a similar way.

For $\alpha\in \NN_0^d$ and some regular enough $u$ we obtain with the product rule that 
  \begin{equation}\label{eq:est_prodrule}
\begin{aligned}
  \|\d^{\alpha}[((\d^{\nu}h_i)\circ & \Upsilon^{-1}) \d^{\mu} u]\|_{L^p(w_{\gam+kp})}\\&\lesssim \sum_{\beta\leq \alpha}\big\| [\d^\beta ((\d^\nu h_i)\circ \Upsilon^{-1})][\d^{\alpha-\beta}\d^\mu u]\big\|_{L^p(w_{\gam+kp})}.
\end{aligned}
\end{equation}
In the case that $|\alpha|,|\beta|\geq 1$ and $y\in \RRdh$, the multivariate Fa\`a di Bruno's formula \cite[Theorem 2.1]{CS96} implies
\begin{equation}\label{eq:Faa}
  |\d_y^\beta(\d^\nu h_i)(\Upsilon^{-1}(y))|\lesssim \sum_{1\leq |\delta|\leq |\beta|}|(\d^\delta\d^\nu h_i)(\Upsilon^{-1}(y))|\sum_{s=1}^{|\beta|}\sum_{p_s(\beta,\delta)}\prod_{m=1}^s|\d^{\vec{\ell}_m}\Upsilon^{-1}(y)|^{\vec{k}_m},
\end{equation}
where the sets $p_s(\beta,\delta)$ are contained in
\begin{align}\label{eq:setp_s}
\!\!\!\! \Big\{ (\vec{k}_1,\ldots,\vec{k}_s;\vec{\ell}_1,\ldots,\vec{\ell}_s) \in (\N^d_0 \setminus \{0\})^{s} \times (\N^d_0 \setminus \{0\})^{s}
 : \sum_{m=1}^s|\vec{k}_m| = |\delta|, \sum_{m=1}^s|\vec{k}_m||\vec{\ell}_m| = |\beta| \Big\}.
\end{align}
By Lemma \ref{lem:localization_weighted_blow-up}\ref{it:lem:localization_weighted_blow-up;dist_preserving}+\ref{it:lem:localization_weighted_blow-up;est}  for $\Upsilon=\Phi$ and Lemma \ref{lem:loc_normal}\ref{it:lem:loc_normal2}+\ref{it:lem:loc_normal5} for $\Upsilon=\Psi$, we have the estimate
\begin{equation}\label{eq:grad_del h_1_grad}
  |(\d^{\delta}\d^\nu  h_i)(\Upsilon^{-1}(y))| \lesssim \frac{[\OO]_{C^{j+1,\lambda}}}{\mrm{dist}(\Upsilon^{-1}(y),\BDom)^{( |\delta|-j-\lambda)_+}}
\lesssim \frac{[\OO]_{C^{j+1,\lambda}}}{y_1^{( |\delta|-j-\lambda)_+}},
\end{equation}
for all $\lambda\in[0,1]$, $j\in\{0,1\}$, $\delta\in \NN_0^d$, $|\nu|=1$ and $y \in \R^d_+$. 
Moreover, by the same lemmas we also have the (non-optimal) estimate
\begin{equation}\label{eq:grad_del h_2_grad}
  |\d^{\vec{\ell}}\Upsilon^{-1}(y)|\lesssim \frac{[\OO]_{C^{j+1}}}{y_1^{(|\vec{\ell}|-j-1)_+}},
\end{equation}
for all $j\in\{0,1\}$, $\vec{\ell}\in \NN_0^d$ and $y \in \R^d_+$.

\textit{Step 2: proof of \ref{it:est1}}. Let $\gam\in(p-1, 2p-1)$, $n\in\{0,1\}$ and $\OO$ a special $\Cc^1$-domain. To prove \ref{it:est1} we need to consider \eqref{eq:est_prodrule} with $|\alpha|\leq k+n$. If $\beta=0$ in \eqref{eq:est_prodrule}, then it follows from \eqref{eq:grad_del h_1_grad} that 
\begin{equation*}\label{eq:est_beta=0}
  \|((\d^\nu h_i)\circ \Upsilon^{-1})(\d^\alpha\d^\mu u)\|_{L^p(w_{\gam+kp})}\lesssim [\OO]_{C^{1}}\|u\|_{W^{k+2+n,p}(w_{\gam+kp})}.
\end{equation*}
By \eqref{eq:Faa}, \eqref{eq:grad_del h_1_grad} and \eqref{eq:grad_del h_2_grad}, we have for $\beta\leq \alpha$ with $|\alpha|, |\beta|\geq 1$ that \eqref{eq:est_prodrule} can be further estimated as
\begin{align*}
  &\big\| [\d^\beta ((\d^\nu h_i)\circ \Upsilon^{-1})][\d^{\alpha-\beta}\d^\mu u]\big\|_{L^p(w_{\gam+kp})}  \\
  & \lesssim [\OO]_{C^1}\sum_{1\leq |\delta|\leq |\beta|}\sum_{s=1}^{|\beta|}\sum_{p_s(\beta,\delta)}\|\d^{\alpha-\beta}\d^\mu u\|_{L^p(w_{\gam+kp-|\delta|p-\sum _{m=1}^s(|\vec{\ell}_m|-1)|\vec{k}_m |p})}\\
  &\lesssim [\OO]_{C^1}\|\d^{\alpha-\beta}\d^\mu u\|_{W^{|\beta|,p}(w_{\gam+kp})}\lesssim [\OO]_{C^1}\|u\|_{W^{k+2+n,p}(w_{\gam+kp})},
\end{align*}
where we have applied Hardy's inequality (Corollary \ref{cor:Sob_embRRdh}) $|\beta|$ times using that 
\begin{equation*}
  \gam+kp-|\delta|p-\sum _{m=1}^s(|\vec{\ell}_m|-1)|\vec{k}_m |p \stackrel{\eqref{eq:setp_s}}= \gam+kp -|\beta|p > (1-n)p-1\geq-1,
\end{equation*}
since $\gam>p-1$, $|\beta|\leq k+n$ and $n\in\{0,1\}$. This completes the proof of \ref{it:est1}.

\textit{Step 3: proof of \ref{it:est2}}. Let $\lambda\in(0,1]$, $\gam\in((1-\lambda)p-1, p-1)$, $n\in\{0,1\}$, $j\in\{0,1\}$ and $\OO$ a special $\Cc^{j+1,\lambda}$-domain.
Consider \eqref{eq:est_prodrule} with $|\alpha|\leq k+j+n$. In the case that $\beta=0$ it follows from \eqref{eq:grad_del h_1_grad} that
\begin{equation*}
  \|((\d^\nu h_i)\circ \Upsilon^{-1})(\d^\alpha\d^\mu u)\|_{L^p(w_{\gam+kp})}\lesssim [\OO]_{C^{j+1, \lambda}}\|u\|_{W^{k+2+j+n,p}(w_{\gam+kp})}.
\end{equation*}
By \eqref{eq:Faa}, \eqref{eq:grad_del h_1_grad} and \eqref{eq:grad_del h_2_grad}, we have for $\beta\leq \alpha$ with $|\alpha|, |\beta|\geq 1$ that \eqref{eq:est_prodrule} can be further estimated as
\begin{equation}\label{eq:split_delta}
  \begin{aligned}
  \big\| &[\d^\beta ((\d^\nu h_i)\circ \Upsilon^{-1})][\d^{\alpha-\beta}\d^\mu  u]\big\|_{L^p(w_{\gam+kp})}\\
    \lesssim&\; [\OO]_{C^{j+1,\lambda}} \sum_{1\leq |\delta|\leq j}\sum_{s=1}^{|\beta|}\sum_{p_s(\beta,\delta)}\|\d^{\alpha-\beta}\d^\mu  u\|_{L^p(w_{\gam+kp -\sum_{m=1}^s (|\vec{\ell}_m|-(j+1))_+|\vec{k}_m|p})}  \\
  &+[\OO]_{C^{j+1,\lambda}}\sum_{j+1\leq |\delta|\leq |\beta|}\sum_{s=1}^{|\beta|}\sum_{p_s(\beta,\delta)} \|\d^{\alpha-\beta}\d^\mu  u\|_{L^p(w_{\gam+kp-(|\delta|-j-\lambda)p-\sum_{m=1}^s (|\vec{\ell}_m|-1)|\vec{k}_m|p})},
\end{aligned}
\end{equation}
where the sum over $1\leq |\delta|\leq j$ is only present if $j=1$ and in this case we have $(|\delta|-j-\lambda)_+=0$.
We first consider the case $j+1\leq |\delta|\leq |\beta|$ for $j\in\{0,1\}$. Note that by \eqref{eq:setp_s} we have
\begin{align*}
\gam+kp-(|\delta|-j-\lambda)p-\sum_{m=1}^s (|\vec{\ell}_m|-1)|\vec{k}_m|p &=\gam+kp-\big(|\beta|-j- \lambda\big)p\\&> (1-n)p-1\geq -1.
\end{align*}
Therefore, Lemma \ref{lem:frac_Hardy2} applied with $s=|\beta|-j-\lambda\leq |\beta|$ yields
\begin{align*}
  \|\d^{\alpha-\beta}\d^\mu  u\|_{L^p(w_{\gam+kp-(|\beta|-j- \lambda)p})}  \lesssim \|\d^{\alpha-\beta}\d^\mu u\|_{W^{|\beta|,p}(w_{\gam+kp})}
  \leq \|u\|_{W^{k+2+j+n,p}(w_{\gam+kp})}. 
\end{align*}
In the case that $j=1$, we additionally estimate the sum over $|\delta|=1$ in \eqref{eq:split_delta}. In the case that $|\vec{\ell}_m|\leq j+1=2$ for all $m\in\{1, \dots, s\}$, we have $(|\vec{\ell}_m|-2)_+=0$ and
\begin{equation*}
  \|\d^{\alpha-\beta}\d^\mu  u\|_{L^p(w_{\gam+kp})}\lesssim \|u\|_{W^{k+3+n,p}(w_{\gam+kp})}.
\end{equation*}
If there exists an $m_0\in \{1,\dots, s\}$ such that $|\vec{\ell}_{m_0}|>2$, then it follows from \eqref{eq:setp_s} and $|\beta|\leq k+1+n$ that
\begin{align*}
  \gam+kp-\sum_{m=1}^s(|\vec{\ell}_m|-2)_+|\vec{k}_m|p &= \gam+kp-\Big(\sum_{\substack{m=1\\m\neq m_0}}^s(|\vec{\ell}_m|-2)_+|\vec{k}_m| +(|\vec{\ell}_{m_0}|-2)|\vec{k}_{m_0}|\Big) p \\
  &\geq \gam +kp -\Big(\sum_{\substack{m=1\\m\neq m_0}}^s|\vec{\ell}_m||\vec{k}_m| +|\vec{\ell}_{m_0}||\vec{k}_{m_0}| - 2|\vec{k}_{m_0}|\Big)p\\
 & \geq \gam+kp - |\beta|p + 2p > (2-n-\lambda)p-1\geq-1.
\end{align*}
Therefore, Lemma \ref{lem:frac_Hardy2} (applied with $s$ replaced by $\sum_{m=1}^s(|\vec{\ell}_m|-2)_+|\vec{k}_m|\leq |\beta|$), yields
\begin{align*}
  \|\d^{\alpha-\beta}\d^\mu  u\|_{L^p(w_{\gam+kp-\sum_{m=1}^s(|\vec{\ell}_m|-2)_+|\vec{k}_m|p})}  \lesssim \|\d^{\alpha-\beta}\d^\mu  u\|_{W^{|\beta|,p}(w_{\gam+kp})}
  \leq \|u\|_{W^{k+3+n,p}(w_{\gam+kp})}. 
\end{align*}
This finishes the proof of \ref{it:est2}.
\end{proof}

We continue with some preliminary estimates for the perturbation term with $|\mu|=1$.

\begin{lemma}[Estimates for $|\mu|=1$]\label{lem:estB3}
  Let $p\in(1,\infty)$, $k\in\NN_0$ and let $X$ be a $\UMD$ Banach space. Take $|\mu|=1$, $|\nu|=2$ and $i\in\{1,\dots, d\}$. Let $\Upsilon\in\{\Phi, \Psi\}$ and define
\begin{equation*}
    P:=\big[(\d^{\nu} h_{i})\circ \Upsilon^{-1}\big]\d^{\mu},
\end{equation*}
  \begin{enumerate}[(i)]
    \item\label{it:est1B3} If $\gam\in (p-1,2p-1)$, then for $n\in\{0,1\}$ it holds that
        \begin{equation*}
          \|P  u\|_{W^{k+n,p}(\RRdh, w_{\gam+kp};X)}\leq C\cdot[\OO]_{C^1}\cdot\|u\|_{W^{k+2+n,p}(\RRdh, w_{\gam+kp};X)},
        \end{equation*}
            for
    \begin{equation*}
      u\in \begin{cases}
             W^{k+2,p}(\RRdh, w_{\gam+kp};X) & \mbox{if } n=0,\\
            W^{k+3,p}_0(\RRdh, w_{\gam+kp};X) & \mbox{if } n=1.
           \end{cases}
    \end{equation*}
    \item\label{it:est2B3} If $\lambda\in (0,1]$, $\gam\in ((1-\lambda)p-1, p-1)$ and $\OO$ is a special $\Cc^{1, \lambda}$-domain, then for $u\in [W^{k,p}(\RRdh, w_{\gam+kp};X), W^{k+2, p}(\RRdh, w_{\gam+kp};X)]_{1-\frac{\lambda}{2}}$ it holds that
      \begin{equation*}
           \qquad\|P u\|_{W^{k,p}(\RRdh, w_{\gam+kp};X)}\leq C\cdot [\OO]_{C^{1,\lambda}}\cdot \|u\|_{[W^{k,p}(\RRdh, w_{\gam+kp};X), W^{k+2, p}(\RRdh, w_{\gam+kp};X)]_{1-\frac{\lambda}{2}}}.
      \end{equation*}
      In particular, for $u\in W^{k+2,p}(\RRdh, w_{\gam+kp};X) $ it holds that 
        \begin{equation*}
          \qquad\|P u\|_{W^{k,p}(\RRdh, w_{\gam+kp};X)}\leq C\cdot [\OO]_{C^{1,\lambda}}\cdot\|u\|_{W^{k+2,p}(\RRdh, w_{\gam+kp};X)}.
        \end{equation*}
      \item\label{it:est3B3} If $\lambda\in (0,1]$, $\gam\in ((1-\lambda)p-1, p-1)$ and $\OO$ is a special $\Cc^{2, \lambda}$-domain, then for $u\in W^{k+2,p}(\RRdh, w_{\gam+kp};X)$ it holds that
        \begin{equation*}
          \qquad\|P u\|_{W^{k+1,p}(\RRdh, w_{\gam+kp};X)}\leq C\cdot [\OO]_{C^{2,\lambda}}\cdot\|u\|_{W^{k+2,p}(\RRdh, w_{\gam+kp};X)}.
        \end{equation*}
  \end{enumerate}
  In all cases, the constant $C>0$ only depends on $p,k,\gam,\lambda, d $ and $X$.
\end{lemma}

Note that in Lemma \ref{lem:estB3}\ref{it:est1B3} with $n=1$, we need two traces of $u$ to be zero. This will not be a problem later on, since the Neumann trace will disappear in the complex interpolation space, see Step 1 in the proof of Theorem \ref{thm:Dirichlet_Laplacian_special}.
\begin{proof}
  For notational convenience we write $W^{k,p}(w_{\gam}):=W^{k,p}(\RRdh,w_{\gam};X)$. 
  
  \textit{Step 1: preparations.} For $\alpha\in \NN_0^d$ and some regular enough $u$ we obtain with the product rule that 
  \begin{equation}\label{eq:est_prodruleB3}
\begin{aligned}
  \|\d^{\alpha}[((\d^{\nu} h_i)\circ & \Upsilon^{-1}) \d^\mu u]\|_{L^p(w_{\gam+kp})}\\&\lesssim \sum_{\beta\leq \alpha}\big\| [\d^\beta ((\d^\nu h_i)\circ \Upsilon^{-1})][\d^{\alpha-\beta}\d^\mu u]\big\|_{L^p(w_{\gam+kp})}.
\end{aligned}
\end{equation}
In the case that $|\alpha|,|\beta|\geq 1$ and $y\in \RRdh$, the multivariate Fa\`a di Bruno's formula \cite[Theorem 2.1]{CS96} implies
\begin{equation}\label{eq:FaaB3}
  |\d_y^\beta(\d^\nu h_i)(\Upsilon^{-1}(y))|\lesssim \sum_{1\leq |\delta|\leq |\beta|}|(\d^\delta\d^\nu h_i)(\Upsilon^{-1}(y))|\sum_{s=1}^{|\beta|}\sum_{p_s(\beta,\delta)}\prod_{m=1}^s|\d^{\vec{\ell}_m}\Upsilon^{-1}(y)|^{\vec{k}_m},
\end{equation}
where the sets $p_s(\beta,\delta)$ are given as in \eqref{eq:setp_s}.
By Lemma \ref{lem:localization_weighted_blow-up}\ref{it:lem:localization_weighted_blow-up;dist_preserving}+\ref{it:lem:localization_weighted_blow-up;est} for $\Upsilon=\Phi$ and Lemma \ref{lem:loc_normal}\ref{it:lem:loc_normal2}+\ref{it:lem:loc_normal5} for $\Upsilon=\Psi$, we have the estimate
\begin{equation}\label{eq:grad_del h_1_del}
  |(\d^{\delta}\d^\nu h_i)(\Upsilon^{-1}(y))| \lesssim\frac{ [\OO]_{C^{j+1,\lambda}}}{\mrm{dist}(\Upsilon^{-1}(y),\BDom)^{( |\delta|+1-j-\lambda)_+}}
\lesssim \frac{[\OO]_{C^{j+1,\lambda}}}{y_1^{( |\delta|+1-j-\lambda)_+}},
\end{equation}
for all $\lambda\in[0,1]$, $j\in\{0,1\}$, $\delta\in \NN_0^d$, $|\nu|=2$ and $y \in \R^d_+$. Moreover, by the same lemmas we also have the (non-optimal) estimate
\begin{equation}\label{eq:grad_del h_2_del}
  |\d^{\vec{\ell}}\Upsilon^{-1}(y)|\lesssim \frac{[\OO]_{C^{1}}}{y_1^{(|\vec{\ell}|-1)_+}},
\end{equation}
for all $\vec{\ell}\in \NN_0^d$ and $y \in \R^d_+$.

\textit{Step 2: proof of \ref{it:est1B3}}.  Let $\gam\in(p-1, 2p-1)$, $n\in\{0,1\}$ and $\OO$ a special $\Cc^1$-domain. To prove \ref{it:est1B3} we need to consider \eqref{eq:est_prodruleB3} with $|\alpha|\leq k+n$. If $\beta=0$ in \eqref{eq:est_prodruleB3}, then it follows from \eqref{eq:grad_del h_1_del} and Hardy's inequality (Corollary \ref{cor:Sob_embRRdh}, using that $\gam+(k-1)p>-1$) that 
\begin{equation*}
  \|((\d^\nu h_i)\circ \Upsilon^{-1})(\d^\alpha\d^\mu u)\|_{L^p(w_{\gam+kp})}\lesssim [\OO]_{C^{1}}\|\d^\alpha\d^\mu u\|_{L^p(w_{\gam+(k-1)p})} \lesssim [\OO]_{C^{1}}\|u\|_{W^{k+2+n,p}(w_{\gam+kp})}.
\end{equation*}
By \eqref{eq:FaaB3}, \eqref{eq:grad_del h_1_del} and \eqref{eq:grad_del h_2_del}, we have for $\beta\leq \alpha$ with $|\alpha|, |\beta|\geq 1$ that \eqref{eq:est_prodruleB3} can be further estimated as
\begin{align*}
  &\big\| [\d^\beta ((\d^\nu h_i)\circ \Upsilon^{-1})][\d^{\alpha-\beta}\d^\mu u]\big\|_{L^p(w_{\gam+kp})}  \\
  & \lesssim [\OO]_{C^1}\sum_{1\leq |\delta|\leq |\beta|}\sum_{s=1}^{|\beta|}\sum_{p_s(\beta,\delta)}\|\d^{\alpha-\beta}\d^\mu u\|_{L^p(w_{\gam+kp-(|\delta|+1)p-\sum _{m=1}^s(|\vec{\ell}_m|-1)|\vec{k}_m |p})}\\
  &\lesssim [\OO]_{C^1}\|\d^{\alpha-\beta}\d^\mu u\|_{W^{|\beta|+1,p}(w_{\gam+kp})}\lesssim [\OO]_{C^1} \|u\|_{W^{k+2+n,p}(w_{\gam+kp})},
\end{align*}
where we have applied Hardy's inequality $|\beta|+1$ times using that 
\begin{equation*}
  \gam+kp-(|\delta|+1)p-\sum _{m=1}^s(|\vec{\ell}_m|-1)|\vec{k}_m |p \stackrel{\eqref{eq:setp_s}}= \gam+kp -(|\beta|+1)p > - np-1,
\end{equation*}
since $\gam>p-1$, $|\beta|\leq k+n$ and $n\in\{0,1\}$. This shows that for $n=1$ we need to take $u\in W_0^{k+3,p}(w_{\gam+kp})$ by Hardy's inequality. This completes the proof of \ref{it:est1B3}.

\textit{Step 3: proof of \ref{it:est2B3}}.
Let $\lambda\in(0,1]$, $\gam\in((1-\lambda)p-1, p-1)$ and $\OO$ a special $\Cc^{1,\lambda}$-domain.
Consider \eqref{eq:est_prodruleB3} with $|\alpha|\leq k$. If $\beta=0$ in \eqref{eq:est_prodruleB3}, then it follows from \eqref{eq:grad_del h_1_del} and Lemma \ref{lem:frac_Hardy3} (applied to $s=1-\lambda$ and using that $X$ is $\UMD$) that
\begin{align*}
  \|((\d^\nu h_i)\circ \Upsilon^{-1})(\d^\alpha\d^\mu u)\|_{L^p(w_{\gam+kp})}&\lesssim [\OO]_{C^{1, \lambda}}\|\d^\alpha\d^\mu u\|_{L^p(w_{\gam+kp-(1-\lambda)p})}\\
   &\lesssim [\OO]_{C^{1, \lambda}}\| u\|_{W^{k+1,p}(w_{\gam+kp-(1-\lambda)p})}\\ &\lesssim [\OO]_{C^{1, \lambda}}\|u\|_{[W^{k,p}(\RRdh, w_{\gam+kp};X), W^{k+2, p}(\RRdh, w_{\gam+kp};X)]_{1-\frac{\lambda}{2}}}.
\end{align*}
By \eqref{eq:FaaB3}, \eqref{eq:grad_del h_1_del} and \eqref{eq:grad_del h_2_del}, we have for $\beta\leq \alpha$ with $|\alpha|, |\beta|\geq 1$ that \eqref{eq:est_prodruleB3} can be further estimated as
\begin{align*}
  &\big\| [\d^\beta ((\d^\nu h_i)\circ \Upsilon^{-1})][\d^{\alpha-\beta}\d^\mu  u]\big\|_{L^p(w_{\gam+kp})}  \\
  & \lesssim [\OO]_{C^{1,\lambda}}\sum_{1\leq |\delta|\leq |\beta|}\sum_{s=1}^{|\beta|}\sum_{p_s(\beta,\delta)}\|\d^{\alpha-\beta}\d^\mu u\|_{L^p(w_{\gam+kp-(|\delta|+1-\lambda)p-\sum _{m=1}^s(|\vec{\ell}_m|-1)|\vec{k}_m |p})}.
\end{align*}
By \eqref{eq:setp_s} and the Hardy inequalities from Lemmas \ref{lem:Hardy} (using $\gam+kp-(|\beta|+1-\lambda)p > -1$) and \ref{lem:frac_Hardy3}, we obtain
\begin{align*}
  \|\d^{\alpha-\beta}\d^\mu u\|_{L^p(w_{\gam+kp-(|\delta|+1-\lambda)p-\sum _{m=1}^s(|\vec{\ell}_m|-1)|\vec{k}_m |p})}&=\|\d^{\alpha-\beta}\d^\mu u\|_{L^p(w_{\gam+kp-(|\beta|+1-\lambda)p})}\\
  &\lesssim \|u\|_{W^{k+1,p}(w_{\gam+kp-(1-\lambda)p})}\\
  &\lesssim \|u\|_{[W^{k,p}(\RRdh, w_{\gam+kp};X), W^{k+2, p}(\RRdh, w_{\gam+kp};X)]_{1-\frac{\lambda}{2}}}.
\end{align*}
This completes the proof of \ref{it:est2B3}.

\textit{Step 4: proof of \ref{it:est3B3}}. Let $\lambda\in(0,1]$, $\gam\in((1-\lambda)p-1, p-1)$ and $\OO$ a special $ \Cc^{2,\lambda}$-domain. Consider \eqref{eq:est_prodruleB3} with $|\alpha|\leq k+1$. If $\beta=0$ in \eqref{eq:est_prodruleB3}, then it follows from \eqref{eq:grad_del h_1_del} that
\begin{equation*}
  \|((\d^\nu h_i)\circ \Upsilon^{-1})(\d^\alpha\d^\mu u)\|_{L^p(w_{\gam+kp})}\lesssim [\OO]_{C^{2, \lambda}}\|\d^\alpha\d^\mu u\|_{L^p(w_{\gam+kp})} \lesssim [\OO]_{C^{2, \lambda}}\|u\|_{W^{k+2,p}(w_{\gam+kp})}.
\end{equation*}
By \eqref{eq:FaaB3}, \eqref{eq:grad_del h_1_del} and \eqref{eq:grad_del h_2_del}, we have for $\beta\leq \alpha$ with $|\alpha|, |\beta|\geq 1$ that \eqref{eq:est_prodruleB3} can be further estimated as
\begin{align*}
  &\big\| [\d^\beta ((\d^\nu h_i)\circ \Upsilon^{-1})][\d^{\alpha-\beta}\d^\mu u]\big\|_{L^p(w_{\gam+kp})}  \\
  & \lesssim [\OO]_{C^{2,\lambda}}\sum_{1\leq |\delta|\leq |\beta|}\sum_{s=1}^{|\beta|}\sum_{p_s(\beta,\delta)}\|\d^{\alpha-\beta}\d^\mu u\|_{L^p(w_{\gam+kp-(|\delta|-\lambda)p-\sum _{m=1}^s(|\vec{\ell}_m|-1)|\vec{k}_m |p})}\\
  &\lesssim [\OO]_{C^{2,\lambda}}\|\d^{\alpha-\beta}\d^\mu u\|_{W^{|\beta|,p}(w_{\gam+kp})}\lesssim [\OO]_{C^{2,\lambda}} \|u\|_{W^{k+2,p}(w_{\gam+kp})},
\end{align*}
where we have used Lemma \ref{lem:frac_Hardy2} with $s$ replaced by $|\beta|-\lambda$ and that
\begin{equation*}
  \gam+kp-(|\delta|-\lambda)p-\sum _{m=1}^s(|\vec{\ell}_m|-1)|\vec{k}_m |p = \gam+kp - |\beta|p+\lambda p>-1.
\end{equation*}
This finishes the proof of \ref{it:est3B3}.
\end{proof}

The fact that we need boundary conditions in the spaces in parts of Lemma \ref{lem:estB3} will complicate the proof of perturbing the $\Hinf$-calculus in Section \ref{sec:proof_special}. In particular, for the Dirichlet Laplacian on special $\Cc^{1,\lambda}$-domains, we need an additional estimate, which we obtain via extrapolation spaces and the adjoint operator.\\

Let $p\in(1,\infty)$, $\gam\in \RR$, $\OO\subseteq \RR^d$ open and let $X$ be a reflexive Banach space (which is implied by the $\UMD$ condition). Then $L^p(\OO, w^{\d\OO}_{\gam};X)$ is reflexive and with the unweighted pairing 
\begin{equation*}
  \langle f,g\rangle_{L^p(\OO, w^{\d\OO}_{\gam};X)\times (L^p(\OO, w^{\d\OO}_{\gam};X))'} = \int_\OO \langle f(x), g(x) \rangle_{X\times X'} \dd x,
\end{equation*}
its dual space is 
\begin{equation*}
  (L^p(\OO, w^{\d\OO}_{\gam};X))' = L^{p'}(\OO, w^{\d\OO}_{\gam'}; X'),
\end{equation*}
where $p'=p/(p-1)$ and $\gam' = -\gam/(p-1)$. Note that if $\gam\in (-1,p-1)$, then $\gam'\in (-1, p'-1)$.\\

We have the following characterisation of the adjoint operator of the Dirichlet Laplacian. We note that for $\gam\in(p-1,2p-1)$ the characterisation of the domain of the adjoint is more sophisticated, see \cite[Proposition 6.6]{LV18}.  
\begin{proposition}[{\cite[Proposition 6.5]{LV18}}]\label{prop:adjoint}
  Let $p\in(1,\infty)$, $\gam\in(-1,p-1)$ and let $X$ be a $\UMD$ Banach space. Let $A_{p,\gam,X}:=\delDir^{\RRdh}$ on $L^p(\RRdh, w_{\gam};X)$ be the Dirichlet Laplacian as in Definition \ref{def:delRRdh}. Then the adjoint operator is $(A_{p,\gam,X})'= A_{p', \gam', X'}$.
\end{proposition}

To continue, we briefly recall the extrapolation scales, see \cite[Appendix E]{KW04} or \cite[Chapter 5]{Am95} for more details. Let $A$ be a sectorial operator on a Banach space $Y$ such that $0\in\rho(A)$.  Then for any $\alpha\in \RR$, we can define the scale of extrapolation spaces
\begin{equation*}
  (E_{\alpha,A}, \|\cdot\|_{E_{\alpha,A}}) = \begin{cases}
                                               (D(A^{\alpha}), \|A^{\alpha}\cdot\|_Y) & \mbox{if }\alpha\geq 0,\\
                                               (Y, \|A^{\alpha}\cdot\|_Y)^{\sim} & \mbox{if }\alpha<0.
                                             \end{cases}
\end{equation*}
where $\sim$ denotes the completion of the space. Let $A'$ denote the adjoint of $A$. In the case that $Y$ is reflexive and $\alpha\in\RR$, the extrapolation scale satisfies
\begin{equation}\label{eq:extrapol-scale}
  E_{-\alpha,A} = (E_{\alpha, A'})',
\end{equation}
with respect to the duality  $\langle Y, Y'\rangle$.\\

With the extrapolation scale and the characterisation of the adjoint, we can prove the following estimate for the perturbation terms coming from the Dirichlet Laplacian in \eqref{eq:pertBDIR} on weighted Lebesgue spaces.

\begin{lemma}\label{lem:lebesB} Let $p\in(1,\infty)$, $\lambda\in (0,1]$, $\gam\in ((1-\lambda)p-1, p-1)$ and let $X$ be a $\UMD$ Banach space. Let $\delDir^{\RRdh}$ on $L^p(\RRdh, w_{\gam};X)$ be the Dirichlet Laplacian as in Definition \ref{def:delRRdh}.
Assume that $\OO$ is a special $\Cc^{1,\lambda}$-domain with $[\OO]_{C^{1,\lambda}}\leq 1$.  Then the perturbation $B$ in \eqref{eq:pertBDIR} satisfies 
  \begin{equation*}
    \|(\mu-\delDir^{\RRdh})^{-\half} B u\|_{L^p(\RRdh, w_{\gam};X)}\leq C\cdot[\OO]_{C^{1, \lambda}}\cdot\|(\mu-\delDir^{\RRdh})^{\half}u \|_{L^p(\RRdh, w_{\gam};X)},
  \end{equation*}
  for all $\mu>0$ and $u\in W^{1,p}_\Dir(\RRdh, w_{\gam};X)$.
\end{lemma}
\begin{proof}
  We write $A:=\mu-\delDir^{\RRdh}$. Note that \eqref{eq:extrapol-scale}, Propositions \ref{prop:adjoint} and \ref{prop:domain_char_RRdh} imply that
  \begin{equation*}
    \|A^{-\half} B u \|_{L^p(\RRdh, w_{\gam};X)}\eqsim \sup \big|\langle Bu, v\rangle_{L^p(\RRdh, w_{\gam};X)\times L^{p'}(\RRdh, w_{\gam'};X')}\big|,
  \end{equation*}
  where the supremum is taken over all $v\in E_{\half,A'}=D((A')^\half)=W^{1,p'}_\Dir(\RRdh, w_{\gam'};X')$ with $\|v\|_{W^{1,p'}(\RRdh, w_{\gam'};X')}\leq 1$. Fix such a $v\in W^{1,p'}_\Dir(\RRdh, w_{\gam'};X')$. Recall from \eqref{eq:pertBDIR} that the perturbation terms in $B$ are of the form $((\d^{\nu}h_1)\circ\Phi^{-1})^\kappa\d^\mu\d_1 $ with $|\mu|=|\nu|=1$ and $\kappa\in\{1,2\}$, or, with $|\nu|=2$, $|\mu|=0$ and $\kappa=1$. Therefore, by Lemma \ref{lem:localization_weighted_blow-up}\ref{it:lem:localization_weighted_blow-up;est}, integration by parts, H\"older's inequality, Lemma \ref{lem:frac_Hardy2} and Proposition \ref{prop:domain_char_RRdh}, we obtain
  \begin{align*}
    \big|\langle& Bu, v\rangle_{L^p(\RRdh, w_{\gam};X)\times L^{p'}(\RRdh, w_{\gam'};X')}\big|\\
     & \lesssim[\OO]_{C^{1,\lambda}}\Big( \sum_{|\mu|=1} \int_{\RRdh}|\langle \d^{\mu}\d_1 u, v\rangle_{X\times X'}|\dd x + \int_{\RRdh}x_1^{-(1-\lambda)}|\langle \d_1 u, v\rangle_{X\times X'}|\dd x\Big) \\
     & \leq [\OO]_{C^{1,\lambda}} \Big(\int_{\RRdh}x_1^{\gam}\|\d_1 u\|^p_X\dd x\Big)^{\frac{1}{p}}\\
     &\qquad \cdot\Big[\sum_{|\mu|=1}\Big(\int_{\RRdh}x_1^{\gam'}\|\d^\mu v\|^{p'}_{X'}\dd x\Big)^{\frac{1}{p'}}+ \Big(\int_{\RRdh}x_1^{\gam'-(1-\lambda)p'}\|v\|^{p'}_{X'}\dd x\Big)^{\frac{1}{p'}}\Big]\\
     &\lesssim[\OO]_{C^{1,\lambda}} \|u\|_{W^{1,p}(\RRdh, w_{\gam};X)}\|v\|_{W^{1,p'}(\RRdh, w_{\gam'};X')}\lesssim [\OO]_{C^{1,\lambda}}\|A^{\half} u\|_{L^p(\RRdh, w_{\gam};X)}.
  \end{align*}
  This proves the desired estimate.
\end{proof}

\subsection{The proofs of Theorems \ref{thm:Dirichlet_Laplacian_special} and \ref{thm:Neumann_Laplacian_special}}\label{sec:proof_special}
With the preliminary estimates on the perturbation terms, we can now continue with proving the boundedness of the $\Hinf$-calculus for the Laplacian on special domains. We start with the proof of Theorem \ref{thm:Dirichlet_Laplacian_special} for the Dirichlet Laplacian.

\begin{proof}[Proof of Theorem \ref{thm:Dirichlet_Laplacian_special}] Let $\OO$ be a special domain as specified in the theorem, which is of the form
\begin{equation*}
    \OO = \{x\in \RRd: x_1>h(\tilde{x})\},
\end{equation*}
and let $h_1$ and $\Phi$ be as in Lemma \ref{lem:localization_weighted_blow-up}.
Recall that we introduced $\Delta^{\Phi}: W^{2,1}_{\loc}(\RRdh;X)\to L^1_{\loc}(\RRdh; X)$ given by 
\begin{equation*}
\begin{aligned}
  \Delta^{\Phi} :\!\!&= \Phi_* \circ \Delta\circ (\Phi^{-1})_{*}\\
  &= \Delta - 2((\nabla h_1) \circ \Phi^{-1}) \cdot \nabla \partial_1 + |(\nabla h_1) \circ \Phi^{-1}|^2\,\partial_1^2  - ((\Delta h_1) \circ \Phi^{-1})\partial_1\\
  &=:\del + B_1+B_2+B_3.
\end{aligned}
\end{equation*}
  Let $-\delDir^\Phi$ denote the realisation of $\del^\Phi$ in $W^{k,p}(\RRdh, w_{\gam+kp};X)$ with domain $D(-\delDir^\Phi)=W^{k+2,p}_\Dir(\RRdh, w_{\gam+kp};X)$. Due to the isomorphisms in Proposition \ref{prop:isomorphisms}, the trace characterisation in Proposition \ref{prop:tracechar_dom} and standard properties of the $\Hinf$-calculus, the desired statements in Theorem \ref{thm:Dirichlet_Laplacian_special} for $-\delDir$ on $W^{k,p}(\OO, w_{\gam+kp}^{\d\OO};X)$ are equivalent to the corresponding statements for $-\delDir^\Phi$ on $W^{k,p}(\RRdh, w_{\gam+kp};X)$. We will apply the perturbation theorems from Section \ref{subsec:calc_pert} to show that the $\Hinf$-calculus for the Laplacian on the half-space is preserved under the perturbation $B:=B_1+B_2+B_3$. We note that the estimates for $B_1$ and $B_2$ are covered in Lemma \ref{lem:estB1B2} and for $B_3$ in Lemma \ref{lem:estB3}.
  
  \textit{Step 1: the case $\gam\in(p-1,2p-1)$.} Let $\gam\in(p-1,2p-1)$ and let $\OO$ be a special $\Cc^1$-domain. Let $\mu>0$ and we write $A_\Dir := \mu-\delDir^{\RRdh}$. We apply Theorem \ref{thm:perturbcalculus} to show that $\mu-(\delDir^{\RRdh}+B)$ has a bounded $\Hinf$-calculus. Let $u\in D(A_\Dir)=W^{k+2,p}_\Dir(\RRdh, w_{\gam+kp};X)$, then by  Lemma \ref{lem:estB1B2}\ref{it:est1}, Lemma \ref{lem:estB3}\ref{it:est1B3} and Remark \ref{rem:domains}, we have
\begin{align*}
  \|Bu \|_{W^{k,p}(\RRdh, w_{\gam+kp};X)} & \lesssim [\OO]_{C^1}\|u\|_{W^{k+2,p}(\RRdh, w_{\gam+kp};X)} \\
   & \eqsim [\OO]_{C^1}\|A_{\Dir}u\|_{W^{k,p}(\RRdh, w_{\gam+kp};X)},
\end{align*}
which shows condition \ref{it:thm:perturbcalculus1} of Theorem \ref{thm:perturbcalculus}. To show that condition \ref{it:thm:perturbcalculus2} of Theorem \ref{thm:perturbcalculus} holds, note that by Lemma \ref{lem:estB1B2}\ref{it:est1} and Lemma \ref{lem:estB3}\ref{it:est1B3} we have that
\begin{equation}\label{eq:Bbounded}
  \begin{aligned}
  B&\colon W^{k+2,p}(\R^d_+,w_{\gam+kp};X) \to W^{k,p}(\R^d_+,w_{\gam+kp};X) \quad \text{ and } \\
  B&\colon W_0^{k+3,p}(\R^d_+,w_{\gam+kp};X) \to W^{k+1,p}(\R^d_+,w_{\gam+kp};X)
\end{aligned}
\end{equation}
are bounded operators. Take $\theta\in (0,\half)$ such that Proposition \ref{prop:Sneiberg} for $k_0=2$ holds and let $u \in D(A_\Dir^{1+\theta})$. Then, by Corollary \ref{cor:frac_domain_Dir} twice, properties of the complex interpolation method using \eqref{eq:Bbounded},  Proposition \ref{prop:Sneiberg} and the invertibility of $A_\Dir$ we have 
\begin{align*}
  \nrm{A_\Dir^\theta Bu}_{W^{k,p}(\R^d_+,w_{\gam+kp};X)} \leq \nrm{Bu}_{D(A_\Dir^\theta)}&\eqsim \nrm{Bu}_{[W^{k,p}(\R^d_+,w_{\gam+kp};X), W^{k+1,p}(\R^d_+,w_{\gam+kp};X)]_{2\theta}}\\
  &\lesssim \nrm{u}_{[W^{k+2,p}(\R^d_+,w_{\gam+kp};X), W_0^{k+3,p}(\R^d_+,w_{\gam+kp};X)]_{2\theta}}\\
  &\lesssim \nrm{u}_{[W_0^{k+2,p}(\R^d_+,w_{\gam+kp};X), W_0^{k+3,p}(\R^d_+,w_{\gam+kp};X)]_{2\theta}}\\
  &\eqsim \nrm{u}_{[W^{k+2,p}_{\Dir}(\R^d_+,w_{\gam+kp};X), W^{k+3,p}_{\Dir}(\R^d_+,w_{\gam+kp};X)]_{2\theta}}\\
  &\eqsim \nrm{u}_{D(A_{\Dir}^{1+\theta})} \eqsim \nrm{A_\Dir^{1+\theta}u}_{W^{k,p}(\R^d_+,w_{\gam+kp};X)}.
\end{align*}
This shows condition \ref{it:thm:perturbcalculus2} of Theorem \ref{thm:perturbcalculus}. Therefore, Theorems \ref{thm:LLRVthm1.1Dir} and  \ref{thm:perturbcalculus} give that $\mu-\delDir^\Phi$ has a bounded $\Hinf$-calculus of angle zero if $[\OO]_{C^1}$ is small enough.
  
  \textit{Step 2: the case $\gam\in((1-\lambda)p-1,p-1)$.} Let $\lambda\in(0,1]$, $\gam\in((1-\lambda)p-1,p-1)$ and let $\OO$ be a special $\Cc^{1,\lambda}$-domain. We apply Theorem \ref{thm:perturbcalculus} to show that $\mu-(\delDir^{\RRdh}+B)$ has a bounded $\Hinf$-calculus. Let $u\in D(A_\Dir)=W^{k+2,p}_\Dir(\RRdh, w_{\gam+kp};X)$.
  Then by Lemma \ref{lem:estB1B2}\ref{it:est2}, Lemma \ref{lem:estB3}\ref{it:est2B3} and Remark \ref{rem:domains}, we have
\begin{align*}
  \|Bu\|_{W^{k,p}(\RRdh, w_{\gam+kp};X)} & \lesssim [\OO]_{C^{1, \lambda}}\|u\|_{W^{k+2,p}(\RRdh, w_{\gam+kp};X)} \\
  & \eqsim [\OO]_{C^{1,\lambda}}\|A_\Dir u\|_{W^{k,p}(\RRdh, w_{\gam+kp})}.
\end{align*}
Thus, condition \ref{it:thm:perturbcalculus1} of Theorem \ref{thm:perturbcalculus} is satisfied. To continue, we verify condition \ref{it:thm:perturbcalculus3} of Theorem \ref{thm:perturbcalculus} for $\alpha=\half$. If $k=0$, then the required estimate follows from Lemma \ref{lem:lebesB}. If $k\in \NN_1$, then by Proposition \ref{prop:complex_int_W_Dir} and Corollary \ref{cor:frac_domain_Dir}, we have
\begin{align*}
  W^{k,p}(\RRdh, w_{\gam+kp};X)&=[W^{k-1,p}(\RRdh, w_{\gam+p +(k-1)p};X), W^{k+1,p}_\Dir(\RRdh, w_{\gam+p +(k-1)p};X)]_\half\\
  & = D(\tilde{A}_\Dir^\half),
\end{align*}
where $\tilde{A}_\Dir:=\mu-\delDir^{\RRdh}$ on $W^{k-1,p}(\RRdh, w_{\gam+p +(k-1)p};X)$. Moreover, note that by definition of the fractional powers and \cite[Lemma 6.4]{LLRV24}, it follows that the fractional powers of $A_\Dir$ and $\tilde{A}_\Dir$ are consistent. Therefore, together with Lemma \ref{lem:estB1B2}\ref{it:est1}, Lemma \ref{lem:estB3}\ref{it:est1B3} and Remark \ref{rem:domains}, we obtain
\begin{align*}
  \|A_\Dir^{-\half}Bu\|_{W^{k,p}(\RRdh, w_{\gam+kp};X)} & \eqsim \|\tilde{A}_\Dir^\half A_\Dir^{-\half}Bu\|_{W^{k-1,p}(\RRdh, w_{\gam+p+(k-1)p};X)} \\
  & = \|Bu\|_{W^{k-1,p}(\RRdh, w_{\gam+p+(k-1)p};X)} \\
   & \lesssim \|u\|_{W^{k+1,p}(\RRdh, w_{\gam+kp};X)}\eqsim \|A_\Dir^{\half}u\|_{W^{k,p}(\RRdh, w_{\gam+kp};X)},
\end{align*}
for $u\in D(A_\Dir^\half)= W^{k+1,p}_\Dir(\RRdh, w_{\gam+kp};X)$. Therefore, Theorems \ref{thm:LLRVthm1.1Dir} and \ref{thm:perturbcalculus} give that $\mu-\delDir^\Phi$ has a bounded $\Hinf$-calculus of angle zero if $[\OO]_{C^{1,\lambda}}$ is small enough.
\end{proof}

We conclude this section with the proof of Theorem \ref{thm:Neumann_Laplacian_special} about the $\Hinf$-calculus for the Neumann Laplacian.
\begin{proof}[Proof of Theorem \ref{thm:Neumann_Laplacian_special}]
   Let $\OO$ be a special $\Cc^{j+1,\lambda}$-domain with $[\OO]_{C^{j+1,\lambda}}\leq \Lambda$, where $\Lambda\in(0,1)$ is as in Lemma \ref{lem:loc_normal}. Let $\Psi$ be the diffeomorphism from Lemma \ref{lem:loc_normal} and define $\del^{\Psi}: W^{2,1}_\loc(\RRdh;X)\to L^1_\loc(\RRdh;X)$ by $\del^\Psi:= \Psi_*\circ \del \circ(\Psi^{-1})_*$. Recall that $\del^\Psi=\del + B$, where the perturbation term $B$ is a linear combination of terms of the form \eqref{eq:pertBNEU}. In particular, we define $B_1,B_2$ and $B_3$ as, respectively, all the perturbation terms in $B$ of the form
   \begin{equation*}
    \begin{aligned}
    \big[(\d^{\nu_1} h_{i_1})\circ \Psi^{-1}\big]\d^{\mu}, \quad &&|\nu_1|=1, \,|\mu|=2,\\
    \big[(\d^{\nu_1} h_{i_1})\circ \Psi^{-1}\big]\big[(\d^{\nu_2} h_{i_2})\circ \Psi^{-1}\big]\d^{\mu},\quad &&|\nu_1|=|\nu_2|=1,\, |\mu|=2,\\
    \big[(\d^{\nu_1} h_{i_1})\circ \Psi^{-1}\big]\d^{\mu},\quad&& |\nu_1|=2,\, |\mu|=1,
\end{aligned}
\end{equation*}
where $i_1,i_2\in\{1,\dots, d\}$.
   
   For $j\in\{0,1\}$ let $-\delNeu^\Psi$ denote the realisation of $-\del^\Psi$ in $W^{k+j,p}(\RRdh, w_{\gam+kp};X)$ with domain $D(\delNeu^\Psi)=W^{k+2+j,p}_\Neu(\RRdh, w_{\gam+kp};X)$. Due to the isomorphisms in Proposition \ref{prop:isomorphisms}, the trace characterisation in Proposition \ref{prop:tracechar_dom} and standard properties of the $\Hinf$-calculus, the desired statements in Theorem \ref{thm:Neumann_Laplacian_special} for $-\delNeu$ on $W^{k+j,p}(\OO, w_{\gam+kp}^{\d\OO};X)$ are equivalent to the corresponding statements for $-\delNeu^\Psi$ on $W^{k+j,p}(\RRdh, w_{\gam+kp};X)$. We will apply the perturbation theorems from Section \ref{subsec:calc_pert} to show that the $\Hinf$-calculus for the Laplacian on the half-space is preserved under the perturbation $B=B_1+B_2+B_3$. We note that the estimates for $B_1$ and $B_2$ are covered in Lemma \ref{lem:estB1B2} and for $B_3$ in Lemma \ref{lem:estB3}.

Let $\mu>0$ and we write $A_\Neu := \mu-\delNeu^{\RRdh}$. We first apply Theorem \ref{thm:perturbcalculus} to show that $\mu-(\delNeu^{\RRdh}+B_1+B_2)$ has a bounded $\Hinf$-calculus on $W^{k+j,p}(\RRdh, w_{\gam+kp};X)$ for $k\in \NN_0$. Let $u\in D(A_\Neu)=W^{k+j+2,p}_{\Neu}(\RRdh, w_{\gam+kp};X)$.
 Then by Lemma \ref{lem:estB1B2}\ref{it:est2} and Remark \ref{rem:domains}, we have
 \begin{align*}
  \|B_1u+B_2 u \|_{W^{k+j,p}(\RRdh, w_{\gam+kp};X)} & \lesssim [\OO]_{C^{j+1, \lambda}}\|u\|_{W^{k+j+2,p}(\RRdh, w_{\gam+kp};X)} \\
   & \eqsim [\OO]_{C^{j+1,\lambda}}\|A_{\Neu}u\|_{W^{k+j,p}(\RRdh, w_{\gam+kp};X)},
\end{align*}
which shows condition \ref{it:thm:perturbcalculus1} of Theorem \ref{thm:perturbcalculus}. Next, we verify condition \ref{it:thm:perturbcalculus2} of Theorem \ref{thm:perturbcalculus} for $\alpha=\half$. Let $u\in D(A_\Neu^{\frac{3}{2}})=W^{k+j+3,p}_{\del,\Neu}(\RRdh, w_{\gam+kp};X)$, then by Proposition \ref{prop:domain_char_RRdh_Neu}, Lemma \ref{lem:estB1B2}\ref{it:est2} and the invertibility of $A_\Neu$, we have
\begin{align*}
  \|A_\Neu^{\half}(B_1+B_2) u \|_{W^{k+j,p}(\RRdh,w_{\gam+kp};X)}&\leq \|B_1u+B_2u\|_{D(A_\Neu^{\half})}
  \lesssim \|u\|_{W^{k+j+3,p}(\RRdh, w_{\gam+kp};X)}\\
  &\eqsim \|u\|_{D(A_\Neu^{\frac{3}{2}})}\eqsim \|A_\Neu^{\frac{3}{2}}u\|_{W^{k+j,p}(\RRdh, w_{\gam+kp};X)}.
\end{align*}
Therefore, Theorems \ref{thm:LLRVthm1.2Neu} and \ref{thm:perturbcalculus} give that $\mu-(\delNeu^{\RRdh}+B_1+B_2)$ has a bounded $\Hinf$-calculus of angle zero if $[\OO]_{C^{j+1,\lambda}}$ is small enough. 

To obtain that $\mu-\delNeu^{\Psi}$ has a bounded $\Hinf$-calculus, it remains to apply Theorem \ref{thm:lowe_ord_pert} to the lower-order perturbation $B_3$. First, for $j=0$ we apply Theorem \ref{thm:lowe_ord_pert} with $\alpha= 1-\frac{\lambda}{2}\in [\half, 1)$. For $u\in D(\mu - (\delNeu^{\RRdh}+B_1+B_2))=D(A_\Neu)$ we obtain with Lemma \ref{lem:estB3}\ref{it:est2B3} and Proposition \ref{prop:frac_domain} (using Theorem \ref{thm:LLRVthm1.2Neu}) that
\begin{align*}
  \|B_3 u\|_{W^{k,p}(\RRdh, w_{\gam+kp};X)}&\lesssim \|u\|_{[W^{k,p}(\RRdh, w_{\gam+kp};X), W^{k+2,p}(\RRdh, w_{\gam+kp};X)]_{1-\frac{\lambda}{2}}}
  \lesssim  \|u\|_{D(A_\Neu^{1-\frac{\lambda}{2}})}.
\end{align*}
For $j=1$ we apply Theorem \ref{thm:lowe_ord_pert} with $\alpha=\half$. For $u\in D(\mu-(\delNeu^{\RRdh}+B_1+B_2)) = D(A_\Neu)$ we obtain with Lemma \ref{lem:estB3}\ref{it:est3B3} that
\begin{align*}
  \|B_3 u\|_{W^{k+1,p}(\RRdh, w_{\gam+kp};X)}&\lesssim  \|u\|_{W^{k+2,p}(\RRdh, w_{\gam+kp};X)}.
\end{align*}
Observe that by Proposition \ref{prop:complex_int_W_Neu}, the bounded $\Hinf$-calculus for $\mu-(\delNeu^{\RRdh}+B_1+B_2)$ and Proposition \ref{prop:frac_domain}, we obtain
\begin{align*}
  W^{k+2,p}_\Neu(\RRdh, w_{\gam+kp};X) & =[W^{k+1,p}(\RRdh, w_{\gam+kp};X), W^{k+3,p}_\Neu(\RRdh, w_{\gam+kp};X)]_\half \\
 & = [W^{k+1,p}(\RRdh, w_{\gam+kp};X), D(\mu-(\delNeu^{\RRdh}+B_1+B_2))]_\half\\
 & = D\big((\mu-(\delNeu^{\RRdh}+B_1+B_2))^\half\big).
\end{align*}
This shows the required estimate \eqref{eq:lower_ord} for both $j=0$ and $j=1$. Therefore, the bounded $\Hinf$-calculus for $\mu-(\delNeu^{\RRdh}+B_1+B_2)$, Theorem \ref{thm:lowe_ord_pert} and Proposition \ref{prop:calc_pert_Id}\ref{it:prop:calc_pert_Id2}, show that $\mu-\delNeu^{\Psi}$ has a bounded $\Hinf$-calculus of angle zero if $[\OO]_{C^{j+1,\lambda}}$ is small enough. Note that the application of Proposition \ref{prop:calc_pert_Id}\ref{it:prop:calc_pert_Id2} requires sectoriality of $\mu-\delNeu^{\Psi}$ for all $\mu>0$, which can be obtained from \cite[Theorem 16.2.3(2)]{HNVW24} applied to the operator $A=\mu-\delNeu^{\smash{\RRdh}}$, provided that $[\OO]_{C^{j+1,\lambda}}$ is small enough.
\end{proof}

\section{Functional calculus for the Laplacian on bounded domains}\label{sec:calc_dom}

In this section, we establish our main results concerning the $\Hinf$-calculus for the Laplacian on bounded domains. We begin by recalling the definition of the Laplacian in this setting. The relevant weighted Sobolev spaces with vanishing boundary conditions were introduced in Definition \ref{def:spaces_bounded}.

\begin{definition}\label{def:delDirObounded}
Let $p\in(1,\infty)$, $k\in\NN_0$, $\lambda\in[0,1]$ and let $X$ be a $\UMD$ Banach space. 
   \begin{enumerate}[(i)]
    \item\label{it:def:delO1bounded} Let $\gam\in ((1-\lambda)p-1, 2p-1)\setminus\{p-1\}$ and $\OO$ a bounded $C^{1,\lambda}$-domain.     
    The \emph{Dirichlet Laplacian $\delDir$ on $W^{k,p}(\OO,w^{\d\OO}_{\gam+kp};X)$} with $k\in \NN_0$ is defined by
  \begin{equation*}
    \delDir u := \del u\quad \text{ with }\quad D(\delDir):=W^{k+2,p}_{\Dir}(\OO, w^{\d\OO}_{\gam+kp};X).
  \end{equation*}
    \item\label{it:def:delO2bounded} Let $\gam\in ((1-\lambda)p-1, p-1)$, $j\in\{0,1\}$ and $\OO$ a bounded $C^{j+1,\lambda}$-domain. The \emph{Neumann Laplacian $\delNeu$ on $W^{k+j,p}(\OO,w^{\d\OO}_{\gam+kp};X)$} is defined by
  \begin{equation*}
    \delNeu u := \del u\quad \text{ with }\quad D(\delNeu):=W^{k+j+2,p}_{\Neu}(\OO, w^{\d\OO}_{\gam+kp};X).
  \end{equation*}
  \item \label{it:def:delO3bounded} Let $\gam\in ((1-\lambda)p-1, p-1)$, $j\in\{0,1\}$ and $\OO$ a bounded $C^{j+1,\lambda}$-domain. The \emph{Neumann Laplacian $\delNeu$ on the quotient space}
            \begin{equation*}
      W^{k+j,p}(\OO,w^{\d\OO}_{\gam+kp};X)/\{c\ind_\OO: c\in X\}\
    \end{equation*}
    is defined by $\delNeu u:=\del u $ with 
      \begin{equation*}
    D(\delNeu):=W^{k+j+2,p}_\Neu(\OO,w^{\d\OO}_{\gam+kp};X)/\{c\ind_\OO: c\in X\}.
  \end{equation*}
    \end{enumerate}   
\end{definition}

We now state the main results of this paper about the $\Hinf$-calculus for the Laplacian on bounded domains. The proofs of the theorems below are given in Sections \ref{sec:proof23} and \ref{sec:cor_X=C}.
\begin{theorem}[$\Hinf$-calculus for $\mu-\delDir$ on domains]\label{thm:Dirichlet_Laplacian}
   Let $p\in(1,\infty)$, $k\in\NN_0$, $\lambda\in[0,1]$, $\gam\in((1-\lambda)p-1, 2p-1)\setminus\{p-1\}$, $\sigma\in(0,\pi)$ and let $X$ be a $\UMD$ Banach space. Moreover, assume that $\OO$ is a bounded $C^{1,\lambda}$-domain.
    Let $\delDir$ on $W^{k,p}(\Dom,w^{\BDom}_{\gam+kp};X)$ be the Dirichlet Laplacian as in Definition \ref{def:delDirObounded}. Then there exists a $\tilde{\mu}>0$ such that for all $\mu>\tilde{\mu}$ the operator $\mu-\delDir$ has a bounded $\Hinf$-calculus with $\om_{\Hinf}(\mu-\delDir)\leq \sigma$.
\end{theorem}

\begin{theorem}[$\Hinf$-calculus for $\mu-\delNeu$ on domains]\label{thm:Neumann_Laplacian}
   Let $p\in(1,\infty)$, $k\in\NN_0$, $\lambda\in(0,1]$, $\gam\in((1-\lambda)p-1, p-1)$, $j\in\{0,1\}$, $\sigma\in(0,\pi)$ and let $X$ be a $\UMD$ Banach space. Moreover, assume that $\OO$ is a bounded $C^{j+1,\lambda}$-domain.
    Let $\delNeu$ on $W^{k+j,p}(\Dom,w^{\BDom}_{\gam+kp};X)$ or $W^{k+j,p}(\Dom,w^{\BDom}_{\gam+kp};X)/\{c\ind_\OO: c\in X\}$ be the Neumann Laplacian as in Definition \ref{def:delDirObounded}\ref{it:def:delO2bounded} or \ref{it:def:delO3bounded}, respectively.
Then there exists a $\tilde{\mu}>0$ such that for all $\mu>\tilde{\mu}$ the operator $\mu-\delNeu$ has a bounded $\Hinf$-calculus with $\om_{\Hinf}(\mu-\delNeu)\leq \sigma$.
\end{theorem}

For $X=\CC$ we obtain that the spectrum of the Laplacian is independent of the involved parameters. Hence, for the Dirichlet Laplacian we also obtain the $\Hinf$-calculus with $\mu=0$ since zero is not contained in the spectrum. 

\begin{theorem}\label{thm:scalarDir}
Suppose that the assumptions of Theorem \ref{thm:Dirichlet_Laplacian} hold with $X=\CC$. Then the following assertions hold.
  \begin{enumerate}[(i)]
     \item \label{it:thmscalar1} The spectrum $\sigma(-\delDir)$ is discrete, contained in $(0,\infty)$ and is independent of $p\in (1,\infty)$, $k\in\NN_0$ and $\gam\in ((1-\lambda)p-1,2p-1)\setminus\{p-1\}$.
     \item \label{it:thmscalar2} Let $\widetilde{\mu }>0$ be the smallest eigenvalue of $-\delDir$. For all
$\mu >-\widetilde{\mu }$ the operator $\mu -\Delta _{{\mathrm{Dir}}}$ has a
bounded $H^{\infty}$-calculus with
$\omega _{H^{\infty}}(\mu -\Delta _{{\mathrm{Dir}}}) = 0$. 
   \end{enumerate}
\end{theorem}

The spectrum of the Neumann Laplacian on bounded domains contains the eigenvalue zero so we cannot allow for $\mu=0$ unless the constant functions are removed from the spaces.

\begin{theorem}\label{thm:scalarNeu} Let $p\in(1,\infty)$, $k\in\NN_0$, $\lambda\in(0,1]$, $\gam\in((1-\lambda)p-1, p-1)$ and $j\in\{0,1\}$. Moreover, assume that $\OO$ is a bounded $C^{j+1,\lambda}$-domain. 
    If $\delNeu$ is the Neumann Laplacian on $W^{k+j,p}(\Dom,w^{\BDom}_{\gam+kp})$ as in Definition \ref{def:delDirObounded}\ref{it:def:delO2bounded} with $X=\CC$, then the following assertions hold.
  \begin{enumerate}[(i)]
     \item \label{it:thmscalarNeu1} The spectrum $\sigma(-\delNeu)$ is discrete, contained in $[0,\infty)$ and is independent of $p\in (1,\infty)$, $k\in\NN_0$, $\gam\in((1-\lambda)p-1, p-1)$ and $j\in\{0,1\}$.
     \item \label{it:thmscalarNeu2} For all $\mu>0$ the operator $\mu-\delNeu$ has a bounded $H^\infty$-calculus with $\omega_{H^\infty}(\mu-\delNeu) = 0$.
   \end{enumerate}
Moreover, if $\delNeu$ is the Neumann Laplacian on $W^{k+j,p}(\Dom,w^{\BDom}_{\gam+kp})/\{c\ind_\OO: c\in \CC\}$ as in Definition \ref{def:delDirObounded}\ref{it:def:delO3bounded} with $X=\CC$, then the following assertion holds.
  \begin{enumerate}[resume*]
     \item \label{it:thmscalarNeu4} Let $\widetilde{\mu }>0$ be the smallest eigenvalue of $-\delNeu$. For all
$\mu >-\widetilde{\mu }$ the operator
$\mu -\Delta _{\operatorname{Neu}}$ has a bounded $H^{\infty}$-calculus
with $\omega _{H^{\infty}}(\mu -\Delta _{\operatorname{Neu}}) = 0$.
   \end{enumerate}
\end{theorem}

\begin{remark}\label{rem:XvaluedLp}\hspace{2em}
\begin{enumerate}[(i)]
  \item\label{rem:XvaluedLp1} It is an open question whether Theorems \ref{thm:Dirichlet_Laplacian} and \ref{thm:Neumann_Laplacian} (in the case where $\delNeu$ is defined as in Definition \ref{def:delDirObounded}\ref{it:def:delO3bounded}) with general $\UMD$ Banach spaces $X$ also hold for $\mu=0$. In the following special cases, one can actually conclude the result of Theorems \ref{thm:Dirichlet_Laplacian} and \ref{thm:Neumann_Laplacian} with $\mu=0$.
      \begin{itemize}
        \item If $X$ is a Hilbert space or isomorphic to a closed subspace of an $L^p$-space, then by redoing the proofs of \cite[Proposition 2.1.2 \& Theorem 2.1.9]{HNVW16} for Sobolev spaces, one sees that the results in the scalar case with $\mu=0$ (Theorems \ref{thm:scalarDir} and \ref{thm:scalarNeu}) extend to the vector-valued case.
        \item If $X$ is a $\UMD$ Banach space and $k=0$, then using \cite[Theorem 2.1.3]{HNVW16} and that the semigroup corresponding to the Laplacian is positive and uniformly exponentially stable, we can obtain the bounded $\Hinf$-calculus with $\mu=0$. The proof of this special case for the Dirichlet Laplacian is provided in Corollary \ref{cor:mu=0X} below. However, the proof does not extend to $k\geq 1$. 
      \end{itemize}
     For the general case ($k\geq 0$ and $X$ a $\UMD$ Banach space) we expect that one can show uniform exponential stability for the semigroup corresponding to the Laplacian via (weighted) kernel bounds for the scalar-valued case. Using a tensor extension and consistency, one could also obtain the required kernel bounds for the vector-valued case. 
  \item The $p$-independence of the spectra of the Laplacian on $L^p$-spaces is well-studied. Moreover, in  \cite{Da95,Ku99} it is proved that on certain weighted $L^p$-spaces the spectrum is independent of the weight. 
  However, the power weights $w_{\gam}^{\d\OO}$ that we use do not fit into their settings. Instead, we will use compactness and consistency of the resolvent to obtain the spectral independence in Theorems \ref{thm:scalarDir} and \ref{thm:scalarNeu}.
\end{enumerate}
\end{remark}

\subsection{Consequences of the bounded $\Hinf$-calculus}\label{subsec:MR_Riesz} In this section, we discuss two consequences of the bounded $\Hinf$-calculus for the Laplacian: maximal regularity and boundedness of the Riesz transform.

\subsubsection{Maximal \texorpdfstring{$L^q$}{Lq}-regularity} Let $T\in(0,\infty]$. We study the time-dependent heat equation on $I:=(0,T)$  given by
\begin{equation*}
  \d_t u(t)-\del u(t) =f(t),\qquad t\in I,
\end{equation*}
on a bounded domain $\OO$ with Dirichlet or Neumann boundary conditions and zero initial condition. Furthermore, we consider this in a setting with temporal weights, where we denote by $A_q(I)$ the class of Muckenhoupt weights. For an extensive introduction to maximal regularity, the reader is referred to \cite[Chapter 17]{HNVW24}. 

The following two corollaries on maximal regularity for the heat equation follow immediately from Theorems \ref{thm:Dirichlet_Laplacian}, \ref{thm:Neumann_Laplacian}, \ref{thm:scalarDir}, \ref{thm:scalarNeu} and \cite[Theorems 17.3.18, 17.2.39 \& Proposition 17.2.7]{HNVW24}. 
\begin{corollary}[Maximal regularity for $-\delDir$]\label{cor:MR_domain}
  Assume that the conditions from Theorem \ref{thm:Dirichlet_Laplacian} hold. In addition, let $q\in(1,\infty)$, $T\in(0,\infty)$ and  $v\in A_q(I)$. Then $-\delDir$ on $W^{k,p}(\OO, w^{\d\OO}_{\gam+kp};X)$ has maximal $L^q(v)$-regularity on $I$, i.e., for all $$f\in L^q(I,v; W^{k,p}(\OO, w^{\d\OO}_{\gam+kp};X))$$ there exists a unique $$u\in W^{1,q}(I,v; W^{k,p}(\OO, w^{\d\OO}_{\gam+kp};X))\cap L^q(I,v; W^{k+2,p}_{\Dir}(\OO, w^{\d\OO}_{\gam+kp};X))$$
  such that $\d_t u-\delDir u =f $ with $u(0)=0$ and
  \begin{equation*}
    \|u\|_{W^{1,q}(I,v; W^{k,p}(\OO, w^{\d\OO}_{\gam+kp};X))} + \| u\|_{L^q(I,v; W^{k+2,p}_\Dir(\OO, w^{\d\OO}_{\gam+kp};X))}\lesssim \|f\|_{L^q(I,v; W^{k,p}(\OO, w^{\d\OO}_{\gam+kp};X))},
  \end{equation*}
  where the constant only depends on $p,q,k,\gam,v,T,d$ and $X$.
  Moreover, if $X=\CC$, then the above statement holds for $I=\RR_+$ as well.
\end{corollary}

\begin{corollary}[Maximal regularity for $-\delNeu$]\label{cor:MR_domain_Neu}
  Assume that the conditions from Theorem \ref{thm:Neumann_Laplacian} hold. In addition, let $q\in(1,\infty)$, $T\in(0,\infty)$ and $v\in A_q(I)$. Then $-\delNeu$ on $W^{k+j,p}(\OO, w^{\d\OO}_{\gam+kp};X)$ has maximal $L^q(v)$-regularity on $I$, i.e., for all $$f\in L^q(I,v; W^{k+j,p}(\OO, w^{\d\OO}_{\gam+kp};X))$$ there exists a unique $$u\in W^{1,q}(I,v; W^{k+j,p}(\OO, w^{\d\OO}_{\gam+kp};X))\cap L^q(I,v; W^{k+2+j,p}_{\Neu}(\OO, w^{\d\OO}_{\gam+kp};X))$$
  such that $\d_t u-\delNeu u =f $ with $u(0)=0$ and
  \begin{equation*}
    \|u\|_{W^{1,q}(I,v; W^{k+j,p}(\OO, w^{\d\OO}_{\gam+kp};X))} + \| u\|_{L^q(I,v; W^{k+2+j,p}_\Neu(\OO, w^{\d\OO}_{\gam+kp};X))}\lesssim \|f\|_{L^q(I,v; W^{k+j,p}(\OO, w^{\d\OO}_{\gam+kp};X))},
  \end{equation*}
  where the constant only depends on $p,q,k,\gam,j,v,T,d$ and $X$.  Moreover, the above statement also holds if we consider $\delNeu$ on the spaces without constant functions as in Definition \ref{def:delDirObounded}\ref{it:def:delO3bounded}. In this case, if additionally $X=\CC$, then the statement also holds for $I=\RR_+$.
\end{corollary}

\begin{remark}\hspace{2em}
  \begin{enumerate}[(i)]
  \item Similar results as in Corollaries \ref{cor:MR_domain} and \ref{cor:MR_domain_Neu} for $\OO=\RRdh$ are obtained in \cite[Section 8]{LLRV24}. 
    \item Corollaries \ref{cor:MR_domain} and \ref{cor:MR_domain_Neu} concern the heat equation with zero initial data. Well-posedness for the heat equation with non-zero initial data can be obtained as a consequence, see \cite[Section 4.4]{GV17} and \cite[Section 17.2.b]{HNVW24}.
  \end{enumerate}
\end{remark}

We connect the above results to the existing literature about PDEs on homogeneous weighted Sobolev spaces, see \cite{Kr99c, Kr99b, Lo00}. For $p\in(1,\infty)$, $k\in\NN_0$, $\theta\in \RR$ and $\OO\subseteq \RR^d$ a bounded $C^1$-domain, the homogeneous Sobolev spaces are given by
\begin{equation*}
  H^{k}_{p,\theta}(\OO)=\Big\{f\in \mc{D}'(\OO): \forall |\alpha|\leq k, \d^{\alpha}f\in L^p(\OO,w^{\d\OO}_{\theta-d+|\alpha|p})\Big\},
\end{equation*}
see for instance \cite[Proposition 2.2]{Lo00}. Note that $L^p(\OO, w^{\d\OO}_{\gam})=H^{0}_{p,\gam+d}(\OO)$. In the setting for the Dirichlet Laplacian with $\gam\in(p-1,2p-1)$ we have the following relation between the involved homogeneous and inhomogeneous spaces:
\begin{align*}
  W^{k,p}(\OO, w^{\d\OO}_{\gam+kp}) & = H^{k}_{p,\gam+d}(\OO), \\
  W^{k+2,p}_\Dir(\OO, w^{\d\OO}_{\gam+kp})&= H^{k+2}_{p, \gam+d-2p}(\OO).
\end{align*}
The first characterisation follows from the fact that $\OO$ is bounded and Hardy's inequality using that $\gam+kp>-1$. The second characterisation follows similarly using that $ W^{k+2,p}_\Dir(\OO, w^{\d\OO}_{\gam+kp})= W^{k+2,p}_0(\OO, w^{\d\OO}_{\gam+kp})$ for $\gam\in(p-1,2p-1)$. Note that we have used that the domain is bounded, for unbounded domains the homogeneous and inhomogeneous spaces cannot be compared. 

In \cite{KK04}, the authors use homogeneous spaces to study spatial regularity for boundary value problems with Dirichlet boundary conditions on bounded 
$C^1$-domains. There, the boundary condition is encoded implicitly within the function space. In contrast, our approach imposes boundary conditions explicitly, allowing greater flexibility -- particularly when extending to more regular domains or handling smaller weight exponents and Neumann boundary conditions.
In the homogeneous setting, some results for the Neumann Laplacian on the half-space (in the special case $k=0$) are contained in \cite{DK18, DKZ16}, but a general study on bounded domains seems to be unavailable.

Finally, we remark that maximal $L^q$-regularity for the Dirichlet Laplacian on $L^p(\OO, w^{\d\OO}_{\gam})$ is also obtained in \cite{KN14}. Here they treat bounded $C^{1,\lambda}$-domains with $\gam\in((1-\lambda)p-1,2p-1)$ which corresponds to our result in Corollary \ref{cor:MR_domain} with $k=0$.

\subsubsection{Riesz transforms} In this section, we discuss the boundedness of the Riesz transform associated with the Dirichlet Laplacian on the half-space and bounded domains. For an elaborate study of Riesz transforms associated with the Laplacian on the half-space, the reader is referred to \cite{DHLWY19}.\\

We start with an extension of the $\Hinf$-calculus of $-\delDir$ from scalar-valued Lebesgue spaces to vector-valued Lebesgue spaces, see also Remark \ref{rem:XvaluedLp}. This extends the result in \cite[Theorem 6.1 \& Corollary 6.2]{LV18}.
\begin{corollary}[$\Hinf$-calculus for $-\delDir$ on $L^p(\OO, w^{\d\OO}_{\gam};X)$]\label{cor:mu=0X} Let $p\in(1,\infty)$, $\lambda\in [0,1]$, $\gam\in ((1-\lambda)p-1, 2p-1)\setminus\{p-1\}$ and let $X$ be a $\UMD$ Banach space. Let $\delDir$ on $L^p(\OO, w^{\d\OO}_{\gam};X)$ be as in Definition \ref{def:delDirObounded}. Then the operator $-\delDir$ has a bounded $\Hinf$-calculus with $\om_{\Hinf}(-\delDir)=0$.
\end{corollary}
\begin{proof}
  We define the operators
  \begin{align*}
      \delDir^\CC&:=\delDir \quad\text{ on } L^p(\Dom,w^{\BDom}_{\gam})\quad \text{ and }\\
      \delDir^X&:=\delDir \quad\text{ on } L^p(\Dom,w^{\BDom}_{\gam};X)
    \end{align*}
    as in Definition \ref{def:delDirObounded}. Theorem \ref{thm:scalarDir} implies that $0\in \rho(-\delDir^\CC)$  and it follows from \cite[Proposition K.2.3]{HNVW24} that the analytic semigroup $S_t$ generated by $\delDir^\CC$ is uniformly exponentially stable. Moreover, the resolvent $R(\lambda, \delDir^\CC)$ is positive for $\lambda>0$ (this follows from the $L^2$-case and consistency in Lemma \ref{lem:consis_resol_domain}) and \cite[Theorem VI.1.8]{EN00} yields that $S_t$ is positive. Therefore, by \cite[Theorem 2.1.3]{HNVW16} the operator $S_t\otimes \id_X$ defined by
    \begin{equation*}
      (S_t\otimes \id_X)(f\otimes x):= S_t f\otimes x, \qquad f\in L^p(\OO, w_{\gam}^{\d\OO}),\;x\in X,
    \end{equation*}
    extends to a bounded operator on $L^p(\OO, w_{\gam}^{\d\OO};X)$ with equal operator norm. It is straightforward to verify that $S_t\otimes \id_X$ is generated by $\delDir^X$ and that $R(\lambda, \delDir^{X})(f\otimes x)=(R(\lambda, \delDir^{\CC})f) \otimes x$ for $f\in L^p(\OO, w_{\gam}^{\d\OO})$, $x\in X$ and $\lambda\in \rho(\delDir^\CC)\cap\rho(\delDir^X)$. The semigroup $S_t\otimes \id_X$ is also uniformly exponentially stable, which shows that $-\delDir^X$ is sectorial. Proposition \ref{prop:calc_pert_Id} and Theorem \ref{thm:Dirichlet_Laplacian} now give the desired result.
\end{proof}

We have the following result for the Riesz transform associated with the Dirichlet Laplacian. 
\begin{corollary}[Riesz transform associated with $-\delDir$]\label{cor:Riesz_Dir}
  Let $p\in(1,\infty)$, $\lambda\in[0,1]$ and let $X$ be a $\UMD$ Banach space. Assume that either
  \begin{enumerate}[(i)]
    \item $\OO=\RRdh$, $k=0$, $\gam\in (-1,2p-1)\setminus\{p-1\}$ and $X$ is a $\UMD$ Banach space, or, 
    \item\label{it:corRiesz} $\OO$ is a bounded $C^{1,\lambda}$-domain, $k=0$, $\gam\in((1-\lambda)p-1, 2p-1)\setminus\{p-1\}$ and $X$ is a $\UMD$ Banach space, or,
     \item\label{cor:Riesz_Dir(iii)} $\OO$ is a bounded $C^{1,\lambda}$-domain, $k\in\NN_0$, $\gam\in((1-\lambda)p-1, 2p-1)\setminus\{p-1\}$ and $X=\CC$.
  \end{enumerate}
  Let $\delDir$ on $W^{k,p}(\OO, w^{\d\OO}_{\gam+kp};X)$ be as in Definition \ref{def:delRRdh} or \ref{def:delDirObounded}. Then
  \begin{equation*}
    \|\grad (-\delDir)^{-\half} f\|_{W^{k,p}(\OO, w^{\d\OO}_{\gam+kp};X)}\leq C \|f\|_{W^{k,p}(\OO, w^{\d\OO}_{\gam+kp};X)},\qquad f\in W^{k,p}(\OO, w^{\d\OO}_{\gam+kp};X),
  \end{equation*}
  for some $C>0$ which only depends on $p, k,\gam, \OO$ and $X$.
\end{corollary}

\begin{proof}
First, we claim that
\begin{equation}\label{eq:map_Dir_frac_resol}
  (-\delDir)^{-\half}: W^{k,p}(\OO, w^{\d\OO}_{\gam+kp};X) \to W^{k+1,p}_\Dir(\OO, w^{\d\OO}_{\gam+kp};X)
\end{equation} 
is bounded. Indeed, since $$(-\delDir)^{-1}:W^{k,p}(\OO, w^{\d\OO}_{\gam+kp};X)\to W^{k+2,p}_\Dir(\OO, w^{\d\OO}_{\gam+kp};X)$$ is bounded (see Theorems \ref{thm:LLRVthm1.1Dir}, \ref{thm:scalarDir} and Corollary \ref{cor:mu=0X}) and the identity operator is bounded on $W^{k,p}(\OO, w^{\d\OO}_{\gam+kp};X)$, it holds by Stein interpolation \cite[Theorem 2.1]{Vo92} that 
\begin{equation*}
  (-\delDir)^{-\half}: W^{k,p}(\OO, w^{\d\OO}_{\gam+kp};X) \to [W^{k,p}(\OO, w^{\d\OO}_{\gam+kp};X), W^{k+2,p}_\Dir(\OO, w^{\d\OO}_{\gam+kp};X)]_\half
\end{equation*}
is bounded. To verify the conditions for Stein interpolation, one uses that $-\delDir$ has $\BIP$, which follows again from the bounded $\Hinf$-calculus in Theorem \ref{thm:LLRVthm1.1Dir}, Theorem \ref{thm:scalarDir} and Corollary \ref{cor:mu=0X}. The claim \eqref{eq:map_Dir_frac_resol} now follows from Proposition \ref{prop:complex_int_W_Dir}.

Therefore, \eqref{eq:map_Dir_frac_resol}, Proposition \ref{prop:complex_int_W_Dir} and Proposition \ref{prop:frac_domain} (using that $-\delDir$ has $\BIP$), imply
\begin{align*}
 \|\grad (-\delDir)^{-\half} f\|_{W^{k,p}(\OO, w^{\d\OO}_{\gam+kp};X)} & \leq \| (-\delDir)^{-\half} f\|_{W^{k+1,p}_\Dir(\OO, w^{\d\OO}_{\gam+kp};X)}\\
 &\eqsim \| (-\delDir)^{-\half} f\|_{[W^{k,p}(\OO, w^{\d\OO}_{\gam+kp};X), W^{k+2,p}_\Dir(\OO, w^{\d\OO}_{\gam+kp};X)]_\half}\\
 &\eqsim \| (-\delDir)^{-\half} f\|_{D((-\delDir)^\frac{1}{2})}\lesssim \|f\|_{W^{k,p}(\OO, w^{\d\OO}_{\gam+kp};X)}.
\end{align*}
This completes the proof.
\end{proof}

\begin{remark}\hspace{2em}
\begin{enumerate}[(i)]
\item Boundedness of the Riesz transforms on $L^p(\RR^d, w;X)$ holds if and only if $w\in A_p(\RRd)$, see \cite[Sections 7.4.3 \& 7.4.4]{Gr14_classical_3rd}. Corollary \ref{cor:Riesz_Dir} also allows for weights outside the class of Muckenhoupt weights. On the other hand, we are restricted to power weights since the interpolation results from Proposition \ref{prop:complex_int_W_Dir} are only available for this type of weights.
  \item With the same proof as in Corollary \ref{cor:Riesz_Dir} and using Theorems \ref{thm:Dirichlet_Laplacian} and \ref{thm:Neumann_Laplacian} it follows that the Riesz transforms associated with $\mu-\delDir$ and $\mu-\delNeu$ are bounded on weighted vector-valued Sobolev spaces for $\mu$ large enough. 
      Following the proof of Corollary \ref{cor:mu=0X}, we could also obtain the bounded $\Hinf$-calculus for $-\delNeu$ on $L^p(\OO, w_{\gam}^{\d\OO};X)/\{c\ind_\OO: c\in X\}$. 
      \item In view of Remark \ref{rem:XvaluedLp}\ref{rem:XvaluedLp1}, the condition in Corollary \ref{cor:Riesz_Dir}\ref{cor:Riesz_Dir(iii)} on the space $X$ can be weakened to $X$ being a Hilbert space or being isomorphic to a closed subspace of an $L^p$-space.
\end{enumerate}

\end{remark}

\subsection{The proofs of Theorems \ref{thm:Dirichlet_Laplacian} and \ref{thm:Neumann_Laplacian}}\label{sec:proof23}
To transfer the $\Hinf$-calculus on special domains (Section \ref{sec:calc_spec_dom}) to bounded domains, we employ a localisation procedure based on the decomposition of weighted Sobolev spaces as in Lemma \ref{lem:decomp}.
For this localisation of the $\Hinf$-calculus, we need the following abstract lemma, which follows from lower order perturbation results.
\begin{lemma}[{\cite[Lemma 6.11]{LV18}}]\label{lem:LV6.11}
  Let $A$ be a linear operator on a Banach space $Y$ and let $\tilde{A}$ be a sectorial operator on a Banach space $\tilde{Y}$ with a bounded $\Hinf$-calculus. Assume that there exist bounded linear mappings $\II\colon Y\to \tilde{Y}$ and $\PP\colon \tilde{Y}\to Y$ satisfying
  \begin{enumerate}[(i)]
    \item \label{it:lemLV6.11_1} $\PP \II = \id$,
    \item \label{it:lemLV6.11_2} $\II D(A)\subseteq D(\tilde{A})$ and $\PP D(\tilde{A})\subseteq D(A)$,
    \item \label{it:lemLV6.11_3} $(\II A - \tilde{A}\II)\PP\colon D(\tilde{A})\to \tilde{Y}$ and $\II(A\PP-\PP\tilde{A})\colon D(\tilde{A})\to \tilde{Y}$ extend to bounded linear operators $[\tilde{Y},D(\tilde{A})]_{\theta}\to \tilde{Y}$ for some $\theta\in (0,1)$.
  \end{enumerate}
  Then $A$ is a closed and densely defined operator and for every $\sigma>\om_{\Hinf}(\tilde{A})$ there exists a $\mu>0$ such that $\mu+A$ has a bounded $\Hinf$-calculus with $\om_{\Hinf}(\mu+A)\leq \sigma$.
\end{lemma}

We now turn to the proofs of Theorems \ref{thm:Dirichlet_Laplacian} and \ref{thm:Neumann_Laplacian} concerning the $\Hinf$-calculus on bounded domains.

\begin{proof}[Proof of Theorems \ref{thm:Dirichlet_Laplacian} and \ref{thm:Neumann_Laplacian}] 
We start with the proof for the Dirichlet Laplacian. Let $\lambda\in[0,1]$, $\gam\in((1-\lambda)p-1, 2p-1)\setminus\{p-1\}$ and let $\OO$ be a bounded $C^{1,\lambda}$-domain. Define $A:=-\delDir$ on $W^{k,p}(\OO,w_{\gam+kp}^{\d\OO};X)$.
We show that the operator $\mu-\delDir$ has a bounded $\Hinf$-calculus for $\mu$ sufficiently large. 

If $\lambda=0$, then take $(V_n)_{n=1}^N, (\OO_n)_{n=1}^N,(\eta_n)_{n=0}^N$ from Lemma \ref{lem:decomp} such that for all $n\in\{1,\dots, N\}$ we have $[\OO_n]_{C^1}< \delta$ where $\delta\in(0,1)$ is small enough such that Theorem \ref{thm:Dirichlet_Laplacian_special} applies for every $\OO_n$.
If $\lambda\in (0,1]$, then let $\eps\in (0,\lambda)$ be such that $\gam>(1-(\lambda-\eps))p-1$. Take $(V_n)_{n=1}^N, (\OO_n)_{n=1}^N, (\eta_n)_{n=0}^N$ from Lemma \ref{lem:decomp} such that for all $n\in\{1,\dots, N\}$ we have $[\OO_n]_{C^{1,\lambda-\eps}}< \delta$ where $\delta\in(0,1)$ is small enough such that Theorem \ref{thm:Dirichlet_Laplacian_special} (applied with $\lambda$ replaced by $\lambda-\eps$) applies for every $\OO_n$.
We define the following operators
  \begin{enumerate}[(i)]
  \item $\tilde{A}:=\bigoplus_{n=0}^N\tilde{A}_n$ on $\WW^{k,p}_{\gam+kp}$ as defined in \eqref{eq:Fk}, where
    \begin{enumerate}
      \item $\tilde{A}_0$ on $W^{k,p}(\RRd;X)$ with $D(\tilde{A}_0):=W^{k+2,p}(\RRd;X)$ is given by $\tilde{A}_0 \tilde{u}:=\del \tilde{u} $,
      \item $\tilde{A}_n$ on $W^{k,p}(\OO_n,w_{\gam+kp}^{\d\OO_n};X)$ with $D(\tilde{A}_n):=W^{k+2,p}_{\Dir}(\OO_n,w_{\gam+kp}^{\d\OO_n};X)$ is given by $\tilde{A}_n \tilde{u}:=\delDir \tilde{u} $ for $n\in\{1,\dots, N\}$,
    \end{enumerate}
    \item $B:D(A)\to \WW^{k,p}_{\gam+kp}$ given by $Bu:=([\del,\eta_n]u)_{n=0}^N$,
    \item $C:D(\tilde{A})\to W^{k,p}(\OO, w_{\gam+kp}^{\d\OO};X)$ given by $C\tilde{u}:=\sum_{n=0}^N[\del,\eta_n]\tilde{u}$.
  \end{enumerate}
  Let $\mu>0$. By \cite[Lemma 2.6]{LLRV24}, Proposition \ref{prop:calc_pert_Id} and Theorem \ref{thm:Dirichlet_Laplacian_special} it holds that $\mu-\tilde{A}_n$ for any $n\in\{ 0,\dots, N\}$ has a bounded $\Hinf$-calculus with $\om_{\Hinf}(\mu-\tilde{A}_n)=0$. Thus $\mu-\tilde{A}$ has a bounded $\Hinf$-calculus with $\om_{\Hinf}(\mu-\tilde{A})=0$ as well.

  Let $\PP$ and $\II$ be as defined in \eqref{eq:retraction}. It is straightforward to verify that the conditions \ref{it:lemLV6.11_1} and \ref{it:lemLV6.11_2} from Lemma \ref{lem:LV6.11} hold. It remains to check condition \ref{it:lemLV6.11_3} in Lemma \ref{lem:LV6.11}. From Proposition \ref{prop:complex_int_W_Dir} we obtain
  \begin{equation*}
     [W^{k,p}(\OO_n, w_{\gam+kp}^{\d\OO_n};X), D(\tilde{A}_n)]_{\half} = W^{k+1,p}_{\Dir}(\OO_n,w_{\gam+kp}^{\d\OO_n};X)\quad \text{ for }n\in\{1,\dots,N\},
  \end{equation*}
  and in combination with (see \cite[Theorems 5.6.9 \& 5.6.11]{HNVW16})
  \begin{equation*}
     [W^{k,p}(\RRd;X), D(\tilde{A}_0)]_{\half} = W^{k+1,p}(\RRd;X),
  \end{equation*}
  this yields
  \begin{equation}\label{eq:intPI}
      \begin{aligned}
    [\WW^{k,p}_{\gam+kp}, D(\tilde{A})]_{\half} &=[W^{k,p}(\RRd;X), D(\tilde{A}_0)]_{\half} \oplus\bigoplus_{n=1}^N [W^{k,p}(\OO_n, w_{\gam+kp}^{\d\OO_n};X), D(\tilde{A}_n)]_{\half} \\
     & =W^{k+1,p}(\RRd;X)\oplus\bigoplus_{n=1}^N W^{k+1,p}_{\Dir}(\OO_n,w_{\gam+kp}^{\d\OO_n};X).
  \end{aligned}
  \end{equation}
 Note that
  \begin{equation*}
    \II A u-\tilde{A}\II u=- Bu,\quad u\in D(A),\quad \text{ and }\quad A\PP\tilde{u}-\PP\tilde{A}\tilde{u}=C\tilde{u},\quad \tilde{u}\in D(\tilde{A}),
  \end{equation*}
and every commutator $[\del, \eta_n]$ is a first-order partial differential operator with smooth and compactly supported coefficients. This and \eqref{eq:intPI} yield that
\begin{align*}
  \II A-\tilde{A}\II&\colon W^{k+1,p}_{\Dir}(\OO, w_{\gam+kp}^{\d\OO};X)\to \WW^{k,p}_{\gam+kp}\quad \text{ and }\\
  \PP&\colon [\WW^{k,p}_{\gam+kp}, D(\tilde{A})]_{\half} \to W^{k+1,p}_{\Dir}(\OO, w_{\gam+kp}^{\d\OO};X)
\end{align*}
are bounded.  Similarly, we obtain by \eqref{eq:intPI} that
\begin{align*}
  A\PP-\PP\tilde{A}&\colon [\WW^{k,p}_{\gam+kp}, D(\tilde{A})]_{\half} \to  W^{k,p}(\OO, w_{\gam+kp}^{\d\OO};X)\quad \text{ and }\\
  \II&\colon W^{k,p}(\OO, w_{\gam+kp}^{\d\OO};X) \to \WW^{k,p}_{\gam+kp}
\end{align*}
are bounded. This shows that $(\II A - \tilde{A}\II)\PP$ and $\II(A\PP-\PP\tilde{A})$ extend to bounded operators from
 $[\WW^{k,p}_{\gam+kp}, D(\tilde{A})]_{\half}$ to $\WW^{k,p}_{\gam+kp}$. Applying Lemma \ref{lem:LV6.11} gives that for all $\sigma\in(0,\pi)$ there exists a $\tilde{\mu}>0$ such that for all $\mu>\tilde{\mu}$ the operator $\mu-\delDir$ on $W^{k,p}(\OO, w_{\gam+kp}^{\d\OO};X)$ has a bounded $\Hinf$-calculus with $\om_{\Hinf}(\mu-\delDir)\leq \sigma$. 
 
 The boundedness of the $\Hinf$-calculus for the Neumann Laplacian on $W^{k+j,p}(\OO, w^{\d\OO}_{\gam+kp};X)$ can be shown similarly as for the Dirichlet Laplacian using Theorem \ref{thm:Neumann_Laplacian_special} and Proposition \ref{prop:complex_int_W_Neu}. 
 
 It remains to prove the boundedness of the $\Hinf$-calculus for $\mu-\delNeu$ on the quotient space $Y/K:= W^{k+j,p}(\OO, w^{\d\OO}_{\gam+kp};X)/\{c\ind_\OO: c\in X\}$. Fix $\sigma\in(0,\pi)$ and let $\mu$ be large enough such that $\mu-\delNeu$ on $W^{k+j,p}(\OO, w^{\d\OO}_{\gam+kp};X)$ has a bounded $\Hinf$-calculus of angle $\om_{\Hinf}(\mu-\delNeu)\leq \sigma$. Let $\om\in(\sigma,\pi)$ and let $\ph\in H^1(\Sigma_{\om})\cap H^\infty(\Sigma_\om)$. For any $c\in K$ we have that $x\in Y/K$ can be represented as $x = y + c$ with $y\in Y$. Note that for $z\in \rho(\mu-\delNeu)$ the equation
 \begin{equation*}
    zu -(\mu-\delNeu)u = c
 \end{equation*}
 has the unique solution $u=c/(z-\mu)$. Therefore, by definition of the functional calculus and Cauchy's integral formula, we obtain 
 \begin{equation}\label{eq:ph_c}
 \begin{aligned}
   \ph(\mu-\delNeu)c &=\frac{1}{2\pi \ii}\int_{\d\Sigma_{\nu}} \ph(z) R(z, \mu-\delNeu)c\dd z\\
    &=\frac{1}{2\pi \ii}\int_{\d\Sigma_{\nu}} \frac{\ph(z)c}{z-\mu}\dd z = \ph(\mu) c\in K,\quad \nu\in (\sigma, \om).
   \end{aligned}
 \end{equation}
By \eqref{eq:ph_c} and the bounded $\Hinf$-calculus for $\mu-\delNeu$ on $Y$, it follows that for $x\in Y/K$ and $c\in K$ we have
 \begin{align*}
  \|(\ph (\mu-\delNeu) x) - \ph(\mu)c\|_Y  & = \|(\ph (\mu-\delNeu) (y+c))-\ph(\mu)c\|_Y = \|\ph(\mu-\delNeu)y \|_Y\\&\lesssim \|\ph\|_{\Hinf(\Sigma_\om)} \|y\|_Y = \|\ph\|_{\Hinf(\Sigma_\om)} \|x-c\|_Y.
 \end{align*}
 Taking the infimum over $c\in K$ yields that $ \|\ph (\mu-\delNeu) x\|_{Y/K}\lesssim \|\ph\|_{\Hinf(\Sigma_\om)}\|x\|_{Y/K}$ for $x\in Y/K$, which proves the boundedness of the $\Hinf$-calculus on $Y/K$ with angle $\om_{\Hinf}(\mu-\delNeu)\leq \sigma$.
\end{proof}

\subsection{The proofs of Theorems \ref{thm:scalarDir} and \ref{thm:scalarNeu}}\label{sec:cor_X=C} 
We continue with the proof of Theorems \ref{thm:scalarDir} and \ref{thm:scalarNeu}, which deal with the $\Hinf$-calculus in the special case of $X=\CC$. We start with some preliminary results about the consistency of resolvents.\\

Let $X_0$ and $X_1$ be two compatible Banach spaces and suppose that $B_{0}\in\mc{L}(X_0)$ and $B_{1}\in \mc{L}(X_1)$. Then we call the operators $B_{0}$ and $B_{1}$ \emph{consistent} if
\begin{equation*}
  B_{0}u=B_{1}u\qquad \text{ for all }u\in X_0\cap X_1.
\end{equation*}
For $z\in \Sigma \subseteq \CC$ the two families of operators $B_{0}(z)\in\mc{L}(X_0)$ and $B_{1}(z)\in\mc{L}(X_1)$ are called consistent if $B_{0}(z)$ and $B_{1}(z)$ are consistent for all $z\in \Sigma$.\\

We introduce the forms on the Hilbert spaces $V$ (as dense subspace of $L^2(\OO)$) given by
\begin{align*}
      a_{\Dir}(v_1,v_2)&:=\int_{\OO} \grad v_1 \cdot \overline{\grad v_2} \dd x,\qquad v_1,v_2\in V= W^{1,2}_0(\OO),\\
      a_{\Neu}(v_1,v_2)&:=\int_{\OO} \grad v_1 \cdot \overline{\grad v_2} \dd x,\qquad v_1,v_2\in V= W^{1,2}(\OO).
    \end{align*}
    Associated with the forms $a_{\Dir}$ and $a_{\Neu}$ are the densely defined closed Laplace operators $-A_{\Dir,2}$ and $-A_{\Neu,2}$ on $L^2(\OO)$, respectively, see for instance \cite[Chapter 12]{Ne22}. The domains of these operators are 
    \begin{align*}
      D(A_{\Dir,2}) &=\{f\in W^{1,2}_0(\OO)\cap W^{2,2}_\loc(\OO): \del f\in L^2(\OO)\},\\
     D(A_{\Neu,2}) &=\{f\in W^{1,2}(\OO)\cap W^{2,2}_\loc(\OO): \del f\in L^2(\OO)\},
    \end{align*}
    see \cite[Sections 12.3.b \& 12.3.c]{Ne22}. A characterisation of the domains as a closed subspace of $W^{2,2}(\OO)$ requires more regularity of the domain (compared to the regularity we consider in Theorems \ref{thm:scalarDir} and \ref{thm:scalarNeu}), see \cite[Sections 12.3.b \& 12.3.c]{Ne22}. For instance, for the Dirichlet Laplacian, $C^2$-regularity is required.\\

We have the following lemma on the consistency of the resolvents for the Dirichlet Laplacian. 
\begin{lemma}\label{lem:consis_resol_domain}
   Let $p\in(1,\infty)$, $k\in\NN_0$, $\lambda\in[0,1]$, $\gam\in ((1-\lambda)p-1,2p-1)\setminus\{p-1\}$ and let $\OO$ be a bounded $C^{1,\lambda}$-domain. 
   Let
    \begin{align*}
      A_{p,k,\gam}&:=\delDir \quad\text{ on } W^{k,p}(\Dom,w^{\BDom}_{\gam+kp})\text{ with }D(A_{p,k,\gam}) = W^{k+2,p}_\Dir(\OO,w^{\d\OO}_{\gam+kp})
    \end{align*}
     be as in Definition \ref{def:delDirObounded} and let 
    \begin{equation*}
      A_{\Dir,2}=\delDir \quad\text{ on } L^2(\OO) \text{ with }D(A_{\Dir,2}) = \{f\in W^{1,2}_0(\OO)\cap W^{2,2}_\loc(\OO): \del f\in L^2(\OO)\}
    \end{equation*}
    be as above.
     Then there exists a $\tilde{\mu}>0$ such that for all $\mu>\tilde{\mu}$ the resolvents $R(\mu, A_{p,k,\gam})$ and $R(\mu, A_{\Dir, 2})$ are consistent.
\end{lemma}
\begin{proof}
  Take $1<q<\min\{p,2\}$ and $\kappa\in (0,2q-1)\setminus\{q-1\}$ such that
\begin{equation}\label{eq:kappa}
  \kappa > \frac{q(\gam+1)}{p} -1 > (1-\lambda)q-1. 
\end{equation}
First, we claim that $L^p(\OO, w^{\d\OO}_{\gam})\hookrightarrow L^q(\OO, w^{\d\OO}_{\kappa})$. Indeed, for $u\in L^p(\OO, w^{\d\OO}_{\gam})$ we have by H\"older's inequality that
\begin{equation*}
  \int_\OO |u(x)|^q w^{\d\OO}_{\kappa}(x)\dd x \leq \Big(\int_\OO |u(x)|^p w^{\d\OO}_{\gamma}(x)\dd x\Big)^{\frac{q}{p}}\Big(\int_\OO w^{\d\OO}_{\frac{\kappa p-q\gamma}{p-q}}(x)\dd x\Big)^{\frac{p-q}{p}}<\infty.
\end{equation*}
The latter integral can be written as an integral over $\RRdh$ (using localisation from Lemma \ref{lem:decomp} and the diffeomorphism from Lemma \ref{lem:localization_weighted_blow-up}), hence the integral is finite since  \eqref{eq:kappa} implies $(\kappa p -q\gam)/(p-q)>-1$. This proves the claim.

To continue, we introduce the space
\begin{equation*}
  Z_{r,\nu}:= \{f\in W^{1,r}_0(\OO, w^{\d\OO}_{\nu})\cap W^{2,r}_{\loc}(\OO): \del f\in L^r(\OO,w^{\d\OO}_{\nu})\}\quad \text{ for } r\in(1,2],\, \nu>-1,
\end{equation*} 
and note that $D(A_2)=Z_{2,0}$. Now, consider the equation
\begin{equation}\label{eq:resolveeq_domains}
    \mu u-\delDir u =f ,\qquad f\in W^{k,p}(\Dom,w^{\BDom}_{\gam+kp})\cap  L^2(\OO).
\end{equation}
By Theorem \ref{thm:Dirichlet_Laplacian} (using that $\gam>(1-\lambda)p-1$) and \cite[Section 12.3.b]{Ne22} there exist unique
\begin{equation*}
   u_0\in W^{k+2,p}_{\Dir}(\Dom,w^{\BDom}_{\gam+kp})\quad \text{ and }\quad u_1\in Z_{2,0}
\end{equation*}
solving \eqref{eq:resolveeq_domains} for $\mu$ sufficiently large. By Hardy's inequality (for bounded Lipschitz domains, see for instance \cite[Section 8.8]{Ku85}) and the claim, we have
\begin{equation*}
    W^{k+2,p}_{\Dir}(\Dom,w^{\BDom}_{\gam+kp})\hookrightarrow W^{2,p}_\Dir(\OO, w^{\d\OO}_{\gam})\hookrightarrow W^{2,q}_\Dir(\OO, w^{\d\OO}_{\kappa}).
\end{equation*}
Moreover, using $\kappa>0$, $q<2$ and elliptic regularity (Theorem \ref{thm:Dirichlet_Laplacian} using \eqref{eq:kappa}), we have 
\begin{equation*}
  Z_{2,0}\hookrightarrow Z_{q,\kappa} = W^{2,q}_\Dir(\OO, w^{\d\OO}_{\kappa}).
\end{equation*}
Note that the equation \eqref{eq:resolveeq_domains} with right-hand side $f\in L^q(\OO, w^{\d\OO}_{\kappa})$ has a unique solution in $W^{2,q}_\Dir(\OO, w^{\d\OO}_{\kappa})$ by Theorem \ref{thm:Dirichlet_Laplacian} (using \eqref{eq:kappa}). It follows that $u_0=u_1$, which proves that the resolvents of $A_{p,k,\gam}$ and $A_{2}$ are consistent. 
\end{proof}

For the Neumann Laplacian, we have the following result concerning the consistency of resolvents. Its proof is similar to the proof of Lemma \ref{lem:consis_resol_domain}.
\begin{lemma}\label{lem:consis_resol_domainNeu}
   Let $p\in(1,\infty)$, $k\in\NN_0$, $\lambda\in[0,1]$, $\gam\in ((1-\lambda)p-1,p-1)$, $j\in\{0,1\}$ and let $\OO$ be a bounded $C^{j+1,\lambda}$-domain.  Let
    \begin{align*}
      A_{p,k,j,\gam}&:=\delNeu \quad\text{ on } W^{k+j,p}(\Dom,w^{\BDom}_{\gam+kp})\text{ with }D(A_{p,k,j,\gam}) = W^{k+j+2,p}_\Neu(\OO,w^{\d\OO}_{\gam+kp})
    \end{align*}
     be as in Definition \ref{def:delDirObounded}\ref{it:def:delO2bounded} and let 
    \begin{equation*}
      A_{\Neu,2}=\delNeu \quad\text{ on } L^2(\OO) \text{ with }D(A_{\Neu,2}) = \{f\in W^{1,2}(\OO)\cap W^{2,2}_\loc(\OO): \del f\in L^2(\OO)\}
    \end{equation*}
    be as above.
     Then there exists a $\tilde{\mu}>0$ such that for all $\mu>\tilde{\mu}$ the resolvents $R(\mu, A_{p,k,j,\gam})$ and $R(\mu, A_{\Neu,2})$ are consistent.
\end{lemma}

We can now turn to the $\Hinf$-calculus on scalar-valued spaces.
\begin{proof}[Proof of Theorems \ref{thm:scalarDir} and \ref{thm:scalarNeu}]
 We start with the proof of Theorem \ref{thm:scalarDir}\ref{it:thmscalar1}. Since the embedding $W^{1,p}(\OO, w_{\gam+kp}^{\d\OO})\hookrightarrow L^p(\OO, w_{\gam+kp}^{\d\OO})$ is compact, see \cite[Theorem 8.8]{GO89}, we have 
 \begin{equation*}
   D(\delDir)=W^{k+2,p}_{\Dir}(\OO, w_{\gam+kp}^{\d\OO})\hookrightarrow W^{k+1,p}(\OO, w_{\gam+kp}^{\d\OO})\stackrel{\text{compact}}{\hookrightarrow} W^{k,p}(\OO, w_{\gam+kp}^{\d\OO}).
 \end{equation*}
Since $(\mu-\delDir)^{-1}$ with $\mu\in \rho(\delDir)$ exists (by Theorem \ref{thm:Dirichlet_Laplacian}), the compact embedding above implies that $(\mu-\delDir)^{-1}$ is compact. Thus by the Riesz--Schauder theorem for compact operators, the resolvent operator $(\mu-\delDir)^{-1}$ has a discrete countable spectrum $\{\sigma_j:j\in \NN_0\}$, where $\sigma_j\neq 0$ are eigenvalues of $(\mu-\delDir)^{-1}$. Moreover, zero is in the spectrum of $(\mu-\delDir)^{-1}$ and is the only accumulation point of the spectrum. Therefore, by the spectral mapping theorem 
 \begin{equation*}
   \sigma(-\delDir)=\{\mu_j: \mu_j=\sigma_j^{-1}-\mu, j\in\NN_0\text{ with }\sigma_j\neq 0\}.
 \end{equation*}
 
 Next, we claim that the spectrum $\sigma(-\delDir)$ is independent of $p\in(1,\infty)$, $k\in\NN_0$ and $\gam\in ((1-\lambda)p-1,2p-1)\setminus\{p-1\}$. Let $A_{p,k,\gam}$ and $A_2$ be as in Lemma \ref{lem:consis_resol_domain}.
    It suffices to show that $\sigma(-A_{p,k,\gam})=\sigma(-A_2)$. We proceed as in the proof of \cite[Proposition 2.6]{Ar94}. Recall that $\sigma(-A_2)$ is discrete and only consists of a countable number of positive eigenvalues, see \cite[Theorem 12.26]{Ne22}. 
By Lemma \ref{lem:consis_resol_domain} and analytic continuation we find that $R(z, -A_2)$ and $R(z,-A_{p,k,\gam})$ are consistent for all $z\in \rho(-A_2)\cap\rho(-A_{p,k,\gam})$. Now, if $\mu\in \rho(-A_2)$, then since $\sigma(-A_{p,k,\gam})$ is discrete and countable it follows that
there exists an $r>0$ such that $\overline{B(\mu, r)}\setminus\{\mu\}\subseteq \rho(-A_2)\cap\rho(-A_{p,k,\gam})$. Therefore, by consistency of the resolvents we obtain
\begin{equation*}
  \int_{\d B(\mu, r)} R(z, -A_{p,k,\gam})\dd z = \int_{\d B(\mu, r)} R(z, -A_2)\dd z = 0,
\end{equation*}
 and thus $\mu\in \rho(-A_{p,k,\gam})$. The other inclusion follows similarly. This proves that $\sigma(-A_{p,k,\gam})=\sigma(-A_2)$ and the claim follows.
 
Finally, using that $\sigma(-A_2)$ is discrete, $\sigma(-A_2)\subseteq [\tilde{\mu},\infty)\subseteq (0,\infty)$ with $\tilde{\mu}:=\min\{\mu_j:j\in\NN_0\}>0$ and the claim gives that  $\sigma(-A_{p,k,\gam})$ is discrete and $\sigma(-A_{p,k,\gam})\subseteq [\tilde{\mu},\infty)\subseteq (0,\infty)$. This completes the proof of Theorem \ref{thm:scalarDir}\ref{it:thmscalar1}.

We continue with the proof of  Theorem \ref{thm:scalarDir}\ref{it:thmscalar2}. From Theorem \ref{thm:Dirichlet_Laplacian} we have that for fixed $\sigma \in(0,\pi)$ and $\mu$ sufficiently large, $\mu-\delDir$ is sectorial with $\om(\mu-\delDir)\leq \sigma$. Combining this with the analyticity of $z\mapsto (z-\delDir)^{-1}$ on $\CC\setminus(-\infty, -\tilde{\mu}]$ yields that for $\mu>-\tilde{\mu}$ and $\sigma'> \sigma$ the operator $\mu-\delDir$ is sectorial with $\om(\mu-\delDir)\leq \sigma'$. Therefore, Theorem \ref{thm:scalarDir}\ref{it:thmscalar2} follows from Proposition \ref{prop:calc_pert_Id}, Theorem \ref{thm:Dirichlet_Laplacian} and the fact that $\sigma\in(0,\pi)$ is arbitrary.

The proof of Theorem \ref{thm:scalarNeu} for the Neumann Laplacian is similar to the proof for the Dirichlet Laplacian above if we use Theorem \ref{thm:Neumann_Laplacian} and Lemma \ref{lem:consis_resol_domainNeu}. Note that for the Neumann Laplacian on $L^2(\OO)$, zero is an eigenvalue and the corresponding eigenspace consists of constant functions, see \cite[Proposition 12.24 \& Theorem 12.26]{Ne22}. Therefore, we obtain the bounded $\Hinf$-calculus for $\mu-\delNeu$ with $\mu>0$ on $W^{k+j,p}(\OO,w^{\d\OO}_{\gam+kp})$. In addition, on $W^{k+j,p}(\OO,w^{\d\OO}_{\gam+kp})/\{c\ind_\OO: c\in X\}$ the eigenvalue zero is removed and we obtain the  bounded $\Hinf$-calculus for $\mu-\delNeu$ with $\mu >-\tilde{\mu }$, where
$\tilde{\mu}$ is the first positive eigenvalue of $-\Delta _{\operatorname{Neu}}$.
\end{proof}

\appendix
\section{Localisation techniques for rough domains}\label{sec:appendix_lemma} 

In this appendix, we will construct diffeomorphisms from special domains $\OO:=\{x\in \RRd: x_1>h(\tilde{x})\}$ to the half-space $\RRdh$. In the literature, such diffeomorphisms are frequently used for smooth domains and in this appendix, we will construct modifications of such diffeomorphisms for rough domains.\\

We start with some definitions. Throughout this appendix, we consider $d\geq 2$.
\begin{definition} Let $U,V\subseteq \RRd$, $\ell\in\NN_1$ and $\lambda\in[0,1]$. A map $\Phi:U\to V$ is called a \emph{$C^{\ell,\lambda}$-diffeomorphism} if $\Phi$ is a bijection, $\Phi\in C^{\ell,\lambda}(U;\RRd)$ and $\Phi^{-1}\in C^{\ell,\lambda}(V;\RRd)$.
\end{definition}
To be able to deal with boundary value problems involving the normal derivative on the boundary, we need a diffeomorphism that preserves the direction of the normal vector. This is the notion of admissibility, see \cite[Section 2.4]{Wl87}.
\begin{definition}
Let $U$ and $V$ be $C^1$-domains, $\ell\in\NN_1$ and $\lambda\in[0,1]$.
    A $C^{\ell, \lambda}$-diffeomorphism $\Phi:U\to V$ is called \emph{admissible at the point $x\in \d U$} if the push-forward $D\Phi$ corresponding to $\Phi$ has the following properties:
    \begin{enumerate}[(i)]
        \item $D\Phi$ maps the tangent space $T_{\d U,x}$ to $T_{\d V, \Phi(x)}$;
        \item $D\Phi$ maps the inner normal $\nu(x)$ of $\d U$ to the inner normal $\nu(\Phi(x))$ of $\d V$.
    \end{enumerate}
\end{definition}
Note that in the definition above, we do not require that the length of the normal vector is preserved. Furthermore, it holds that $\Phi$ is admissible at $x\in \d U$ if and only if $\Phi^{-1}$ is admissible at $y=\Phi(x)\in \d V$. 

We note that for special $C^1$-domains of the form $\OO:=\{x\in \RR^d: x_1>h(\tilde{x})\}$ for some $h\in C^{1}(\RR^{d-1})$, the inner normal direction at the boundary of $\OO$ is $\nu(x)=\nu(\tilde{x})=(1,-\grad_{\tilde{x}}h(\tilde{x}))^\top$ and the tangent space is spanned by the vectors
\begin{equation*}
    (\d_{x_2}h(\tilde{x}), 1, 0,\dots, 0)^{\top}, \dots, (\d_{x_d}h(\tilde{x}), 0, 0,\dots, 1)^{\top}.
\end{equation*}

\begin{example}\label{ex:diff} Let $\ell\in \NN_1$ and $\lambda\in[0,1]$. 
    For a special domain $\OO$ corresponding to $h\in \Cc^{\ell,\lambda}(\RR^{d-1})$, the following two \emph{classical} diffeomorphisms are well known in the literature.
\begin{enumerate}[(i)]
    \item\label{it:diff1} The most frequently used diffeomorphism $\Phi_{\cl}:\overline{\OO}\to \overline{\RRdh}$ is given by 
\begin{align*}
    \Phi_{\cl}(x)&= (x_1-h(\tilde{x}), \tilde{x}), \qquad x=(x_1,\tilde{x})\in \overline{\OO},\\
    \Phi_{\cl}^{-1}(y)&=(y_1+h(\tilde{y}), \tilde{y}),\qquad y=(y_1, \tilde{y})\in \overline{\RRdh},
\end{align*}
see, e.g., \cite{DHP03, Ev10, KrBook08}. This is a $C^{\ell, \lambda}$-diffeomorphism which, in general, is not admissible for all $x\in \d\OO$. This can be seen, e.g., from the push-forward of $\Phi_{\cl}^{-1}$, which is given by
\begin{equation*}
    D\Phi_{\cl}^{-1}(y) = \begin{pmatrix}
        1& (\grad_{\tilde{y}} h(\tilde{y}))^\top\\ 0& I_{d-1}
    \end{pmatrix},\qquad y\in\overline{\RRdh}.
\end{equation*}
\item\label{it:diff2} Let $\ell\geq 2$. Following the construction in \cite[Theorem 2.12]{Wl87}, one can obtain a less regular diffeomorphism that is admissible for all points at the boundary. Let $$\nu(x)=(1, \tilde{\nu}(\tilde{x}))=(1, -\grad_{\tilde{x}}h(\tilde{x}))$$ be the inner normal direction at the boundary of $\OO$. Consider $\overline{\Psi}_{\cl}:\overline{\RRdh}\to \overline{\OO}$ given by
\begin{equation*}
    \overline{\Psi}_{\cl}(y)= (y_1 + h(\tilde{y}), \tilde{y}+ y_1 \tilde{\nu}(\tilde{y})).
\end{equation*}
The push-forward of $\overline{\Psi}_{\cl}$ at the boundary $y_1=0$ is
\begin{equation}\label{eq:Psicl}
\begin{aligned}
      D\overline{\Psi}_{\cl}(y)|_{y_1=0} &= \begin{pmatrix}
        1 & (\grad_{\tilde{y}} h(\tilde{y}))^{\top}\\
        \tilde{\nu}(\tilde{y}) & I_{d-1} + y_1 D_{\tilde{y}}\tilde{\nu}(\tilde{y})
    \end{pmatrix}\bigg|_{y_1=0}\\
    &= 
    \begin{pmatrix}
        1 & (\grad_{\tilde{y}} h(\tilde{y}))^{\top}\\
        -\grad_{\tilde{y}} h(\tilde{y})& I_{d-1}
    \end{pmatrix},
\end{aligned}
\end{equation}
for all $\tilde{y}\in \RR^{d-1}$. Hence, for $y_1=0$ we have $|\det D\overline{\Psi}_{\cl}(y)|\geq 1$ and thus there exists a $\delta>0$ such that $|\det D\overline{\Psi}_{\cl}(y)|\geq \frac{1}{2}>0$ for all $y\in\overline{\RRdh}$ with $0\leq y_1\leq \delta$. By the inverse function theorem, there exists an inverse $\Psi_\cl$ to $\overline{\Psi}_{\cl}$ on this strip. Then $\Psi_{\cl}$ is a $C^{\ell-1, \lambda}$-diffeomorphism since $\tilde{\nu}\in C^{\ell-1, \lambda}(\RR^{d-1})$. Moreover, \eqref{eq:Psicl} shows that $\Psi_\cl$ is an admissible $C^{\ell-1, \lambda}$-diffeomorphism for all $x\in \d\OO$. In the case that $\ell=1$, we obtain that $\Psi_\cl$ is a homeomorphism.
\end{enumerate}
\end{example}

If $\Phi_{\cl}$ is the $C^{\ell,\lambda}$-diffeomorphism from Example \ref{ex:diff}\ref{it:diff1}, then the change of coordinates mapping $(\Phi_{\cl})_*f:= f\circ \Phi_{\cl}^{-1} $ is an isomorphism between 
$W^{k,p}(\OO, w_{\gam}^{\d\OO})$ and $W^{k,p}(\RRdh, w_{\gam})$ for $\ell \geq k$. To be able to deal with rougher domains, i.e., $\ell<k$, we will mollify the diffeomorphism to make it smooth in the interior. This causes blow-up behaviour of higher-order derivatives near the boundary. In the rest of this appendix, we present these mollified versions of $\Phi_{\cl}$ and $\Psi_{\cl}$, denoted $\Phi$ and $\Psi$ respectively.\\

First, we consider a mollified version of the diffeomorphism $\Phi_{\cl}$ in Example \ref{ex:diff}\ref{it:diff1}. 
This mollified diffeomorphism is in the literature also known as the \emph{Dahlberg--Kenig--Stein pullback}, which dates back to \cite{Da79,Da86} and is, for instance, applied in \cite{DKPV97, FJR78, KP01}. The Dahlberg--Kenig--Stein pullback is often used for domains with low regularity (less than $C^1$), see the above-mentioned literature. For our purposes, we require estimates on higher-order derivatives of the pullback in the case of more regular domains (more than $C^1$). The following lemma is an extension of the result for $C^1$-domains in \cite[Lemmas 2.6 \& 3.8]{KK04}, which is based on the works \cite{GH80, Li85}. The result of Lemma \ref{lem:localization_weighted_blow-up} is also obtained in \cite[Lemma 2.14]{KimD07} with a somewhat different proof as below. Nonetheless, we provide the proof since parts of it will be reused for constructing the diffeomorphism $\Psi$ in Lemma \ref{lem:loc_normal}.

\begin{lemma}\label{lem:localization_weighted_blow-up}
Let $\OO$ be a special $\Cc^{1}$-domain.
Then there exist continuous functions $h_1\colon\overline{\Dom} \to \R$ and $h_2 \colon \overline{\R^d_+} \to \R$ which satisfy the following properties.
\begin{enumerate}[(i)]
    \item\label{it:lem:localization_weighted_blow-up;Psi_inverse} The map $\Phi:\overline{\Dom} \to \overline{\R^d_+}$
    given by
    $$
    \Phi(x)= (x_1-h_1(x),\tilde{x}), \qquad x=(x_1,\tilde{x}) \in \overline{\Dom},
    $$
    is a $C^{1}$-diffeomorphism with inverse $\Phi^{-1}:\overline{\R^d_+} \to \overline{\Dom}$ given by
    $$
    \Phi^{-1}(y) = (y_1+h_2(y),\tilde{y}),\qquad y=(y_1,\tilde{y}) \in \overline{\R^d_+}.
    $$
    \item\label{it:lem:localization_weighted_blow-up;dist_preserving} We have
    \begin{equation*}
    \begin{aligned}
    \mrm{dist}(\Phi(x),\partial\R^d_+) &\eqsim \mrm{dist}(x,\BDom), \qquad&& x \in \Dom,\\
    \mrm{dist}(\Phi^{-1}(y),\BDom)
    &\eqsim \mrm{dist}(y,\partial\R^d_+),\qquad &&y \in \R^d_+,
    \end{aligned}
    \end{equation*}
    where the implicit constants depend on $[\OO]_{C^1}$.
    \item\label{it:lem:localization_weighted_blow-up;smoothness} We have $h_1 \in C^\infty(\Dom)$ and $h_2 \in C^\infty(\R^d_+)$.
    \end{enumerate} 
 In addition, let $\ell\in\NN_1$, $\lambda\in[0,1]$ and let $\OO$ be a special $\Cc^{\ell,\lambda}$-domain with $[\OO]_{C^{\ell,\lambda}}\leq 1$.
    \begin{enumerate}[resume*]
    \item\label{it:lem:localization_weighted_blow-up;est} The map $\Phi$ in \ref{it:lem:localization_weighted_blow-up;Psi_inverse} is a $C^{\ell,\lambda}$-diffeomorphism and for all $\alpha\in \N_0^d$, $\ell_0\in \{0, \dots, \ell\}$ and $\lambda_0\in[0,\lambda]$, we have
    \begin{equation*}
    \begin{aligned}
    |\d^{\alpha} h_1(x)| &\leq  C\cdot[\OO]_{C^{\ell,\lambda}}\cdot\mrm{dist}(x,\BDom)^{-(|\alpha|-\ell_0-\lambda_0)_+},\qquad  &&x \in \Dom,\\
    |\d^{\alpha} h_2(y)| &\leq C\cdot[\OO]_{C^{\ell,\lambda}}\cdot\mrm{dist}(y,\partial\R^d_+)^{-(|\alpha|-\ell_0-\lambda_0)_+},\qquad &&y \in \R^d_+,
    \end{aligned}
    \end{equation*}
    where the constant $C>0$ only depends on $\ell, \lambda, \alpha$ and $d$.
\end{enumerate}
\end{lemma}

We have a similar result for the mollified version of the admissible diffeomorphism $\Psi_{\cl}$ in Example \ref{ex:diff}\ref{it:diff2}.
\begin{lemma}\label{lem:loc_normal}
Let $\OO$ be a special $\Cc^{1}$-domain. Then there exists a $\Lambda\in (0,1)$ such that if $[\OO]_{C^{1}}\leq \Lambda$, then there exist continuous functions $h_1:\overline{\OO}\to \RR$, $\tilde{h}_1:\overline{\OO}\to \RR^{d-1}$, $h_2:\overline{\RRdh}\to \RR$ and $\tilde{h}_2:\overline{\RRdh}\to \RR^{d-1}$ which satisfy the following properties.
\begin{enumerate}[(i)]
    \item\label{it:lem:loc_normal1} The map $\Psi: \overline{\OO}\to \overline{\RRdh} $ given by
    \begin{equation*}
        \Psi(x)=(x_1-h_1(x), \tilde{x}-\tilde{h}_1(x)),\qquad x=(x_1, \tilde{x})\in \overline{\OO},
    \end{equation*}
    is a $C^{1}$-diffeomorphism with inverse $\Psi^{-1}:\overline{\RRdh} \to \overline{\OO}$ given by
    \begin{equation*}
        \Psi^{-1}(y)=(y_1+h_2(y), \tilde{y}+ \tilde{h}_2(y)),\qquad y=(y_1,\tilde{y})\in \overline{\RRdh}.
    \end{equation*}
    \item\label{it:lem:loc_normal2} We have
    \begin{equation*}
    \begin{aligned}
    \mrm{dist}(\Psi(x),\partial\R^d_+) &\eqsim \mrm{dist}(x,\BDom), \qquad&& x \in \Dom,\\
    \mrm{dist}(\Psi^{-1}(y),\BDom)
    &\eqsim \mrm{dist}(y,\partial\R^d_+),\qquad &&y \in \R^d_+.
    \end{aligned}
    \end{equation*}
    \item\label{it:lem:loc_normal3} We have $h_1\in C^\infty(\OO)$, $\tilde{h}_1\in C^\infty(\OO; \RR^{d-1})$, $h_2\in C^\infty(\RRdh)$ and $\tilde{h}_2\in C^\infty(\RRdh; \RR^{d-1})$.
\item\label{it:lem:loc_normal4} The diffeomorphism $\Psi$ is admissible for all $x\in \d\OO$.
\end{enumerate}
In addition, let $\ell\in\NN_1$, $\lambda\in [0,1]$ and let $\OO$ be a special $\Cc^{\ell,\lambda}$-domain with $[\OO]_{C^{\ell,\lambda}}\leq 1$.
\begin{enumerate}[resume*]
\item\label{it:lem:loc_normal5} The map $\Psi$ in \ref{it:lem:loc_normal1} is a $C^{\ell,\lambda}$-diffeomorphism and for all $\alpha\in \NN^d_0$, $\ell_0\in \{0,\dots, \ell\}$ and $\lambda_0\in [0,\lambda]$, we have
\begin{align*}
     |\d^{\alpha} h_1(x)|+ |\d^{\alpha} \tilde{h}_1(x)|  &\leq  C\cdot[\OO]_{C^{\ell,\lambda}}\cdot\mrm{dist}(x,\BDom)^{-(|\alpha|-\ell_0-\lambda_0)_+}, &&x \in \Dom,\\
    |\d^{\alpha} h_2(y)|+ |\d^{\alpha} \tilde{h}_2(y)| &\leq C\cdot[\OO]_{C^{\ell,\lambda}}\cdot\mrm{dist}(y,\partial\R^d_+)^{-(|\alpha|-\ell_0-\lambda_0)_+}, &&y \in \R^d_+,
\end{align*}
where the constant $C>0$ only depends on $\ell, \lambda, \alpha$ and $d$.
\end{enumerate}
\end{lemma}
\begin{remark}
  We make the following remarks about Lemmas \ref{lem:localization_weighted_blow-up} and \ref{lem:loc_normal}.
  \begin{enumerate}[(i)]
    \item The first three statements in both lemmas are standard results for localisation. For the standard localisation procedure, one can take $h_1$ and $h_2$ equal to $h$, see Example \ref{ex:diff}. In our case, since $h$ is not smooth enough, we need to use a mollifier to make $h_2$ smooth in Lemma \ref{lem:localization_weighted_blow-up}. Afterwards, $h_1$ is determined using the inverse function theorem. Additionally, in Lemma \ref{lem:loc_normal} we need to mollify the normal vector to make $\tilde{h}_2$ smooth.
    \item The important part of Lemma \ref{lem:localization_weighted_blow-up} is \ref{it:lem:localization_weighted_blow-up;est}, which allows us to estimate higher-order derivatives of the diffeomorphism $\Phi$ and its inverse. If the number of derivatives exceeds the smoothness of the domain, then there is a blow-up near the boundary. We note that the construction of $\Phi$ is independent of $\ell$ and $\lambda$.
    \item The important novelty of Lemma \ref{lem:loc_normal} is that $\Psi$ preserves the direction of the normal and tangential vectors. This is done by choosing $\tilde{h}_2(y)= y_1 \tilde{\mc{V}}(y)$, where $\tilde{\mc{V}}(y)$ is a mollified version of the normal direction $-\grad_{\tilde{y}} h(\tilde{y})$. While the diffeomorphism in Example \ref{ex:diff}\ref{it:diff2} is less regular than the domain, in Lemma \ref{lem:loc_normal} the regularity of the diffeomorphism is equal to the regularity of the domain. Moreover, the estimates on the derivatives in Lemma \ref{lem:loc_normal}\ref{it:lem:loc_normal5} are the same as in Lemma \ref{lem:localization_weighted_blow-up}\ref{it:lem:localization_weighted_blow-up;est}.
    \item The condition $[\OO]_{C^{\ell,\lambda}}\leq 1$ in Lemmas \ref{lem:localization_weighted_blow-up}\ref{it:lem:localization_weighted_blow-up;est} and \ref{lem:loc_normal}\ref{it:lem:loc_normal5} slightly simplifies the proofs. This condition is not necessary and can be removed. The more restrictive condition on $[\OO]_{C^{1}}$ in Lemma \ref{lem:loc_normal} seems to be necessary to construct a global inverse. 
    For our application in Section \ref{sec:calc_dom}, imposing such conditions is not restrictive since $[\OO]_{C^{\ell,\lambda}}$ can be made arbitrarily small in our localisation procedure. 
    \item  Our proofs would also work for Lipschitz domains if one uses an inverse function theorem for Lipschitz functions. In the setting of Lemma \ref{lem:localization_weighted_blow-up}, this is done in \cite{KimD07}. The result in Lemma \ref{lem:loc_normal} appears to be new even in the case of $C^1$-domains.
  \end{enumerate}
\end{remark}
The rest of this appendix is devoted to the proofs of Lemmas \ref{lem:localization_weighted_blow-up} and \ref{lem:loc_normal}. 

\begin{proof}[Proof of Lemma \ref{lem:localization_weighted_blow-up}] 
    Let $\ph \in \Cc^\infty(\R^{d-1})$ be a non-negative function with its support in the unit ball such that $\int_{\R^{d-1}} \ph(\tilde{x})\dd \tilde{x} =1$. Let $h\in \Cc^{1}(\RR^{d-1})$ such that $\OO= \{x\in\RR^d:x_1>h(\tilde{x})\}$, see Definition \ref{def:domains}.

\textit{Step 1: proof of \ref{it:lem:localization_weighted_blow-up;Psi_inverse}, \ref{it:lem:localization_weighted_blow-up;dist_preserving} and \ref{it:lem:localization_weighted_blow-up;smoothness}.} Define the mapping $\overline{\Phi}:\RRd\to \RRd$ by
\begin{equation*}
    \begin{aligned}
        \overline{\Phi}(y)&:= (y_1+ h_2(y), \tilde{y}),\quad \text{ where }\\
        h_2(y)&:=\int_{\RR^{d-1}}h(\tilde{y}- L^{-1}y_1 \tilde{z})\ph(\tilde{z})\dd \tilde{z},
    \end{aligned}
\end{equation*}
for some suitable $L>0$ to be chosen later. Note that $\overline{\Phi}$ maps $\d\RRdh$ to $\d\OO$. It holds that $h_2\in C^{1}(\RRd)\cap C^{\infty}(\RRd\setminus\d\RRdh)$. The Jacobian matrix of $\overline{\Phi}$ is given by
\begin{equation}\label{eq:DPhi}
    D\overline{\Phi}(y) = \begin{pmatrix}
        1+ \d_{y_1}h_2(y)& (\grad_{\tilde{y}}h_2(y))^\top\\
        0 & I_{d-1}
    \end{pmatrix},\qquad y\in \RR^d.
\end{equation}
Since
\begin{equation*}
    \d_{y_1}h_2(y) = -L^{-1}\int_{\RR^{d-1}}(\grad h)(\tilde{y}-L^{-1}y_1\tilde{z})\cdot \tilde{z}\ph(\tilde{z})\dd \tilde{z},\qquad y\in \RR^d,
\end{equation*}
it follows that $|\d_{y_1}h_2(y)|\leq \half$ for all $y\in \RR^d$ if $L = 2[\OO]_{C^1}$. Hence, for all $y\in \RRd$, we have $1+ \d_{y_1}h_2(y)\geq \half$ and hence
\begin{equation}\label{eq:det}
    |\det D\overline{\Phi}(y) |\geq \tfrac{1}{2}>0,
\end{equation} 
meaning that $D\overline{\Phi}$ is invertible.  

We construct an inverse of $\overline{\Phi}$ with Hadamard's inverse function theorem \cite[Theorem 6.2.8]{KP02}. Note that if $(y_n)_{n\geq 1}$ is a sequence in $\R^d$ such that $|y_n|\to \infty$, then also $|\overline{\Phi}(y_n)|\to \infty$  as $n\to \infty$. Indeed, this follows from the fact that $h_2$ is bounded. Hence, $\overline{\Phi}$ is proper in the sense of \cite[Definition 6.2.2]{KP02}. This and \eqref{eq:det} imply that all the conditions of \cite[Theorem 6.2.8]{KP02} are satisfied and thus that there exists a continuous inverse $\Phi:\RR^d\to \RR^d$ of $\overline{\Phi}$. Since actually $\overline{\Phi}\in C^1(\RR^d; \RR^d)\cap C^{\infty}(\RR^d\setminus\d\RRdh;\RR^d)$, it also follows that $\Phi\in C^1(\RRd;\RR^d)\cap C^\infty(\RRd\setminus\d\OO;\RR^d)$ by the inverse function theorem \cite[Theorem 3.3.2]{KP02}. Moreover, we obtain
\begin{equation}\label{eq:formulaDPhi}
    D\Phi(x) \cdot D\overline{\Phi}(\Phi(x)) = I_d,\qquad x\in \RRd,
\end{equation}
see, e.g., \cite[Equation (3.37)]{KP02}. From now on, we will write $\Phi^{-1}$ for $\overline{\Phi}$. Moreover, if $x\in \OO$ and $y=\Phi(x)\in \RRdh$, then by the definition of $\Phi^{-1}=\overline{\Phi}$ it holds that
\begin{equation}\label{eq:defh1}
    \Phi_1(x)= y_1 = x_1 - h_2(y)= x_1 - h_2(\Phi(x)),
\end{equation}
and hence we can write $\Phi(x)=(x_1-h_1(x), \tilde{x})$, where $h_1(x)= h_2(\Phi(x))$. This completes the proof of \ref{it:lem:localization_weighted_blow-up;Psi_inverse} and \ref{it:lem:localization_weighted_blow-up;smoothness}. 

To continue, we prove that the diffeomorphism preserves the distance to the boundary. Since both $\Phi$ and $\Phi^{-1}$ are Lipschitz on $\RR^d$, we have
\begin{equation}\label{eq:biLip}
    |y-y'|\lesssim |\Phi^{-1}(y)-\Phi^{-1}(y')|\lesssim |y-y'|,\qquad y,y'\in \RR^d.
\end{equation}
Let $x\in \OO$ and set $y=\Phi(x)\in \RRdh$. Then \eqref{eq:biLip} gives
\begin{equation*}
    \dist(x, \d\OO)\leq |x-(h(\tilde{x}), \tilde{x})| = |\Phi^{-1}(y)-\Phi^{-1}(0, \tilde{y})|\lesssim y_1=\Phi_1(x)=\dist(\Phi(x), \d\RRdh).
\end{equation*}
In addition, let $x'\in \d\OO$ be such that $\dist(x, \d\OO)=|x-x'|$ and set $y'=\Phi(x')$. Note that $y'= (0, \tilde{y}\,')$ for some $\tilde{y}\,'\in \RR^{d-1}$, hence \eqref{eq:biLip} also implies
\begin{equation*}
    \dist(\Phi(x), \d\RRdh) = \Phi_1(x)=y_1\leq |y-y'|\lesssim |\Phi^{-1}(y)-\Phi^{-1}(y')| = |x-x'|= \dist(x, \d\OO).
\end{equation*}
Thus, we have proved that $\dist(x, \d\OO)\eqsim \dist(\Phi(x), \d\RRdh)$, where the implicit constants depend on the Lipschitz constants of $\Phi$ and $\Phi^{-1}$ and thus on  $[\OO]_{C^1}$. The other equivalence in \ref{it:lem:localization_weighted_blow-up;dist_preserving} follows by substituting $x$ by $\Phi^{-1}(y)$. This completes the proof of \ref{it:lem:localization_weighted_blow-up;dist_preserving}.

\textit{Step 2: proof of estimates on $h_2$ in \ref{it:lem:localization_weighted_blow-up;est}.} Let $\ell\in\NN_1$, $\lambda\in[0,1]$ and let $\OO$ be a special $\Cc^{\ell,\lambda}$-domain with $[\OO]_{C^{\ell,\lambda}}\leq 1$. For multi-indices we write $\alpha=(\alpha_1,\tilde{\alpha})\in\NN_0\times \NN_0^{d-1}$. By the chain rule it holds that
\begin{equation*}
  \d_{y_1}h(\tilde{y}-L^{-1}y_1\tilde{z}) = (\grad h)(\tilde{y}-L^{-1}y_1\tilde{z}) \cdot \frac{-\tilde{z}}{L} = -L^{-1}\sum_{|\nu|=1}(\d^\nu h)(\tilde{y}-L^{-1}y_1\tilde{z}) \tilde{z}^\nu,
\end{equation*}
and by iteration one can check for any $\alpha \in \NN_0^{d}$ that 
\begin{align}
  \d^{\alpha}h_2(y_1,\tilde{y})&=\frac{1}{(-L)^{\alpha_1}}\int_{\RR^{d-1}}\sum_{|\nu|=\alpha_1} (\d^{\nu+\tilde{\alpha}}h)(\tilde{y}-L^{-1}y_1\tilde{z})\tilde{z}^{\nu}\ph(\tilde{z})\dd \tilde{z}.\label{eq:F1}
\end{align}
Take $\ell_0 \in \{0,\ldots,\ell\}$. If $|\alpha|\leq \ell_0$, it follows from \eqref{eq:F1} that
\begin{equation*}
  |\d^{\alpha}h_2(y_1,\tilde{y})|\leq C\|h\|_{C^\ell(\RR^{d-1})}\sum_{|\nu|=\alpha_1}\int_{\RR^{d-1}} | \tilde{z}^\nu\ph(\tilde{z})|\dd \tilde{z}\leq C [\OO]_{C^{\ell,\lambda}},
\end{equation*}
which proves the estimates for $h_2$ in \ref{it:lem:localization_weighted_blow-up;est}. Now let $|\alpha|\geq \ell_0+1$ and let $\beta, \overline{\beta}\in \NN_0^d$ be such that $\beta+\overline{\beta}=\alpha$ with $|\beta|=\ell_0$ and $|\overline{\beta}|=|\alpha|-\ell_0$. 
From \eqref{eq:F1} and a substitution $\tilde{z}\mapsto((\tilde{y}-\tilde{z})L)/y_1$ it follows that
\begin{align}
  \d^{\beta}h_2(y_1,\tilde{y}) &=\frac{1}{(-L)^{\beta_1}}\Big(\frac{y_1}{L}\Big)^{1-d}\int_{\RR^{d-1}}\sum_{|\nu|=\beta_1} (\d^{\nu+\tilde{\beta}}h)(\tilde{z})\Big(\frac{(\tilde{y}-\tilde{z})L}{y_1}\Big)^{\nu} \ph\Big(\frac{(\tilde{y}-\tilde{z})L}{y_1}\Big)\dd \tilde{z}\label{eq:F2}.
\end{align}
By computing the $\overline{\beta}$-derivatives using \eqref{eq:F2}, we claim that 
\begin{equation}\label{eq:claimh_2est}
  \begin{aligned}
  \d^{\alpha}h_2(y_1,\tilde{y}) & =\d^{\overline{\beta}}\d^{\beta}h_2(y_1,\tilde{y}) \\
   & = C \frac{1}{y_1^{|\alpha|-\ell_0}} \Big(\frac{y_1}{L}\Big)^{1-d}\int_{\RR^{d-1}}\sum_{|\nu|=\beta_1} (\d^{\nu+\tilde{\beta}}h)(\tilde{z}) \ph_{\beta, \overline{\beta}, \nu}\Big(\frac{(\tilde{y}-\tilde{z})L}{y_1}\Big)\dd \tilde{z},
\end{aligned}
\end{equation}
where $\ph_{\beta, \overline{\beta}, \nu}\in \Cc^\infty(\RR^{d-1})$ and $\int \ph_{\beta, \overline{\beta}, \nu}(\tilde{z})\dd \tilde{z}=0$.  Indeed, if $\overline{\beta}=e_j$ is the $j$-th unit vector for some $j\in \{2,\dots, d\}$, then by writing $\tilde{y}=(y_2,\dots, y_d)$ and $\tilde{z}=(z_2,\dots, z_d)$, a calculation shows that
\begin{align*}
  \d_{y_j}\Big[\Big(\frac{(\tilde{y}-\tilde{z})L}{y_1}\Big)^{\nu} \ph\Big(\frac{(\tilde{y}-\tilde{z})L}{y_1}\Big)\Big]   =&\;\frac{L}{y_1}\Big[ \nu_j \Big(\frac{(y_j-z_j)L}{y_1}\Big)^{\nu_j-1} \prod_{\substack{n=2\\n\neq j}}^d \Big(\frac{(y_n-z_n)L}{y_1}\Big)^{\nu_n}\ph\Big(\frac{(\tilde{y}-\tilde{z})L}{y_1}\Big)\\
  &+ \Big(\frac{(\tilde{y}-\tilde{z})L}{y_1}\Big)^\nu (\d_j\ph )\Big(\frac{(\tilde{y}-\tilde{z})L}{y_1}\Big)
  \Big]\\
  =:&\; y_1^{-1}\ph_{\beta, e_j, \nu}\Big(\frac{(\tilde{y}-\tilde{z})L}{y_1}\Big).
\end{align*}
Moreover, note that
\begin{equation}\label{eq:ph=0}
  \begin{aligned}
  \int_{\RR^{d-1}}\ph_{\beta, e_j, \nu}(\tilde{z})\dd \tilde{z}&= \Big(\frac{y_1}{L}\Big)^{1-d} y_1 \d_{y_j}\int_{\RR^{d-1}}\Big(\frac{(\tilde{y}-\tilde{z})L}{y_1}\Big)^{\nu} \ph\Big(\frac{(\tilde{y}-\tilde{z})L}{y_1}\Big) \dd\tilde{z}\\
  &=y_1\d_{y_j}\int_{\RR^{d-1}}\tilde{z}^\nu\ph(\tilde{z})\dd \tilde{z}=0,
\end{aligned}
\end{equation}
and clearly we have $\ph_{\beta, e_j, \nu}\in \Cc^\infty(\RR^{d-1})$. This shows \eqref{eq:claimh_2est} for $\overline{\beta}=e_j$ with $j\in \{2,\dots, d\}$. If $\overline{\beta}=e_1$, then a calculation shows that
\begin{align*}
 \d_{y_1}\Big[\Big(\frac{y_1}{L}\Big)^{1-d}&\Big(\frac{(\tilde{y}-\tilde{z})L}{y_1}\Big)^{\nu} \ph\Big(\frac{(\tilde{y}-\tilde{z})L}{y_1}\Big)\Big]\\=&\; \frac{1}{y_1}\Big(\frac{y_1}{L}\Big)^{1-d} \Big(\frac{(\tilde{y}-\tilde{z})L}{y_1}\Big)^{\nu}  \Big[(1-d-\beta_1) \ph\Big(\frac{(\tilde{y}-\tilde{z})L}{y_1}\Big)\\
 &\;-(\grad \ph)\Big(\frac{(\tilde{y}-\tilde{z})L}{y_1}\Big)\cdot \Big(\frac{(\tilde{y}-\tilde{z})L}{y_1}\Big) \Big]\\
 =:&\;y_1^{-1} \Big(\frac{y_1}{L}\Big)^{1-d}\ph_{\beta, e_1, \nu}\Big(\frac{(\tilde{y}-\tilde{z})L}{y_1}\Big).
\end{align*}
The properties of $\ph_{\beta, e_1, \nu}$ follow similarly as in \eqref{eq:ph=0}. Therefore, we have proved \eqref{eq:claimh_2est} for $|\overline{\beta}|=1$. For $|\overline{\beta}|\geq 2$ we can argue by induction to show that 
\begin{equation*}
  \d_y^{\overline{\beta}}\Big[\Big(\frac{y_1}{L}\Big)^{1-d}\Big(\frac{(\tilde{y}-\tilde{z})L}{y_1}\Big)^{\nu} \ph\Big(\frac{(\tilde{y}-\tilde{z})L}{y_1}\Big)\Big] = C\frac{1}{y_1^{|\alpha|-\ell_0}} \Big(\frac{y_1}{L}\Big)^{1-d}\ph_{\beta, \overline{\beta}, \nu}\Big(\frac{(\tilde{y}-\tilde{z})L}{y_1}\Big).
\end{equation*}
This follows in the same manner as for $|\overline{\beta}|=1$ by considering the $\d_{y_1}$ and $\d_{y_j}$ separately. Therefore, \eqref{eq:claimh_2est} follows.

Let $\lambda_0\in[0,\lambda]$. Performing the substitution $\tilde{z}\mapsto \tilde{y}-L^{-1}y_1\tilde{z}$ in \eqref{eq:claimh_2est} and using that $\ph_{\beta, \overline{\beta}, \nu}$ integrates to zero, gives
\begin{equation*}\label{eq:est_h2}
  \begin{aligned}
  |\d^{\alpha}h_2(y)| & \leq C \, y_1^{-(|\alpha|-\ell_0)}\int_{\RR^{d-1}}\sum_{|\nu|=\beta_1}\big|(\d^{\nu+\tilde{\beta}}h)(\tilde{y}-L^{-1}y_1\tilde{z})-(\d^{\nu+ \tilde{\beta}}h)(\tilde{y})\big|\, |\ph_{\beta, \overline{\beta}, \nu}(\tilde{z})|\dd \tilde{z}\\
  &\leq  C\, 
  \|h\|_{C^{\ell, \lambda_0}(\RR^{d-1})}\, y_1^{-(|\alpha|-\ell_0)}\int_{\RR^{d-1}}|L^{-1}y_1 \tilde{z}|^{\lambda_0}|\ph_{\beta, \overline{\beta}, \nu}(\tilde{z})|\dd \tilde{z}\\
  &\leq  C \, [\OO]_{C^{\ell, \lambda}}\,y_1^{-(|\alpha|-\ell_0-\lambda_0)}.
\end{aligned}
\end{equation*}
This completes the proof of all the estimates for $h_2$ in \ref{it:lem:localization_weighted_blow-up;est}.

\textit{Step 3: proof of estimates on $h_1$ in \ref{it:lem:localization_weighted_blow-up;est}.} It remains to prove the estimates on $\d^\alpha h_1$ for $\alpha\in \NN_0^d$. As $h_2$ is bounded by Step 2, we find that $h_1(x)=h_2(\Phi(x))$ (see \eqref{eq:defh1}) is bounded as well, which proves the required estimate for $\alpha=0$. 
For $|\alpha|=1$ we first show that $\|D\Phi\|$ is bounded. Note that by \eqref{eq:formulaDPhi} we have 
\begin{equation*}
    \|D\Phi(x)\| = \|[(D\Phi^{-1})(\Phi(x))]^{-1}\| = \frac{\|\operatorname{adj}[ (D \Phi^{-1})(\Phi(x))]\|}{|\det [(D\Phi^{-1})(\Phi(x))]|},\qquad x\in\OO,
\end{equation*}
where $\operatorname{adj}A$ denotes the adjugate matrix of the matrix $A$, i.e., the transpose of the cofactor matrix. Recall that the adjugate matrix of $A$ only consists of polynomials of entries of $A$. Hence, combining this with \eqref{eq:det} and the fact that the entries of $D\Phi^{-1}$ are bounded by Step 2, we find that
\begin{equation}\label{eq:alpha=1}
    \|D \Phi(x)\| \leq C,\qquad x\in \OO.
\end{equation}
Now the estimate on $\d^\alpha h_1$ with $|\alpha|=1$ follows again from $h_1(x)=h_2(\Phi(x))$ together with the estimates on $h_2$ from Step 2, $\dist(\Phi(x), \d\RRdh)\eqsim \dist(x, \d\OO)$ and \eqref{eq:alpha=1}.

For the general case, we proceed by induction on $|\alpha|\geq 1$. We prove that 
\begin{equation}\label{eq:normal_IH}
    |\dist(x, \d\OO)|^{(|\alpha|-\ell_0-\lambda_0)_+}\|D^{|\alpha|} h_1(x)\|\leq C\, [\OO]_{C^{\ell,\lambda}},\qquad x\in\OO, 
\end{equation}
for all $\ell_0\in\{0,\dots, \ell\}$ and $\lambda_0\in[0,\lambda]$. 
Recall that for $f:\RR^d\supseteq U\to \RR^d$,  $x\in U$ and $n\in \NN_1$, the $n$-th order Fréchet derivative $D^n f(x)$ is an $n$-linear mapping from $(\RR^d)^n$ to $\RRd$, i.e., for $(\xi_1, \dots, \xi_n)\in (\RRd)^n$ we have $D^n f(x) (\xi_1, \dots, \xi_n)\in \RRd$.

The statement \eqref{eq:normal_IH} for $|\alpha|=1$ is proved above. Let $m\geq 1$ and assume that \eqref{eq:normal_IH} holds for all  $|\alpha|\leq m$, $\ell_0\in\{0,\dots, \ell\}$ and $\lambda_0\in[0,\lambda]$. It remains to prove \eqref{eq:normal_IH} for $|\alpha|=m+1$. 

By taking derivatives of the formula \eqref{eq:formulaDPhi}, isolating the highest-order derivatives on $\Phi$ and applying the multivariate Faà di Bruno's formula, we obtain the estimate (see \cite[Lemma 4]{CR72})
\begin{align}\label{eq:est_IH}
    \|D^{m+1} \Phi(x)\|\lesssim \|D\Phi(x)\|\cdot \sum_{j=2}^{m+1}\|(D^{j}\Phi^{-1})(\Phi(x))\|\cdot\sum_{\beta}\prod_{r=1}^{m}\|D^r \Phi(x)\|^{\beta_r},\qquad x\in\OO,
\end{align}
where $\beta\in \NN_0^{m}$ are such  that
\begin{equation}\label{eq:est_IH_cond}
    \sum_{r=1}^{m}\beta_r = j\quad \text{ and }\quad \sum_{r=1}^{m} r \beta_r =m+1.
\end{equation}
Note that $\|D\Phi(x)\|$ is uniformly bounded by \eqref{eq:alpha=1} and that for $n\geq 2$ we have $\|D^{n} h_1\| = \|D^{n} \Phi\|$ and $\|D^{n} h_2\| = \|D^{n}\Phi^{-1}\|$. Therefore, multiplying \eqref{eq:est_IH} with $|\dist(x, \d\OO)|^{(m+1-\ell_0-\lambda_0)_+}$, it suffices to show the uniform boundedness of 
\begin{equation}\label{eq:normal_IHestimate}
    |\dist(x, \d\OO)|^{(m+1-\ell_0-\lambda_0)_+}\|(D^j h_2)(\Phi(x))\|\prod_{r=2}^m\|D^r h_1(x)\|^{\beta_r},
\end{equation}
for $j\in\{2,\dots, m+1\}$ and $\beta$ such that \eqref{eq:est_IH_cond} holds. Suppose that there exist $\kappa, \kappa_2,\dots, \kappa_m\in(0,\infty)$ such that 
\begin{equation}\label{eq:defkappasnormal}
    \begin{aligned}
    (j-\ell- \lambda)_+&\leq \kappa \leq j\\
    (r-\ell-\lambda)_+&\leq \kappa_r\leq r, \qquad r\in\{2,\dots, m\} 
\end{aligned}
\end{equation}
and
\begin{equation}\label{eq:conditionkappasnormal}
    \kappa+ \sum_{r=2}^m\beta_r \kappa_r = (m+1-\ell_0-\lambda_0)_+.
\end{equation}
Then, \eqref{eq:normal_IHestimate} can be estimated as
\begin{equation}\label{eq:IIHresult}
    |\dist(x,\d\OO)|^{\kappa}\|(D^j h_2)(\Phi(x))\|\prod_{r=2}^m\big(|\dist(x,\d\OO)|^{\kappa_r}\|D^r h_1(x)\|\big)^{\beta_r}\leq C\, [\OO]_{C^{\ell,\lambda}},
\end{equation}
where we have applied the induction hypothesis \eqref{eq:normal_IH} and the estimates for $h_2$ from Step 2 together with $\dist(\Phi(x), \d\RRdh)\eqsim \dist(x, \d\OO)$.
It remains to show the existence of $\kappa$'s satisfying \eqref{eq:defkappasnormal} and \eqref{eq:conditionkappasnormal}. We distinguish several cases.

If $m+1\leq \ell_0$, then $(m+1-\ell_0-\lambda_0)_+=0$ and we can take $\kappa=\kappa_2= \dots= \kappa_m=0$. From now on, we assume that $m\geq \ell_0$. If $j\geq \ell_0+1$, then we can take
\begin{equation*}
    \kappa= j-\ell_0-\lambda_0 \quad \text{ and }\quad \kappa_r=r-1  \text{ for }r\in\{2, \dots, m\},
\end{equation*}
and \eqref{eq:est_IH_cond} implies that \eqref{eq:conditionkappasnormal} is satisfied. For the remaining case $j\leq \ell_0$, we will not provide the explicit values of the $\kappa$'s, but only show the existence of the $\kappa$'s. Taking the largest possible $\kappa$'s in \eqref{eq:defkappasnormal}, gives 
\begin{equation*}
    \kappa + \sum_{r=2}^m \beta_r \kappa_r =j +\sum_{r=2}^m r\beta_r =j+m+1-\beta_1 \geq m+1-\ell_0-\lambda_0,
\end{equation*}
where we have used \eqref{eq:est_IH_cond} and $\beta_1\leq j$. We will now take the smallest possible $\kappa$'s in \eqref{eq:defkappasnormal}. First assume that  $m\geq \ell+1$ and that there exists an $\tilde{r}\in\{\ell+1, \dots, m\}$ such that $\beta_{\tilde{r}}\geq 1$. In this case, we have
\begin{align*}
    \kappa + \sum_{r=2}^m \beta_r \kappa_r &=(j-\ell-\lambda)_+ +\sum_{r=2}^m (r-\ell-\lambda)_+\beta_r
    =\sum_{r=\ell+1}^m (r-\ell-\lambda)\beta_r\\
    &\leq \sum_{r=1}^m r\beta_r -(\ell+\lambda)\sum_{r=\ell+1}^{m} \beta_r=m+1  -(\ell+\lambda)\sum_{r=\ell+1}^{m} \beta_r\\
    &\leq m+1-\ell-\lambda\leq m+1-\ell_0-\lambda_0, 
\end{align*}
where we have used $(j-\ell-\lambda)_+\leq (j-\ell_0-\lambda_0)_+=0$ and \eqref{eq:est_IH_cond}. 
If $m\leq \ell$ or $\beta_r=0$ for all $r\in \{\ell+1, \dots, m\}$, then 
$\kappa + \sum_{r=2}^m \beta_r \kappa_r =0 \leq m+1-\ell_0-\lambda_0.$
The existence of the $\kappa$'s shows that \eqref{eq:IIHresult} holds. This finishes the induction.

Finally, we remark that the estimates on $h_1$ and $h_2$ imply that if $h\in \Cc^{\ell,\lambda}(\RR^{d-1})$, then $\Phi\in C^{\ell,\lambda}(\overline{\OO}; \RR^d)$ and $\Phi^{-1}\in C^{\ell,\lambda}(\overline{\RRdh};\R^d)$. This proves that $\Phi$ is a $C^{\ell,\lambda}$-diffeomorphism and this finishes the proof of \ref{it:lem:localization_weighted_blow-up;est}.
\end{proof}

To prepare for the proof of Lemma \ref{lem:loc_normal}, we prove the following elementary lemma.

\begin{lemma}\label{lem:linearalgebra}
    Let $A$ and $B$ be $d\times d$-matrices. 
    If  $\det(A)\neq 0$  and
    $
{\|A\|^{d-1}\|B\|}\cdot|\det(A)|^{-1}<1,
    $
   then we have
    \begin{equation*}
|\det(A+B)| \geq |\det(A)|\Bigl(1-\frac{\|A\|^{d-1}\|B\|}{|\det(A)|}\Bigr)^d.
    \end{equation*}
\end{lemma}
\begin{proof}
 Let $0<\sigma_1\leq\cdots\leq \sigma_d$ be the singular values of $A$. Then we have
$$
\|A^{-1}\| = \frac{1}{\sigma_1} = \frac{\prod_{j=2}^d \sigma_j}{|\det(A)|} \leq  \frac{\|A\|^{d-1}}{|\det(A)|}.
$$
In particular, we have $\|A^{-1}B\|\leq \|A^{-1}\|\|B\|\leq \|A\|^{d-1}\|B\|\cdot |\det(A)|^{-1}<1$. Now let $\lambda_1,\ldots,\lambda_d$ be the eigenvalues of $A^{-1}B$. Then $|\lambda_j|\leq\|A^{-1}B\|<1$ for all $j\in\{1, \dots, d\}$ and 
$$
|\det(I_d+A^{-1}B)| = \prod_{j=1}^d |1+\lambda_j| \geq \prod_{j=1}^d 1-|\lambda_j| \geq (1-\|A^{-1}B\|)^d,
$$
which, combined with the norm estimate for $\|A^{-1}\|$, yields
\begin{align*}
    |\det(A+B)| &= |\det(A)||\det(I_d+A^{-1}B)|\\& \geq |\det(A)| (1-\|A^{-1}B\|)^d\geq |\det(A)|\Bigl(1-\frac{\|A\|^{d-1}\|B\|}{|\det(A)|}\Bigr)^d,
\end{align*}
finishing the proof.
\end{proof}

To conclude this appendix, we give the proof of Lemma \ref{lem:loc_normal}.
\begin{proof}[Proof of Lemma \ref{lem:loc_normal}] 
 Let $\ph \in \Cc^\infty(\R^{d-1})$ be a non-negative and radially symmetric function with its support in the unit ball  such that $\int_{\R^{d-1}} \ph(\tilde{x})\dd \tilde{x} =1$. Define $h\in \Cc^{1}(\RR^{d-1})$ such that $\OO= \{x\in\RR^d:x_1>h(\tilde{x})\}$, see Definition \ref{def:domains}. Furthermore, let 
\begin{equation}\label{eq:normalvector}
    \nu(x)=\big(1, \tilde{\nu}(\tilde{x})\big)^\top= \big(1, -\grad_{\tilde{x}}h(\tilde{x})\big)^\top, \qquad x\in \RR^{d},
\end{equation}
be an inward pointing normal vector at $\d\OO$.

\textit{Step 1: proof of \ref{it:lem:loc_normal1}, \ref{it:lem:loc_normal2} and \ref{it:lem:loc_normal3}.} 
Define the mapping $\overline{\Psi}:\RR^d\to \RR^d$ given by
\begin{equation*}
    \begin{aligned}
    \overline{\Psi}(y) &:= (y_1+ h_2(y), \tilde{y}+\tilde{h}_2(y)),\quad \text{ where}\\
    h_2(y)&:=\int_{\RR^{d-1}}h(\tilde{y}- L^{-1}y_1\tilde{z})\ph(\tilde{z})\dd \tilde{z},\\
    \tilde{h}_2(y)&:= y_1 \tilde{\mc{V}}(y):= y_1\int_{\RR^{d-1}} \tilde{\nu}(\tilde{y}-L^{-1}y_1\tilde{z})\ph(\tilde{z})\dd \tilde{z},
\end{aligned}
\end{equation*}
for some suitable $L>0$ to be chosen later. Note that $\overline{\Psi}$ maps $\d\RRdh$ to $\d\OO$. It holds that $h_2\in C^{1}(\RRd)\cap C^{\infty}(\RRd\setminus\d\RRdh)$. Furthermore, we claim that 
$\tilde{h}_2\in C^{1}(\RRd; \RR^{d-1})\cap C^{\infty}(\RRd\setminus\d\RRdh; \RR^{d-1})$. Indeed, it is clear that $\tilde{h}_2\in C(\RRd; \RR^{d-1})\cap C^{\infty}(\RRd\setminus\d\RRdh; \RR^{d-1})$ since $\tilde{\nu} = -\grad_{\tilde{x}}h\in C(\RR^{d-1};\RR^{d-1})$. Note that
\begin{equation*}
    \d_1 \tilde{h}_2(0,\tilde{a})= \lim_{y_1\to 0}\frac{y_1\tilde{\mc{V}}((y_1, \tilde{a}))}{y_1}= \tilde{\mc{V}}(0,\tilde{a})=\tilde{\nu}(\tilde{a}),\qquad \tilde{a}\in \RR^{d-1},
\end{equation*}
and all the tangential partial derivatives of $\tilde{h}_2$ at $\d\RRdh$ are zero. If $y\in \RR^d\setminus\d\RRdh$ and $|\alpha|=1$, then by a similar computation as \eqref{eq:claimh_2est} (with $|\overline{\beta}|=1$, $\beta=0$, $\ell_0=0$ and $h$ replaced by $\tilde{\nu}$) and a substitution $\tilde{z}\mapsto \tilde{y}-L^{-1}y_1\tilde{z}$, we obtain
\begin{align*}
    \d^{\alpha}\tilde{\mc{V}}(y)&= y_1^{-1}\Big(\frac{y_1}{L}\Big)^{1-d}\int_{\RR^{d-1}}\tilde{\nu}(\tilde{z})\ph_{0,\alpha,0}\Big(\frac{(\tilde{y}-\tilde{z})L}{y_1}\Big)\dd \tilde{z}\\
    &=y_1^{-1}\int_{\RR^{d-1}}\tilde{\nu}(\tilde{y}-L^{-1}y_1\tilde{z})\ph_{0,\alpha,0}(\tilde{z})\dd\tilde{z},
\end{align*}
where $\ph_{0,\alpha,0}\in \Cc^\infty(\RR^{d-1})$ and $\int \ph_{0,\alpha,0}(\tilde{z})\dd\tilde{z}=0$. Let $y=(y_1,\tilde{y})\in \RR^{d}\setminus\d\RRdh$, then we obtain
\begin{equation}\label{eq:partderh2}
    \begin{aligned}
    \d_{y_1}\tilde{h}_2(y) &= \tilde{\mc{V}}(y)+ y_1 \d_{y_1}\tilde{\mc{V}}(y)= \int_{\RR^{d-1}}\tilde{\nu}(\tilde{y}-L^{-1}y_1\tilde{z})\big[\ph(\tilde{z})+ \ph_{0,e_1,0}(\tilde{z})\big]\dd\tilde{z},\\
    \d_{y_j}\tilde{h}_2(y) &= y_1\d_{y_j}\tilde{\mc{V}}(y)=\int_{\RR^{d-1}}\tilde{\nu}(\tilde{y}-L^{-1}y_1\tilde{z})\ph_{0,e_j,0}(\tilde{z})\dd\tilde{z},\qquad j\in\{2,\dots, d\}.
\end{aligned}
\end{equation}
By applying the dominated convergence theorem and the properties of $\ph$ and $\ph_{0,e_j,0}$ for $j\in \{1,\dots, d\}$, we find that for all $\tilde{a}\in \RR^{d-1}$
\begin{equation*}
    \lim_{y\to (0, \tilde{a})}\d_{y_1}\tilde{h}_2(y)= \tilde{\nu}(\tilde{a})\quad \text{ and }\quad  \lim_{y\to (0, \tilde{a})}\d_{y_j}\tilde{h}_2(y) =0,\quad j\in \{2,\dots, d\}.
\end{equation*}
Hence, all the partial derivatives of $\tilde{h}_2$ at $\d\RRdh$ exist and are continuous. Therefore, $\tilde{h}_2\in C^1(\RR^d; \RR^{d-1})$ and we have proved the claim.

The Jacobian matrix of $\overline{\Psi}$ is given by
\begin{equation}\label{eq:DPsi}
    \begin{aligned}
    D \overline{\Psi}(y)&=\begin{pmatrix}
        1+ \d_{y_1}h_2(y)& (\grad_{\tilde{y}}h_2(y))^\top\\
        \d_{y_1}\big(y_1\tilde{\mc{V}}(y)\big)& I_{d-1}+ y_1D_{\tilde{y}}\tilde{\mc{V}}(y)
    \end{pmatrix}\\
    &= D\overline{\Phi}(y)+ \begin{pmatrix}
        0& 0\\
        \d_{y_1}\big(y_1\tilde{\mc{V}}(y)\big)& y_1 D_{\tilde{y}}\tilde{\mc{V}}(y)
    \end{pmatrix}=: D\overline{\Phi}(y)+ P(y),\quad y\in \RRd,
\end{aligned}
\end{equation}
where $\overline{\Phi}$ is as defined in the proof of Lemma \ref{lem:localization_weighted_blow-up} (see \eqref{eq:DPhi}) and $P$ is a perturbation. 
To show that the mapping $\overline{\Psi}$ is invertible, we make use of Lemma \ref{lem:linearalgebra} applied to $A= D\overline{\Phi}$ and $B=P$. Note that \eqref{eq:partderh2} implies
 $\sup_{y\in \RRd}\|P(y)\|\leq C[\OO]_{C^1}$. Furthermore, we recall from Lemma \ref{lem:localization_weighted_blow-up}\ref{it:lem:localization_weighted_blow-up;est} and the definition of $\overline{\Phi}$ that $\sup_{y\in \RRd}\| D\overline{\Phi}(y)\|\leq C$. From Lemma \ref{lem:linearalgebra}, \eqref{eq:det} and the before mentioned estimates, we obtain that  there exists a $\Lambda\in (0,1)$ such that if $[\OO]_{C^1}\leq \Lambda$, then
\begin{align*}
   \inf_{y\in \RRd} |\det D\overline{\Psi}(y)|&\geq \inf_{y\in\RRd} |\det D\overline{\Phi}(y)|\Big(1-\frac{\|D\overline{\Phi}(y)\|^{d-1}\|P(y)\|}{|\det D\overline{\Phi}(y)|}\Big)^d\\
   &\geq c(1- C[\OO]_{C^1})^d\geq c(1- C\Lambda)^d>0,
\end{align*}
for some $C,c>0$.

Using  Hadamard's inverse function theorem, we can argue similarly as in the proof of Lemma \ref{lem:localization_weighted_blow-up} to obtain an inverse $\Psi\in C^1(\RR^d;\RR^d)\cap C^\infty(\RR^d\setminus\d\OO;\RR^d)$ to $\overline{\Psi}$. We will write $\Psi^{-1}$ for $\overline{\Psi}$. 

Moreover, if $x\in \OO$ and $y=\Psi(x)\in \RRdh$, then by the definition of $\Psi^{-1}=\overline{\Psi}$ it holds that
\begin{equation}\label{eq:defh1tildeh1}
\begin{aligned}
     \Psi_1(x)&= y_1 = x_1 - h_2(y)= x_1 - h_2(\Psi(x)),\\ \tilde{\Psi}(x)& = \tilde{y} = \tilde{x}-\tilde{h}_2(y) = \tilde{x}-\tilde{h}_2(\Psi(x)),
\end{aligned}
\end{equation}
and hence we can write $\Psi(x)=(x_1-h_1(x), \tilde{x}-\tilde{h}_1(x))$ with $h_1(x)=h_2(\Psi(x))$ and $\tilde{h}_1(x)= \tilde{h}_2(\Psi(x))$.
This completes the proof of \ref{it:lem:loc_normal1} and \ref{it:lem:loc_normal3}. Statement \ref{it:lem:loc_normal2} follows similarly as in the proof of Lemma \ref{lem:localization_weighted_blow-up}\ref{it:lem:localization_weighted_blow-up;dist_preserving}.

\textit{Step 2: proof of \ref{it:lem:loc_normal4}.}
 We prove that the diffeomorphism $\Psi$ is admissible for any $x\in \d\OO$. First, note that
\begin{equation*}
    (\d_{y_1}h_2(y))|_{y_1=0} = -L^{-1} (\grad h)(\tilde{y})\cdot\int_{\RR^{d-1}} \tilde{z}\ph(\tilde{z})\dd \tilde{z}=0,\quad y\in \RRd,\label{eq:Jacobi_left_upper_y1=0}
\end{equation*}
where we have used the radial symmetry of $\ph$ to obtain that the latter integral vanishes, i.e., $\int_{\RR} z_i\ph(\tilde{z})\dd z_i=0$ for all $i\in\{2,\dots, d\}$. Together with \eqref{eq:normalvector}, we see that the push-forward of $\overline{\Psi}=\Psi^{-1}$ (recall \eqref{eq:DPsi}) at the boundary $\d\RRdh$ is given by 
\begin{equation*}
    D \overline{\Psi}(0, \tilde{y})=\begin{pmatrix}
        1 & (\grad_{\tilde{y}}h(\tilde{y}))^{\top}\\
        \tilde{\mc{V}}(0, \tilde{y}) & I_{d-1}
    \end{pmatrix}
    =\begin{pmatrix}
        1 & (\grad_{\tilde{y}}h(\tilde{y}))^{\top}\\
        -\grad_{\tilde{y}}h(\tilde{y}) & I_{d-1}
    \end{pmatrix},\qquad \tilde{y}\in \RR^{d-1}.
\end{equation*}
It is clear that for all $\tilde{y}\in \RR^{d-1}$, the push-forward $D \overline{\Psi}(0, \tilde{y})$ maps the normal vector $(1,0)$ at $\d\RRdh$ to the normal vector $\nu(x)$ at $\d\OO$. Similarly, any tangent vector at $\d\RRdh$ is mapped to the tangent space at $\d\OO$. Moreover, note that the push-forward does not preserve the length of the vectors. This proves that $\overline{\Psi}=\Psi^{-1}$ is admissible for any $y\in \d\RRdh$, hence $\Psi$ is admissible for any $x\in \d\OO$ as well.

\textit{Step 3: proof of \ref{it:lem:loc_normal5}.} Let $\ell\in \NN_1$, $\lambda\in [0,1]$ and let $\OO$ be a special $\Cc^{\ell,\lambda}$-domain with $[\OO]_{C^{\ell, \lambda}}\leq \Lambda$.
Note that the estimates on the derivatives of $h_2$ follow immediately from Lemma \ref{lem:localization_weighted_blow-up}\ref{it:lem:localization_weighted_blow-up;est}. Let $\alpha=(\alpha_1,\tilde{\alpha})\in \NN_0\times \NN_0^{d-1}$, $\ell_0\in \{0,\dots, \ell\}$ and $\lambda_0\in [0,\lambda]$. For the estimate on $\d^{\alpha}\tilde{h}_2$, we find with the product rule 
\begin{equation}\label{eq:esttildeh2}
\begin{aligned}
      |\d^\alpha\tilde{h}_2(y)|&\lesssim |y_1\d^{\alpha}\tilde{\mc{V}}(y)| +|\d^{(\alpha_1-1,\tilde{\alpha})}\tilde{\mc{V}}(y)|,\quad y\in \RRdh,
\end{aligned}
\end{equation}
where the latter term is only present if $\alpha_1\geq 1$.
Since $\tilde{\mc{V}}$ is obtained by mollifying $\tilde{\nu}=-\grad h$, we can redo Step 2 in the proof of Lemma \ref{lem:localization_weighted_blow-up} but with $h$ replaced by $-\grad h$, to obtain 
\begin{equation*}
    |\d^\beta \tilde{\mc{V}}(y)|\leq C\, [\OO]_{C^{\ell,\lambda}}\, y_1^{-(|\beta|+1-\ell_0-\lambda_0)_+},\qquad y\in\RRdh,\, \beta\in\NN_0^d. 
\end{equation*}
Applying this estimate to \eqref{eq:esttildeh2} gives the desired estimates for $\tilde{h}_2$.

It remains to prove the estimates on $h_1$ and $\tilde{h}_1$, which are defined as below \eqref{eq:defh1tildeh1}. Hence, as $h_2$ and $\tilde{h}_2$ are bounded, we find that $h_1$ and $\tilde{h}_1$ are bounded as well. To prove the estimates on the derivatives $\d^\alpha h_1$ and $\d^\alpha \tilde{h}_1$ for $|\alpha|\geq 1$, one can proceed by induction similar to Step 3 in the proof of Lemma \ref{lem:localization_weighted_blow-up}.
\end{proof}

\bibliographystyle{plain}
\bibliography{references_LLRV25}
\end{document}